\documentclass[12pt,reqno]{amsart}
\usepackage{amsmath}
\usepackage{a4wide}
\usepackage{xfrac}
\usepackage{stmaryrd,mathrsfs,bm,amsthm,mathtools,yfonts,amssymb,color}
\usepackage{xcolor}

\begingroup
\newtheorem{theorem}{Theorem}[section]
\newtheorem{lemma}[theorem]{Lemma}
\newtheorem{proposition}[theorem]{Proposition}
\newtheorem{corollary}[theorem]{Corollary}

\endgroup

\theoremstyle{definition}
\newtheorem{example}[theorem]{Example}
\newtheorem{definition}[theorem]{Definition}
\newtheorem{remark}[theorem]{Remark}
\newtheorem{ipotesi}[theorem]{Assumption}

\newcommand\sP{{\mathscr P}}

\newcommand\supp{{\rm spt}}
\newcommand{\sing}{{\rm Sing}}
\newcommand{\reg}{{\rm Reg}}
\newcommand\rD{{\rm D}}
\newcommand\res{\mathop{\hbox{\vrule height 7pt width .3pt depth 0pt
\vrule height .3pt width 5pt depth 0pt}}\nolimits}
\newcommand{\im}{{\rm Im}}
\newcommand{\gr}{{\rm Gr}}
\newcommand{\bT}{\mathbf{T}}
\newcommand{\bG}{\mathbf{G}}
\newcommand{\bJ}{{\mathbf{J}}}
\newcommand{\p}{{\mathbf{p}}}

\newcommand{\sW}{{\mathscr{W}}}
\newcommand{\sC}{{\mathscr{C}}}
\newcommand{\bGam}{{\bm \Gamma}}
\newcommand\bmo{{\bm m}_0}

\newcommand{\cM}{{\mathcal{M}}}
\newcommand{\bU}{{\mathbf{U}}}
\newcommand{\bL}{{\mathbf{L}}}
\newcommand{\phii}{{\bm{\varphi}}}
\newcommand{\Phii}{{\bm{\Phi}}}

\newcommand{\cU}{{\mathcal{U}}}
\newcommand{\cV}{{\mathcal{V}}}

\newcommand{\cL}{{\mathcal{L}}}
\newcommand{\cK}{{\mathcal{K}}}
\newcommand{\bD}{{\mathbf{D}}}
\newcommand{\bH}{{\mathbf{H}}}
\newcommand{\bI}{{\mathbf{I}}}
\newcommand{\bE}{{\mathbf{E}}}

\newcommand{\bB}{{\mathbf{B}}}
\newcommand{\bC}{{\mathbf{C}}}
\newcommand{\be}{{\mathbf{e}}}
\newcommand{\bd}{{\mathbf{d}}}

\newcommand\Z{{\mathbb Z}}

\newcommand\N{{\mathbb N}}

\newcommand\C{{\mathbb C}}
\newcommand\R{{\mathbb R}}
\newcommand{\eps}{{\varepsilon}}
\newcommand{\bA}{\mathbf{A}}
\newcommand{\bmax}{\mathbf{m}}
\def\Xint#1{\mathchoice
{\XXint\displaystyle\textstyle{#1}}%
{\XXint\textstyle\scriptstyle{#1}}%
{\XXint\scriptstyle\scriptscriptstyle{#1}}%
{\XXint\scriptscriptstyle\scriptscriptstyle{#1}}%
\!\int}
\def\XXint#1#2#3{{\setbox0=\hbox{$#1{#2#3}{\int}$ }
\vcenter{\hbox{$#2#3$ }}\kern-.6\wd0}}
\def\mint{\Xint-}
\newcommand{\cone}{{\times\hspace{-0.6em}\times\,}}
\newcommand{\Lip}{{\rm {Lip}}}

\newcommand{\dist}{{\rm {dist}}}

\newcommand\weaks{{\stackrel{*}{\rightharpoonup}}\,}

\newcommand{\cB}{{\mathcal{B}}}
\newcommand{\cG}{{\mathcal{G}}}

\newcommand{\cF}{{\mathcal{F}}}
\newcommand{\cH}{{\mathcal{H}}}

\newcommand{\cQ}{{\mathcal{Q}}}
\newcommand{\cR}{{\mathcal{R}}}

\newcommand\sS{{\mathscr S}}

\newcommand\Pe{{\mathscr P}}

\newcommand{\mass}{{\mathbf{M}}}

\newcommand\e{\mathbf{e}}
\newcommand{\Om}{\Omega}
\def\I#1{{\mathcal{A}}_{#1}}
\def\Is#1{{\mathcal{A}}_{#1} (\R^{n})}
\newcommand{\Iqs}{{\mathcal{A}}_Q(\R^{n})}
\newcommand{\Iq}{{\mathcal{A}}_Q}
\def\a#1{\left\llbracket{#1}\right\rrbracket}
\newcommand{\abs}[1]{\left|#1\right|}
\newcommand{\norm}[2]{\left\|#1\right\|_{#2}}

\newcommand{\D}{\textup{Dir}}
\newcommand{\de}{\partial}
\newcommand{\xii}{{\bm{\xi}}}
\newcommand{\ro}{{\bm{\rho}}}

\newcommand{\etaa}{{\bm{\eta}}}

\newcommand\B{{\mathbf{B}}}
\newcommand{\bh}{\mathbf{h}}

\title[The singular set of area-minimizing currents]{The size of the singular set of area-minimizing currents}
\author{Camillo De Lellis}
\address{Mathematik Institut der Universit\"at Z\"urich}
\email{delellis@math.uzh.ch}

\begin{document}

\maketitle

\section{Introduction}

The Plateau's problem investigates those surfaces of least area spanning a given contour. It is one of the most classical
problems in the calculus of variations, it lies at the crossroad
of several branches of mathematics and it has generated a large amount of mathematical theory in the last one hundred years. 
The problem itself and its various generalizations have found fundamental applications in several mathematical and scientific 
branches. Since it is a prototype of a vast family of questions with geometric and physical significance,
the techniques developed to analyze it have proved to be very useful in a variety of other situations.

Although its original formulation is restricted to $2$-dimensional surfaces spanning a given curve $\gamma$ in the $3$-dimensional space, in this note we will consider the case of (suitably generalized) $m$-dimensional surfaces in oriented Riemannian manifolds $\Sigma$. 
We will restrict ourselves to compact (sufficiently smooth) $\Sigma$ and since all the considerations will be of a local nature we will often assume that $\Sigma$ itself is isometrically embedded in some euclidean space (of dimension $m+n$). In this way the competitor surfaces (classical or generalized) spanning the contour $\gamma$ will always be (suitable generalizations of) subsets of the standard euclidean space, constrained to be subsets of $\Sigma$. Although this is not very elegant from a geometric point of view, it allows us to avoid a lot of technicalities.

\medskip

The very formulation of the Plateau's problem has proved to be a quite challenging mathematical question. In particular, how general are the surfaces that one should consider? What is the correct concept of ``spanning'' and the correct concept of ``$m$-dimensional volume'' that one should use? The author believes that there are no final answers to these two questions: many different significant ones have been given in the history of our subject and, depending upon the context, the features of one formulation might be considered more important than those of the others.  

One popular way to generalize the concept of surface is to consider it in duality with differential forms: on a smooth $m$-dimensional oriented submanifold of $\Sigma$ we can integrate compactly supported forms and this action gives a natural relation between smooth submanifolds and linear maps on the space of smooth (compactly supported) forms $\mathcal{D}^m (\Sigma)$. Following a pioneering idea of De Rham (cf. \cite{DeRham}) we introduce the $m$-dimensional {\em currents}, which are linear functionals on $\mathcal{D}^m (\Sigma)$ satisfying a suitable continuity property. The action of a current $T$ on a form $\omega$ is then given by $T (\omega)$ and we can introduce naturally a concept of boundary ``enforcing'' Stokes' theorem: $\partial T (\omega) := T (d\omega)$. In fact currents can be thought as an extension of the classical notion of chain, namely of formal linear combinations of submanifolds, and the space of currents with the operator $\partial$ forms a chain complex: this in turn leads to an homology theory which can be shown to be equivalent to the several other (classical) homology theories. There is also a natural generalization of the concept of $m$-dimensional volume and using standard functional analytic methods it is then possible to show the existence of minimizers for the Plateau's problem. As a nice byproduct of this theory we can represent each class in (real) homology with a cycle of least mass.

\medskip

Note however that De Rham's theory allows real multiplicities. This is not natural from the point of view of Plateau's problem (in fact generalized minimizers might have non-integer multiplicities, cf. Example \ref{e:meridiani} below). We could try to remedy to this shortcome by considering only the subspace of currents which can be represented as classical chains with integer multiplicities and then look at its closure in the appropriate topology. This point of view was taken in the celebrated paper \cite{FF} by Federer and Fleming and the corresponding objects, called {\em integral currents}, provide an ideal framework for the general {\em oriented} Plateau's problem. A very nice feature of the Federer and Fleming theory is that each integral homology class can be represented by a cycle of least area. In fact the theory is powerful enough to deal with more general coefficient groups.

\medskip

There are two first fundamental issues that a satisfactory variational theory needs to address: existence and regularity. Of course these are by no means the {\em only} important aspects of Plateau's problem: however almost all the other necessarily build on these two important pieces of information, namely that
\begin{itemize}
\item[(a)] there is a minimizer for a large class of boundaries;
\item[(b)] the minimizer is sufficiently regular, so that one can compute interesting geometric quantities and infer additional conclusions.
\end{itemize}
The success of the Federer and Fleming's theory is due to the vast applicability of its existence part in all dimensions and codimensions. Thanks to the efforts of several outstanding mathematicians a rather far-reaching (and satisfactory) regularity theory was achieved in the seventies in codimension $1$ (see for instance \cite{Giusti}). This theory has been digested by the subsequent generations of scholars working in differential geometry and PDEs, leading ultimately to many breakthroughs in different problems in geometry, PDEs and mathematical physics. Indeed the codimension $1$ case is considerably easier than the higher codimension: the reason is that integral $m$-dimensional cycles in $\R^{m+1}$ are in fact (countable integral combinations of) boundaries of sets. The Federer-Fleming theory is thus equivalent to the theory of sets of finite perimeter (or Caccioppoli sets), which was indeed introduced a few years earlier by De Giorgi, developing a pioneering idea of Caccioppoli (cf. \cite{Caccioppoli,DeGiorgi1,DeGiorgi3}; for an english translation of De Giorgi's works we refer to \cite{DeGiorgi2}). 

\medskip

In the higher codimension case the most important conclusion of the regularity theory can be attributed to the monumental work of a single person, F. J. Almgren Jr. (\cite{Alm}). Unlike the codimension one case, only a relatively small portion of the theory of Almgren has been truly understood. In a recent series of papers Emanuele Spadaro and the author have given a new, much shorter, account of Almgren's regularity program, relying on the several advances in geometric measure theory  of the last two decades and on some new ideas. The aim of this note is to give a rather detailed picture of the several issues that this program must face and of how they are resolved.

\medskip

We will assume that the reader is familiar with the basic concepts of the theory of integer rectifiable currents: we will therefore only recall them very quickly to lay a common notational ground (the reader who is not familiar with this part of geometric measure theory is instead encouraged to read first the survey \cite{CDMsurvey}).
 As part of our exposition we will isolate the features of codimension $1$. In Section \ref{s:first} we will recall the first considerations in the regularity theory and summarize the state of the art in the subject.  
We will then review the regularity theory for minimizers in codimension $1$, focusing on its most important steps, cf. Section \ref{s:cod-1}. The reader who is very familiar with these topics can skip the corresponding sections. However some of the discussions in there
will play an important role in the rest of the note, where we will describe the details of the proofs of Almgren's theorem as presented in the papers \cite{DS1,DS2,DS3,DS4,DS5}.

\section{Terminology and fundamental results in the theory of currents}

\subsection{General currents}
The idea of treating (oriented) surfaces as linear functionals on a suitable space of smooth objects dates back at least to De Rham, cf. \cite{DeRham}. In what follows we denote by $\mathcal{D}^m (\Omega)$ the space of smooth compactly supported $m$-forms in $\Omega$, where $\Omega$ is a differentiable manifold. We will omit $\Omega$ when it is clear from the context. 

\begin{definition}[Current, De Rham, cf. {\cite[Definition 26.1]{Simon}}]\label{d:current}
An $m$-dimensional current $T$ is a continuous linear map $T: \mathcal{D}^m (\Omega) \to \mathbb R$. Here the continuity condition must be understood in the following sense: $T (\omega^k) \to T(\omega)$ whenever $\{\omega^k\}\subset \mathcal{D}^m$ is a sequence such that
\begin{itemize}
\item[(a)] there is an open set $\Omega'\subset\subset \Omega$ with $\supp (\omega^k)\subset\Omega'$ for every $k$; 
\item[(b)] $\omega^k\to \omega$ in $C^j (\Omega')$ for every $j$.
\end{itemize}
A sequence of currents $T^k$ (of the same dimension $m$) converges to $T$ if $T^k (\omega)\to T (\omega)$ for every $\omega\in \mathcal{D}^m$ (cf. \cite[eq. 26.12]{Simon}).
\end{definition}

As already mentioned the concept of boundary is defined ``enforcing'' Stokes' Theorem.

\begin{definition}[Boundary, De Rham, cf. {\cite[eq. 26.3]{Simon}}]\label{d:bordo}
We say that an $(m-1)$-dimensional current $S$ is the boundary of an $m$-dimensional current $T$ if
\begin{equation}
T (d\omega) = S (\omega) \qquad \mbox{for every $\omega\in \mathcal{D}^{m-1}$.}
\end{equation}
$S$ will then be denoted by $\partial T$.
\end{definition}

The class of smooth oriented submanifolds $\Gamma\subset \Omega$ with smooth boundaries can then naturally be viewed as a subset of the space of $m$-dimensional currents in $\Omega$. In order to distinguish between any such $\Gamma$ and its ``action'' as linear functional via integration of forms, we will use the notation $\a{\Gamma}$ for the current, namely $\a{\Gamma} (\omega):= \int_\Gamma \omega$ (cf. \cite[eq. 26.2]{Simon}).  If we consider points $P\in \Omega$ as $0$-dimensional submanifolds, consistently with our convention we will denote by $\a{P}$ the usual Dirac delta at $P$, although a more common symbol would be $\delta_P$.

Although the definition above is quite general, in the rest of this note we will always assume that $\Omega$ is an open subset of the euclidean space $\R^{m+n}$. If $\Sigma$ is then a submanifold of $\R^{m+n}$, we will consider currents $T$ in $\Sigma$ as currents $T$ in $\R^{m+n}$ whose support $\supp (T)$ is contained in $\Sigma$. As usual the support $\supp (T)$ is the complement of the maximal open set $U$ for which $T (\omega) =0$ whenever $\supp (\omega)\subset U$ (cf. \cite[Section 4.1]{Fed}). 

\medskip

Recall that a simple $m$-vector is an element of $\Lambda_m (\R^{m+n})$ of the form $v_1\wedge \ldots \wedge v_m$.  Moreover there is a natural definition of {\em length} of a simple $m$-vector: $|v_1\wedge \ldots \wedge v_m|$ is the $m$-dimensional Hausdorff measure of the parallelogram spanned by the vectors $v_1, \ldots, v_m$. $|\cdot|$ can be easily extended to a Hilbert norm on $\Lambda_m (\R^{m+n})$ and the corresponding scalar product takes the following form when computed on simple vectors:
\[
\langle v_1\wedge \ldots \wedge v_m, w_1\wedge \ldots \wedge w_m \rangle = {\rm det}\, (v_i\cdot w_j)\, .
\]

\begin{definition}[Comass and mass, cf. {\cite[Section 1.8]{Fed}}]\label{d:comass}
Let $\omega\in \Lambda^m (\R^{m+n})$. Then the comass of $\omega$ is the norm 
\[
\|\omega\| := \max \left\{\langle \omega, v_1\wedge \ldots \wedge v_m\rangle: |v_1 \wedge \ldots \wedge v_m|=1\right\}\, ,
\]
(here $\langle , \rangle$ is the standard duality pairing between covectors and vectors). The dual norm on $\Lambda_m (\R^{m+n})$, called the mass norm, is
\[
\|\vec{T}\| = \max \left\{\langle \omega, \vec{T}\rangle : \|\omega\|\leq 1 \right\}\, .
\]
If $\omega\in \mathcal{D}^m (\R^{m+n})$ we then define its comass norm to be
\[
\|\omega\|_c := \max_p \|\omega (p)\|\, .
\]
\end{definition}

\begin{definition}[Mass of a current, cf. {\cite[Section 4.1.7]{Fed}}]\label{d:mass}
Given an $m$-dimensional current $T$ we denote by $\mass (T)$ its mass, namely the quantity
\begin{equation}\label{e:mass}
\mass (T) = \sup \left\{T (\omega) :  \|\omega\|_c\leq 1\right\}\;\, .
\end{equation}
For any open set $\Omega$ we define
\begin{equation}\label{e:tot_var}
\|T\| (\Omega) = \sup \left\{T (\omega) : \mbox{$\supp (\omega)\subset \Omega$ and $\|\omega\|_c\leq 1$}\right\}\,  .
\end{equation}
\end{definition}

When $T$ is a current of finite mass, namely $\mass (T) = \|T\| (\R^{m+n}) < \infty$, $\|T\|$ turns out to be a Radon measure
(cf. \cite[Sections 4.1.5 and 4.1.7]{Fed}). The same holds under the weaker assumption that $\|T\| (\Omega) < \infty$ for any bounded open set $\Omega$: in this case we will say that $T$ has locally finite mass. If we endow $\Lambda_m (\R^{m+n})$ with the mass norm, $T$ can then be seen as a measure taking values in a (finite-dimensional) Banach space and $\|T\|$ is the total variation measure of $T$. It is customary to take the ``polar (or Radon-Nikod\'ym) decomposition'' (cf. \cite[eq. 26.7 and Definition 27.1]{Simon}) of $T$ and infer the existence of a Borel-valued map $\vec{T}$ such that $T = \vec{T} \|T\|$, namely
\begin{equation}\label{e:RN}
T (\omega) = \int \langle \omega (p) , \vec{T} (p) \rangle\, d\|T\| (p)\qquad \forall \omega\in \mathcal{D}^m (\R^{m+n})\, .
\end{equation}
The comass of $\vec{T} (x)$ equals $1$ at every point $x$. In particular at those points where $\vec{T} (x)$ is simple, it can be represented as $e_1 \wedge \ldots \wedge e_m$ for a set $e_1, \ldots, e_m$ of vectors which span a parallelogramm of volume $1$.

\medskip

If $F$ is a smooth {\em proper} map between two Euclidean spaces $\R^N$ and $\R^k$, the pullback $F^\sharp \omega$ of an element $\omega\in \mathcal{D}^m (\R^k)$ is an element of $\mathcal{D}^m (\R^N)$. By duality this gives naturally a notion of pushforward of currents, namely $F_\sharp T (\omega) = T (F^\sharp \omega)$ (cf. \cite[eq. 26.20]{Simon}). The assumption that $F$ is proper is needed to guarantee that $F^\sharp \omega$ has compact 
support if $\omega$ has compact support. This is however not needed when $\supp (T)$ is compact: since this will be mostly the case, in several occasions we will push currents forward via maps which are just smooth.

Moreover, when $\Omega\subset \Omega'$ are open and $T$ is a current on $\Omega'$ we can easily restrict its action to forms supported in $\Omega$: for such restriction we use the notation $T\res \Omega$. When $T$ has finite mass, we can define $T\res E$ for any Borel set $E$ and consider it as a current in the original domain of definition by simply setting
\[
T\res E (\omega) := \int_E \langle \omega (p) , \vec{T} (p) \rangle\, d\|T\| (p)\, .
\]
In the future, to make our discussion simple, we will say that two currents $T$ and $S$ agree on an open subset $\Omega$ if $T\res \Omega = S \res \Omega$. If the currents have (locally) finite mass, this notion can be extended to Borel subsets.
We use a similar notation on Borel measures $\mu$ (which we always assume to be defined on some subset $K$ of the Euclidean space): given a $\mu$-measurable $E\subset K$, we use the symbol $\mu \res E$ for the measure $\mu\res E (A) := \mu (A\cap E)$.

\medskip

The product of an $m$-dimensional current $T$ in $\R^k$ and an $\bar{m}$-dimensional current $S$ in $\R^{\bar k}$ is defined on forms of type $\omega\wedge \nu$ with $\omega \in \mathcal{D}^m (\R^k)$ and $\nu\in \mathcal{D}^{\bar m} (\R^{\bar k})$ by
\[
T\times S (\omega \wedge \nu) := T (\omega) S (\nu)\, .
\]
The action is then extended by linearity to general forms in $\mathcal{D}^{m+\bar m} (\R^{k+\bar k})$.

If $S$ and $T$ are two currents in $\R^k$ we define their {\em join} $S\cone T$ as 
\[
H_\sharp (\a{[0,1]}\times S \times T)
\]
where $H: \R\times \R^k\times R^k \ni (t,x,y) \mapsto y + (1-t) x$. Of particular interest for us is the case where $S$ is $\a{P}$ for some point $P$: the join $\a{P}\cone T$ is then the {\em cone with vertex $P$ and basis $T$} and to simplify the notation we will use $P\cone T$.

\subsection{The Federer-Fleming theory} As already mentioned, general currents allow to take real linear combinations of classical surfaces. Following Federer-Fleming it is instead possible to build an homology theory with integer coefficients.

\begin{definition}[Integral currents, Federer-Fleming, cf. {\cite[Definition 27.1]{Fed}}]\label{d:integral}
A current $T$ of locally finite mass is {\em integer rectifiable} if there is a sequence of $C^1$ oriented surfaces $\Sigma_i\subset \R^{m+n}$, a sequence of pairwise disjoint closed subsets $K_i\subset \Sigma_i$ and a sequence of positive integers $k_i$ such that
\begin{equation}\label{e:def_rect}
T (\omega) = \sum_i k_i \int_{K_i} \omega \qquad \forall \omega\in \mathcal{D}^m\, .
\end{equation}
$T$ is {\em integral} if both $T$ and $\partial T$ are integer rectifiable. 
\end{definition}

In the definition above it is implicitely assumed that the right hand side of \eqref{e:def_rect} is a convergent series. In fact it is not difficult to see that, under the above assumptions, $\|T\| (\Omega) = \sum_i k_i \cH^m (K_i\cap \Omega) = \sum_i k_i {\rm Vol}^m\, (K_i\cap \Omega)$, where $\cH^\alpha$ denotes the $\alpha$-dimensional Hausdorff measure.

The space of integer rectifiable currents is not any more a {\em linear space} and there is no simple functional-analytic principle which provides a good compactness property. A fundamental result in the theory of Federer-Fleming is that, nonetheless, the space of integral currents is compact in a suitable sense.

\begin{theorem}[Compactness of integral currents, Federer-Fleming, cf. {\cite[Theorem 32.2]{Simon}}]\label{t:FFcomp}
If $\{T^k\}$ is a sequence of integral $m$-dimensional currents in $\Omega \subset \R^{m+n}$ open such that 
\[
\sup_k (\mass (T^k) + \mass (\partial T^k)) < \infty\, ,
\] 
then there is a subsequence, not relabeled, and an integral $m$-dimensional current $T$ such that $T_k \to T$.
\end{theorem}

As a corollary we achieve 

\begin{corollary}[Existence of area minimizing integral currents]\label{c:existence}
Let $Z$ be an $(m-1)$-dimensional integer rectifiable current and $\bar T$ an $m$-dimensional integral current with $\partial \bar T = Z$ and $\mass (\bar T) < \infty$. Then there is an integer rectifiable current $T_0$ such that 
$\partial T_0 = Z$ and 
\[
\mass (T_0) = \min \{\mass (T) : \mbox{ $T$ is int. rect. and }\partial T = Z\}\, .
\]

If $\supp (\bar T)\subset K$ for some closed set $K$ we also have the existence of an integer rectifiable current $T_0$ such that $\partial T_0 = Z$, $\supp (T_0) \subset K$ and 
\[
\mass (T_0) = \min \{\mass (T) : \mbox{$T$ is int. rect., $\partial T = Z$ and  $\supp (T)\subset K$}\}\, .
\]
\end{corollary}

Indeed both Theorem \ref{t:FFcomp} and Corollary \ref{c:existence} can be extended to a fairly large class of metric spaces, cf. \cite{AK}. 

Of course a similar corollary can be drawn via classical functional analysis for general currents: if $Z$ is a general $m-1$-dimensional current, $K$ a closed set and $T$ an $m$-dimensional current of finite mass with $\partial T = Z$ and $\supp (T)\subset K$, then there is an $m$-dimensional current $T_0$ with $\partial T_0 = Z$, $\supp (T_0) \subset K$ and
\[
\mass (T_0) = \min \{\mass (T) : \mbox{$\partial T = Z$ and  $\supp (T)\subset K$}\}\, .
\]
However the latter formulation allows for general real coefficients.

\begin{example}\label{e:meridiani}
Consider the south and north poles $S$ and $N$ in the standard sphere $\mathbb S^2 \subset \mathbb R^3$ and let $Z$ be the $0$-dimensional current $\a{N}-\a{S}$. For any meridian $\gamma$ joining $S$ to $N$ the corresponding current $\a{\gamma}$ is a minimizer of the mass among all currents $T$ with $\partial T = Z$ and $\supp (T)\subset \mathbb S^2$. However the same holds for any convex combination $\lambda \a{\gamma} + (1-\lambda) \a{\eta}$ where $\eta$ is any other meridian and $\lambda \in [0,1]$. In fact one can push this idea even further. Let us parametrize the meridians as $\{\gamma_t\}_{t\in \mathbb S^1}$, where $t$ is the intersection of $\gamma_t$ with the equator $\{x_3=0\}\cap \mathbb S^2$. If $\mu$ is a probability measure on $\mathbb S^1$, then the current
\[
T_0 (\omega) := \int_{\mathbb S^1} \a{\gamma_t} (\omega)\, d\mu (t)
\]
is also a minimizer of the mass (among those currents $T$ with $\supp (T)\subset \mathbb S^2$ and $\partial T = Z$). 
\end{example}

This does not seem a serious issue as there are anyway ``classical minimizers'' in the example above. However, we have the following remarkable theorem (for a very short and elegant proof we refer to \cite{White1}).

\begin{theorem}[Lavrentiev gap, Young \cite{Young}]\label{l:lav}
For every smooth closed embedded curve $\gamma$ in $\mathbb R^4$ define
\begin{align}
&M (\gamma) := \inf \left\{ \mass (T): T \mbox{ is integer rectifiable and } \partial T = \a{\gamma}\right\}\\
&m (\gamma) := \min \left\{ \mass (T): \partial T = \a{\gamma}\right\}\, .
\end{align}
Then there are $\gamma$'s for which $M (\gamma) > m (\gamma)$.
\end{theorem}

Theorem \ref{t:FFcomp} does not exhaust the major results of the foundational paper of Federer and Fleming. Indeed we wish to mention three other important cornerstones. First of all, the rectifiability of the boundary can be recovered from that of the current under the only assumption that the boundary has finite mass.

\begin{theorem}[Boundary rectifiability, Federer-Fleming, cf. {\cite[Theorem 30.3]{Simon}}]\label{t:FFrect}
If $T$ is integer rectifiable and $\partial T$ has locally finite mass, then $T$ is integral.
\end{theorem}

Secondly, any integral current can be suitably approximated by a sequence of ``polyhedral chains'' with integer coefficients. This is the content of the so-called Deformation Lemma (see \cite[Theorem 29.1 and Corollary 29.3]{Simon}), whose precise statement would require the introduction of some terminology and goes beyond the scopes of this note. The Deformation Lemma shows that the space of integeral currents is the closure of the space of classical chains with integer coefficients in a topology which is very natural for the Plateau's problem.

We record here a rather useful corollary of the Deformation Lemma.

\begin{theorem}[Isoperimetric inequality, Federer-Fleming, cf. {\cite[Theorem 30.1]{Simon}}]\label{t:FFiso}
There are constants $C (m,n)$ with the following property.
Assume $S$ is an integer rectifiable $m$-dimensional current in $\R^{m+n}$ with $\partial S = 0$. Then there is an integral current $T$ with $\partial T = S$ and $\mass (T) \leq C (\mass (S))^{(m+1)/m}$.
\end{theorem}

We finish this section by introducing a couple of other objects which will be very convenient in the rest of the note.

\begin{lemma}[Density]\label{l:density}
If $T$ is an integer rectifiable current, then the number
\begin{equation}\label{e:density}
\Theta (T, p) := \lim_{r\downarrow 0} \frac{\|T\| (\bB_r (p))}{\omega_m r^m}
\end{equation}
exists and it is a positive integer for $\|T\|$-a.e. $p$. 
\end{lemma}

As for the density we introduce the notion of tangent planes via a ``blow-up procedure''.  In order to give the corresponding statement we introduce two conventions:
\begin{itemize}
\item First of all, we will consider all $m$-dimensional planes $\pi$ as oriented. Thus, for each $\pi$ we have a unique integral current $\a{\pi}$.
\item Given a current $T$ we will denote by $T_{p,r}$ the result of translating it so that $p$ becomes the origin and enlarging it of a factor $r^{-1}$. Formally, if $\iota_{p,r}$ denotes the map $x\mapsto (x-p)/r$, then $T_{p,r} := (\iota_{p,r})_\sharp T$. Note that when $T= \a{\Gamma}$ for some smooth surface $\Gamma$ then $T_{p,r} = \a{\iota_{p,r} (\Gamma)}$.
\end{itemize}

\begin{lemma}[Canonical representation of integer rectifiable currents]\label{l:t_planews}
If $T$ is an integer rectifiable current, then for $\|T\|$-a.e. $p$ there is a unique plane $\pi (p)$ such that
\begin{equation}
T_{p,r} \to \Theta (T, p) \a{\pi (p)} \qquad \mbox{as $r\downarrow 0$.}
\end{equation}
$\pi (p)$ will then be called the tangent plane to $T$ at $p$.
If the sets $K_i$ and the submanifolds $\Sigma_i$ are as in Definition \ref{d:integral}, then $\pi (p) = T_p \Sigma_i$ (the classical oriented tangent to $\Sigma_i$ at $p$) for $\cH^m$-a.e. $p\in K_i$. Moreover, if $\vec{T}\|T\|$ is the Radon-Nykodim decomposition of $T$, then we can find $m$ Borel vector fields $e_1, \ldots e_m$ with the properties that
\begin{itemize}
\item $e_1 (p), \ldots , e_m (p)$ is a (positively oriented) orthonormal basis of $\pi (p)$ for $\|T\|$-a.e. $p$;
\item $\vec{T} (p) = e_1 (p) \wedge \ldots \wedge e_m (p)$ for $\|T\|$-a.e. $p$.
\end{itemize}
\end{lemma}

The outcome of Lemma \ref{l:density} and Lemma \ref{l:t_planews} is that
\[
T = \vec{T}\, \Theta\, \cH^m \res R\, ,
\]
where:
\begin{itemize}
\item $\Theta$ is integer valued,
\item $R$ is a countably $m$-rectifiable set (namely the union, up to an $\cH^m$negligible set, of countably many compact subsets of $C^1$ submanifolds) 
\item and $\vec{T}$ is an orientation for the tangent space to $R$ at $\cH^m$-a.e. point.
\end{itemize}

\subsection{The codimension $1$ case: sets of finite perimeter} Integral currents of codimension $1$ have a special feature: they can be seen, locally, as boundaries of integral currents of ``top dimension''. By definition, integer rectifiable currents $T$ of dimension $m+n$ are represented by $\sum k_i \a{E_i}$, where the $E_i$'s are pairwise disjoint closed subsets of $\R^{m+n}$: the action of $\a{E_i}$ on a ``top form'' $f dx_1 \wedge \ldots \wedge dx_{m+n}$ is then given by the standard Lebesgue integral
\[
\a{E_i} (f dx_1\wedge \ldots \wedge dx_{m+n}) = \int_{E_i} f (x)\, dx\, .
\] 

\begin{definition}[Caccioppoli sets, De Giorgi, cf. {\cite[Definition 3.35]{AFP}}]\label{d:CS}
A measurable set $E\subset \R^{m+1}$ with finite Lebesgue measure is a set of finite perimeter if its indicator function ${\bf 1}_E$ is a function of bounded variation, namely if
\[
{\bf P} (E) : = \sup \left\{\int_E {\rm div}\, X : \mbox{ $X \in C^\infty_c (\R^{m+1}, \R^{m+1})$ and $\|X\|_{C^0}\leq 1$}\right\} < \infty\, .
\]
${\bf P} (E)$ is called the {\em perimeter} of $E$.
\end{definition}

Like the mass, the perimeter can be localized to define a Radon measure. In fact such measure coincides with the total variation of the distributional derivative $D {\bf 1}_E$, which is usually denoted by $|D{\bf 1}_E|$ (cf. \cite[Theorem 3.3.6]{AFP}). If $E$ is a smooth set and $\Omega$ is an open set, then $\|D{\bf 1}_E\| (\Omega)$ is the $m$-dimensional volume of that portion of $\partial E$ which lies in $\Omega$. For this reason it is customary to use the notation ${\bf P} (E, \Omega)$ for the {\em relative perimeter} $\|D{\bf 1}_E\| (\Omega)$ when $\Omega$ is an open set and $E$ is a Caccioppoli set (again cf. \cite[Definition 3.3.5]{AFP}).

The fundamental link between the two theories is then given by the following

\begin{proposition}[Currents and sets of finite perimeter, cf. {\cite[Theorem 27.6 and Corollary 27.8]{Simon}}]\label{p:currents-sets}
Let $E$ be a measurable subset of $\R^{m+1}$ with finite Lebesgue measure. $\a{E}$ is then an integral current if and only if $E$ is a Caccioppoli set. Moreover, in this case $\mass (\partial \a{E}) = {\bf P} (E)$. 

Consider next an  $m$-dimensional integer rectifiable current $T$ with finite mass in $\R^{m+1}$ and let $\bB_\rho (x)\subset \R^{m+1}$ be such that $\supp (\partial T)\cap \bB_\rho (x) = 0$. Then there are countably many sets of finite perimeter $E_i$ and positive integer $k_i$ such that
\begin{itemize}
\item[(i)] $\|T\| (\bB_{\rho} (x)) = \sum_i k_i {\bf P} (E_i, \bB_\rho (x))$;
\item[(ii)] $T = \sum k_i \partial \a{E_i}$ on $\bB_\rho (x)$.
\end{itemize}
\end{proposition}

Theorem \ref{t:FFcomp}, Corollary \ref{c:existence}, Theorem \ref{t:FFrect} and Theorem \ref{t:FFiso} are all generalizations of theorems proved by De Giorgi for sets of finite perimeter (see \cite[Sections 3.3 and 3.5]{AFP}).

\subsection{Slicing} The final tool which we recall in this section is the slicing theory of currents. 

\begin{proposition}[Slicing, {cf. \cite[Section 4.3]{Fed}}]
Let $T$ be a current of dimension $m$ in $U\subset \R^{m+n}$ with $\mass (T) + \mass (\partial T) < \infty$. Consider $\pi: U\to \R^k$ Lipschitz with $k\leq m$ and the $m-k$-dimensional current $T\res d\pi$ defined by 
\[
T\res d\pi (\omega) := T (d\pi_1\wedge \ldots \wedge d\pi_k \wedge \omega)\, .
\]
Then there exists a map $\R^k \ni t \mapsto \langle T, \pi, t\rangle$ taking values in the space of $m-k$-dimensional currents which is (weakly) measurable and satisfies the following properties:
\begin{itemize}
\item $\langle T, \pi, t\rangle$ is supported in $\pi^{-1} (\{t\})$;
\item $T\res d\pi (\omega) = \int \langle T, \pi, t\rangle (\omega)\, dt$ for every test form $\omega$;
\item $\|T\res d\pi\| = \int \|\langle T, \pi, t\rangle\|\, dt$ (and thus $\langle T, \pi, t\rangle$ has finite mass for a.e. $t$).
\end{itemize}
Moreover, such slicing map satisfies the following additional properties
\begin{itemize}
\item $\partial \langle T, \pi, t \rangle = (-1)^k \langle \partial T, \pi, t\rangle$;
\item If $T$ is integer rectifiable, then $\langle T, \pi, t\rangle$ is integer rectifiable for a.e. $t$.
\end{itemize}
When $T = \Theta \vec{T} \cH^m \res R$ is integer rectifiable and $m=k$ then we have the explicit formula
\begin{equation}\label{e:slice}
\langle T, \pi, t \rangle = \sum_{x\in R \cap \pi^{-1} (\{t\})} \epsilon (x) \Theta (x) \a{x}
\end{equation}
where $\epsilon (x)$ takes the value $\pm 1$ depending on whether the restriction of $ d \pi|_x$ to the tangent plane $T_x R$ oriented by $\vec{T}$ preserves or reverses orientation ($\epsilon (x) =0$ when such restriction is not injective; for a.e. $t$, none of the $\epsilon (x)$ appearing in the sum \eqref{e:slice} vanishes). 
\end{proposition}

The slicing map generalizes the ``intersection'' of a smooth submanifold with the counterimages of a smooth map.

\section{First considerations in the regularity theory}\label{s:first}

Consider  a current $T$ as in the second statement of Corollary \ref{c:existence}, and assume that the set $K$ is a smooth complete submanifold $\Sigma$ without boundary.  Then $T$ must have the following minimality property. 

\begin{definition}\label{d:AM}
Let $\Omega\subset \R^{m+n}$ be open and $\Sigma\subset \R^{m+n}$ be a smooth complete submanifold without boundary of dimension $m+\bar{n}$. We say that an $m$-dimensional integer rectifiable current $T$ with finite mass is {\em area minimizing} in $\Sigma\cap \Omega$ (cf. \cite[Definition 33.1]{Simon}) if
\begin{itemize}
\item $\supp (T)\subset \Sigma$;
\item $\mass (T+\partial S) \geq \mass (T)$ for every $(m+1)$-dim. integral $S$ with $\supp (S)\subset \Sigma\cap \Omega$.
\end{itemize} 
The above definition can easily be extended to the case when $T$ has locally finite mass, by requiring $\|T+\partial S\| (U) \geq \|T\| (U)$ whenever $U$ is a bounded open set contanining $\supp (S)$.
\end{definition}

In the rest of the paper $\Sigma$ will always be sufficiently smooth and the number $\bar{n}$ will be called the {\em codimension} of $T$. All our efforts will now be dedicated to the following question: what kind of regularity is implied by the area minimizing condition? In order to make our discussion simpler, we define precisely the (interior) singular and regular sets of the current $T$:

\begin{definition}\label{e:regular_set}
We say that $p\in \supp (T)\setminus \supp (\partial T)$ is an interior regular point if there is a positive radius $r>0$, a smooth embedded submanifold $\Gamma$ of $\Sigma$ and a positive integer $Q$ such that $T\res \bB_r (p) = Q\a{\Gamma}$. The set of interior regular points, which of course is relatively open in $\supp (T)\setminus \supp (\partial T)$, is denoted by $\reg (T)$. Its complement $\supp (T)\setminus (\supp (\partial T) \cup \reg (T))$ is the interior singular set of $T$ and will be denoted by $\sing (T)$.
\end{definition}

Observe that, since we will always argue at the local level and the problem is scaling invariant, we can assume all sorts of nice properties upon $\Sigma$ (for instance that it is a global graph of a smooth function with good bounds on its $C^k$ norms). 

\medskip

 We next state the best theorems proved so far concerning the regularity of area minimizing currents.
The first theorem summarizes the achievements of several outstanding mathematicians from the end of the sixties till the nineties: De Giorgi, Almgren, Fleming, Simons, Federer, Bombieri, Giusti and Simon. It is fair to say that, with the notable exception of Simon's rectifiability result, the  various aspects of the following theorem have been well digested in the mathematical communities of elliptic PDEs 
and geometric analysis.

\begin{theorem}[Regularity in codimension 1]\label{t:cod1}
Assume that $\Omega$, $\Sigma$ and $T$ are as in Definition \ref{d:AM} and that $\bar{n}=1$. Then
\begin{itemize}
\item[(i)] For $m \leq 6$ $\sing (T)\cap \Omega$ is empty (Fleming \& De Giorgi (m=2), Almgren (m=3), Simons ($4\leq m \leq 6$), see \cite{DeGiorgi4,Fleming, DeGiorgi5,Alm2,Simons} and also \cite{Reifenberg2,Triscari});
\item[(ii)] For $m =7$ $\sing (T)\cap \Omega$ consists of isolated points (Federer, see \cite{Federer});
\item[(iii)] For $m \geq 8$ $\sing (T)\cap \Omega$ has Hausdorff dimension at most $m-7$ (Federer, \cite{Federer}) and it is countably $(m-7)$-rectifiable  (Simon, \cite{Simon2});
\item[(iv)] For every $m\geq 7$ there are area minimizing integral currents $T$ in the euclidean space $\R^{m+1}$ for which $\cH^{m-7} (\sing (T))>0$ (Bombieri-De Giorgi-Giusti, \cite{BDG}).
\end{itemize}
\end{theorem}

As already mentioned, after discussing the features of the codimension $1$ case, the rest of the note will be devoted to the understanding of the higher codimension, i.e. $\bar{n} \geq 2$. For this case the best results are the following.

\begin{theorem}[Regularity in codimension $\bar n \geq 2$]\label{t:cod>1}
Assume that $\Omega$, $\Sigma$ and $T$ are as in Definition \ref{d:AM} and that $\bar{n}\geq 2$. Then
\begin{itemize}
\item[(i)] For $m=1$ $\sing (T)\cap \Omega$ is empty;
\item[(ii)] For $m\geq 2$ $\sing (T)\cap \Omega$ has Hasudorff dimension at most $m-2$ (Almgren, \cite{Alm});
\item[(iii)] For every $m\geq 2$ there are area minimizing integral currents $T$ in $\R^{m+2}$ for which $\cH^{m-2} (\sing (T))>0$  (Federer, \cite{Federer3}).
\end{itemize}
\end{theorem}

Almgren's result was subsequently sharpened by Chang (cf. \cite{Chang}) for $2$-dimensional area minimizing currents.

\begin{theorem}[$m=2$, $\bar n \geq 2$]\label{t:Chang}
Assume that $\Omega$, $\Sigma$ and $T$ are as in Definition \ref{d:AM}, that $\bar{n}\geq 2$ and $m=2$. Then
$\sing (T)\cap \Omega$ consists of isolated points.
\end{theorem}

In Section \ref{s:branching} we will discuss extensively the difficulties that any argument for (ii) must face. Almgren's original typewritten proof was more than 1700 pages long and was published posthumously thanks to the efforts of his students Scheffer and Taylor in a book of almost 1000 pages. In this note we will describe the main steps of Almgren's program following the papers \cite{DS1,DS2,DS3,DS4,DS5} by Emanuele Spadaro and the author. 

Chang's result builds heavily on Almgren's book. Moreover Chang's paper \cite{Chang} does not provide the proof of one major step of the argument, the existence of a ``branched center manifold'': the construction of such object requires the understanding of 4/5 of Almgren's monograph and a suitable modification of its most obscure and involved part, which gives the construction of the ``non-branched center manifold'' (cf. Sections \ref{s:cm}, \ref{s:cm2} and \ref{s:N} below). Building upon \cite{DS1,DS2,DS3,DS4,DS5}, in forthcoming joint papers with Emanuele Spadaro and Luca Spolaor we will give the first proof of the existence of a ``branched center manifold'' and extend Chang's theorem to a large class of objects which are almost minimizing in a suitable sense, cf. \cite{DSS1, DSS2, DSS3, DSS4}.
That proof (and Chang's theorem) will however not be discussed in this survey.

\medskip

In the rest of this section we will recall the preliminaries of the regularity theory, most notably the monotonicity formula and its consequences. 

\subsection{Compactness} A first basic fact about area minimizing currents is that, under the same assumptions of the compactness theorem of Federer and Fleming, they are also a compact class.

\begin{theorem}[Compactness of area minimizing currents, cf. {\cite[Theorem 34.5]{Simon}}]\label{t:AM-cpt}
Let $\Sigma_k$ be a sequence of $C^2$ submanifolds of $\R^{m+n}$ of dimension $m+\bar{n}$ which converge in $C^2$ to $\Sigma$ and let $T_k$ be a sequence of integer rectifiable area minimizing currents in $\Sigma_k$ of dimension $m$ with $\sup_k \mass (T_k) < \infty$. Assume that $\partial T_k = 0$ on some open set $\Omega$ and that $T_k\res \Omega \to T$. Then
\begin{itemize}
\item $T$ is area minimizing in $\Omega\cap \Sigma$;
\item $\|T_k\|\res \Omega \weaks \|T\|$ in the sense of Radon measures.
\end{itemize}
\end{theorem}

\subsection{Stationarity and stability} Given a current $T$ supported in a submanifold $\Sigma$, we define the first and second variations along a smooth compactly supported vector field $X$ tangent to $\Sigma$ in the usual fashion (cf. \cite[Section 5.1.7]{Fed}): if we denote by $\Phi_t$ the one-parameter family of diffeomorphisms of $\Sigma$ generated by $X$, then
\[
\delta^k T (X) := \left. \frac{d^k}{dt^k}\right|_{t=0} \mass ((\Phi_t)_\sharp T)\, ,
\]
whenever $T$ has finite mass. If $T$ has only locally finite mass, we can still define its variations as
\[
\delta^k T (X) := \left. \frac{d^k}{dt^k}\right|_{t=0} \|(\Phi_t)_\sharp T\| (U)\, ,
\]
where $U$ is any bounded open set which contains $\supp (X)$. $\delta^1 T$ is usually denoted by $\delta T$ and is called first variation. $\delta^2 T$ is called second variation.
As usual we introduce stationary and stable currents as follows.

\begin{definition}
If $\Sigma$ is a complete smooth submanifold of $\R^{m+n}$ and $T$ an integer rectifiable current with $\supp (T)\subset \Sigma$ we then say that
\begin{itemize}
\item $T$ is stationary in $\Sigma$ if $\delta T (X) =0$ for every smooth compactly supported vector field $X$ tangent to $\Sigma$;
\item $T$ is stable in $\Sigma$ if it is stationary and $\delta^2 T (X) \geq 0$ for every smooth compactly supported vector field tangent to $\Sigma$.
\end{itemize}
\end{definition}

Of course area minimizing currents are stable. 

\subsection{The monotonicity formula} Testing the first variation condition with radial vector fields leads to the classical monotonicity formula (cf. \cite[Section 17]{Simon}).

\begin{theorem}[Monotonicity formula]\label{t:monot}
Let $T$ be an integer rectifiable $m$-dimensional current supported in a $C^2$ Riemannian manifold $\Sigma \subset \R^{m+n}$. If $T$ is stationary in $\Sigma$, then the following formula holds for every $x\not \in \supp (\partial T)$ and any $0<\sigma<\rho< \dist (x, \supp (\partial T))$:
\begin{align}
&\frac{\|T\| (\bB_\rho (x))}{\rho^m} - \frac{\|T\| (\bB_\sigma (x))}{\sigma^m} = \int_{\bB_\rho (x)\setminus \bB_\sigma (x)} \frac{|\nabla^\perp r|^2}{r^m}\, d\|T\|\nonumber\\
&\qquad\qquad + \frac{1}{m} \int_{\bB_\rho (x)} (y-x)\cdot H \left((\max \{r, \sigma\})^{-n} - \rho^{-n}\right)\, d\|T\| (y)\, ,
\end{align}
where
\begin{itemize}
\item $r (y) := |y-x|$;
\item $\nabla^\perp f (y)$ denotes the projection of $\nabla f (y)$ on the orthogonal complement of the tangent plane to $T$ at $y$;
\item $H (y) = \sum_i A_\Sigma (e_i, e_i)$, where $e_i$ is an orthonormal basis of the tangent plane to $T$ at $y$ and $A_\Sigma$ is the second fundamental form of $\Sigma$ (namely $A_\Sigma (\xi, \eta) = - (\nabla_\xi \eta)^\perp$, where $X^\perp$ denotes the projection of the vector field $X$ on the orthogonal complement of the tangent to $\Sigma$). 
\end{itemize}
\end{theorem}

A first obvious consequence of the monotonicity formula is that the density of an area minimizing current is in fact defined at {\em every} point: thus from now on we will use $\Theta (T, p)$ (cf. \eqref{e:density}) as a well defined quantity for every $p\not\in \supp (\partial T)$. Indeed it is a simple exercise to show that $\Theta$ is an upper semicontinuous function (and this because, modulo some small technicalities, the map $p\to \|T\| (\bB_r (p))$ is continuous at each fixed $r$). 

We next describe another crucial consequence of the monotonicity formula. Let $T$ be an area minimizing current (in the euclidean space, to simplify our discussion) and $p\in \supp (T)\setminus \supp (\partial T)$. Recall the homothetic rescalings $T_{p,r}$. It is obvious that each $T_{p,r}$ is an area minimizing current in $\iota_{p,r} (\Sigma)$. Observe also that:
\begin{itemize}
\item For each bounded open set $\Omega$, we clearly have $\supp (\partial T_{p,r}) \cap \Omega = \emptyset$, provided $r$ is small enough;
\item For each $R>0$ we have a uniform bound for $\|T_{p,r}\| (\bB_R (0))$: the latter is indeed the number $r^{-m} \|T\| (\bB_{Rr} (p))$, which is bounded independently of $r$ thanks to the monotonicity formula;
\item $\iota_{p,r} (\Sigma)$ converges, in $C^2$, to the tangent plane $T_p\Sigma$ to $\Sigma$ at $p$. 
\end{itemize}
Thus, by Theorem \ref{t:AM-cpt}, for every fixed bounded open $\Omega$ we can extract a subsequence $\{T_{p, r_k}\}_{r_k\downarrow 0}$ which converges in $\Omega$ to an area minimizing current $T_0$. Actually, by a standard diagonal argument we can find a ``global'' limit current $T_0$ which is an integral current on each bounded open subset of $\R^{m+n}$, which has no boundary
and which is area minimizing in $T_p \Sigma$. Since the latter is a linear subspace of $\R^{m+n}$ it is not difficult to see that $T_0$ is then area-minimizing in $\R^{m+n}$.

Now, it is also obvious that $r\mapsto r^{-m} \|T_0\| (\bB_r)$ is a constant and equals $\omega_m \Theta (T, p)$. The right hand side of the monotonicity formula implies that this is possible if and only if $T_0$ is a {\em cone} with vertex $0$, namely $(\iota_{0,r})_\sharp T_0 = T_0$. The standard terminology for $T_0$ is {\em tangent cone to $T$ at $p$}: note however that its uniqueness (i.e. the convergence of every sequence $T_{p, r_k}$ to the same cone) is not guaranteed and in fact it is one of the major unsolved problems in the field. The uniqueness has been proved for $2$-dimensional currents $T$ in {\em any} codimension by White in his remarkable paper \cite{White2} and it has been shown by Simon in codimension $1$ at any isolated singularity in the fundamental work \cite{Simon3}. The latter result is indeed a consequence of a remarkably general approach, which applies to other variational problems (such as the uniqueness of tangent maps to energy minimizing maps) but also to the study of the asymptotic behaviour of solutions to parabolic equations, see \cite{Simon3}.

\medskip

This motivates the following

\begin{definition}[Tangent cones]\label{d:tcones}
An area minimizing cone of dimension $m$ in $\R^{m+n}$ is an area-miminimizing integer rectifiable current $S$ of dimension $m$  such that
$\partial S=0$ and $S_{0,r} = S$ for every positive $r$.
Next, if $T$ and $S$ are two currents such that, for some $p\in \supp (T)$ and some $r_k\downarrow 0$, $T_{p, r_k}$ converges to $S$, we then say that $S$ is {\em tangent to $T$ at $p$}. 
\end{definition}

We summarize in the following theorem the most important consequences of the monotonicity formula:

\begin{theorem}[Tangent cones, {cf. \cite[Section 7.3]{Simon}}]\label{t:cones}
Let $T$ be an area minimizing integral current of dimension $m$ in a $C^2$ submanifold $\Sigma$ of dimension $m+\bar{n}$ in $R^{m+n}$. Then
\begin{itemize}
\item[(i)] $r\mapsto e^{C\bA r} r^{-m} \|T\| (\bB_r (p))$ is a monotone function for each $p\not\in \supp (\partial T)$ on the interval $]0, \dist (p, \supp (\partial T)[$, where $C$ is a suitable dimensional constant and $\bA$ is the $C^0$ norm of the second fundamental form of $\Sigma$;
\item[(ii)] The density $\Theta (T, p)$ is well defined at every $p\not\in \supp (\partial T)$, it is at least $1$ at each point $p\in \supp (T)\setminus \supp (\partial T)$ and it is upper semicontinuous;
\item[(iii)] For every $p\not\in \supp (\partial T)$ and every sequence $r_k\downarrow 0$ there is a subsequence, not relabeled, and an area minimizing cone $T_0$ such that $T_{p, r_k}\to T_0$; $T_0\neq 0$ if and only if $p\in \supp (T)$;
\item[(iv)] If $T_k$ and $\Sigma_k$ are as in Theorem \ref{t:AM-cpt}, then $\supp (T_k)$ converges, locally in $\Omega$ in the sense of Hausdorff, to $\supp (T)$.
\end{itemize}
\end{theorem}

One technical detail that plays a crucial role in the proof of Theorem \ref{t:cod>1} is the following ``improvement'' of Theorem \ref{t:cones}(i)

\begin{theorem}[{Cf.  \cite[Lemma A.1]{DS3}}]\label{t:monot_imp}
There is a constant $C (m,\bar{n}, n)$ with the following properties.
Let $T$, $\Sigma$ and $\bA$ be as in Theorem \ref{t:cones}. Then
\[
r\mapsto e^{C \bA^2 r^2} \frac{\|T\| (\bB_r (p))}{r^m} \quad \mbox{is monotone}
\]
on the interval $]0, \dist (p, \supp (\partial T))[$ for every $p\not \in \supp (\partial T)$. 
\end{theorem}

\section{The regularity theory in codimension $1$}\label{s:cod-1}

The first breakthrough in the regularity theory is due to De Giorgi: he realized in his fundamental work \cite{DeGiorgi4} that in codimension $1$
the existence of one flat tangent plane at $p$ is enough to conclude that $p$ is a regular point. His theorem was then extended to any codimension by Almgren in \cite{Alm3} (see also \cite{SS}) under an important assumption on the density which we will discuss extensively in a moment (indeed the latter theorem can be suitably generalized to Hilbert spaces, cf. \cite{ADS}). In fact Almgren's statement covers many more geometric functionals, which satisfy an appropriate ellipticity assumption. In the framework of minimal (i.e. only {\em stationary}) surfaces the most important generalization of De Giorgi's $\eps$-regularity theorem is due to Allard in \cite{Allard} (cf. also \cite[Chapter 4]{Simon} and \cite{DLAllard}): his theorem, valid for a far reaching generalization of classical stationary surfaces (namely integer rectifiable varifolds with sufficiently summable generalized mean curavature) is the starting point of a variety of applications of the minimal surface theory to geometric and topological problems. 

\medskip

We will first recall the De Giorgi-Almgren $\eps$-regularity theorem in all dimensions and codimensions, after introducing the key parameter of ``flatness''

\begin{definition}[Spherical excess]\label{d:eccesso}
Let $T$ be an integer rectifiable $m$-dimensional current and $\pi$ be an $m$-dimensional plane, oriented by the unit simple $m$-vector $\vec{\pi}$. The (spherical) excess of $T$ in the ball $\bB_\rho (p)$  with respect to $\pi$ is the quantity
\begin{equation}\label{e:eccesso_pi}
\bE (T, \bB_\rho (p), \pi) := \frac{1}{\omega_m \rho^m} \int_{\bB_\rho (p)} |\vec{T} (x) - \vec{\pi}|^2 \, d\|T\| (x)\, .
\end{equation}
The {\em excess} in $\bB_\rho (p)$ is 
\begin{equation}\label{e:eccesso}
\bE (T, \bB_\rho (p)) :=\min \{\bE (T, \bB_\rho (p), \pi): \pi \mbox{ is an oriented $m$-plane}\}\, . 
\end{equation}
If $\pi$ achieves the minimum in the right hand side of \eqref{e:eccesso} we then say that $\pi$ {\em optimizes} the excess.
\end{definition}

Since we will often deal with $m$-dimensional balls in $m$-dimensional planes $\pi$, we introduce here the notation $B_r (p, \pi)$ for the set $\bB_r (p) \cap (p+\pi)$. 

\begin{theorem}[$\varepsilon$-regularity]\label{t:DG}
Let $T$ be an $m$-dimensional integer rectifiable area minimizing current in a $C^2$ submanifold $\Sigma$ of dimension $m+\bar{n}$. There are constants $\alpha>0$, $\eps>0$ and $C$, depending only upon $m$ and $\bar{n}$, such that the following holds. Assume that for some $\rho>0$ and some $m$-dimensional plane $\pi$ we have
\begin{itemize}
\item[(a)] $\partial T \res \bB_{2\rho }(p) = 0$;
\item[(b)] $\Theta (T, p) =Q$ and $\Theta = Q$ $\|T\|$-a.e. on $\bB_{2\rho} (p)$, for some positive integer $Q$;
\item[(c)] $\|T\| (\bB_{2\rho} (p)) \leq (Q \omega_m +\eps) (2\rho)^m$;
\item[(d)] $E := \bE (T, \bB_{2\rho} (p), \pi) < \eps$ and $\rho \bA := \rho \max_{\Sigma\cap \bB_{2\rho} (p)} |A_\Sigma| < \eps$.
\end{itemize}
Then $T\res \bB_\rho (p) = Q \a{\Gamma}$ for a surface $\Gamma$ which is the graph of a suitable $C^{1,\alpha}$ function $u: B_r (p, \pi)\to \pi^\perp$. Moreover $[Du]_{0, \alpha} \leq C (E^{\sfrac{1}{2}} + \rho \bA) \rho^{-\alpha}$. 
\end{theorem}

From now on we will often encounter oriented graphs of (at least Lipschitz) functions $u$ and the corresponding integer rectifiable currents: the latter will be denoted by $\bG_u$, whereas we will use $\gr (u)$ for the set-theoretic graph.

\subsection{De Giorgi's idea}\label{ss:DG} The essential point in the proof of Theorem \ref{t:DG} is that, under the above assumptions, the current $T$ is close to the graph of an harmonic function $v$ and hence the excess in a ball is close to the $L^2$ mean oscillation of the gradient of $Dv$: since the latter quantity has nice decay properties for harmonic functions, it is possible to show that the spherical excess of $T$ decays suitably. We illustrate roughly this well-known idea because we want to stress that the H\"older exponent $\alpha$ in Theorem \ref{t:DG} can be made arbitrarily close to $1$. Presently this is a very minor observation because any $C^{1, \alpha}$ graph which is stationary for the area functional enjoys higher regularity by the classical Schauder estimates. However we will see later that the remark is crucial in the proof of Theorem \ref{t:cod>1}. 

More precisely, assume a-priori that $T$ is the graph $\bG_u$ of a Lipschitz function $u: p+\pi \to \pi^\perp$. For every $\Omega \subset \pi$ we can compute the mass of $T$ in the cylinder $\bC:= \Omega \times \pi^\perp$ using the area formula:
\[
\mass (T\res \bC) = \int_\Omega \sqrt{1 + |Du|^2 + \sum_{k\geq 2} \sum [\det (M^k (Du))]^2}
\]
where with $M^k (Du)$ we denote an arbitrary $k\times k$ minor of $Du$.

The smallness of the excess $\bE (T, \bB_{2\rho} (p), \pi)$ indicates that most of the tangents to $\bG_u$ are close to the horizontal plane $\pi$: if this control  were uniform we could make a Taylor expansion of the integrand to achieve
\[
\mass (T\res \bC) = |\Omega| + \frac{1}{2} \int_\Omega |Du|^2 + O (|Du|^4)\, .
\]
Thus we can assume that $u$ is rather close to a minimizer of the Dirichlet energy, i.e. that it is close to an harmonic function $v$.

Recall moreover $\mass (T\res \bC)-|\Omega| = \frac{1}{2} \int_{\bC} |\vec{T}-\vec{\pi}|^2\, d\|T\|$. Similarly,
 we can compare $\bE (T, \bB_{2\rho}, \pi)$ to the average integral 
\[
\mint_{B_{2\rho} (p, \pi)} |Du|^2\, 
\]
and analogous computations show that $\bE (T, \bB_\rho (p))$ is comparable to a similar ``optimized'' quantity
\[
\min_A \mint_{B_\rho (p, \pi)} |Du-A|^2 = \mint_{B_\rho (p, \pi)} \Big|Du- \mint_{B_\rho (p, \pi)} Du\Big|^2\, .
\]
For harmonic functions $v$ we have the following decay estimate, which could be proved using the expansion of the trace $v|_{\partial B_\rho (p, \pi)}$ in spherical harmonics
(see \cite[Chapter 5, Section 2]{SW})
\begin{equation}\label{e:decay_arm}
\mint_{B_\rho (p, \pi)} \Big|Dv- \mint_{B_\rho (p, \pi)} Dv\Big|^2 \leq \frac{1}{4}\, \mint_{B_{2\rho} (p, \pi)} |Dv|^2\, .
\end{equation}
We could then hope to transfer such decay to the current in the form
\begin{equation}\label{e:decay}
\bE (T, \bB_\rho (p)) \leq 2^{-2+2\delta} \bE (T, \bB_{2\rho} (p), \pi)\, ,
\end{equation}
where the constant $\delta> 0$ takes into account (quite a few) error terms, which must be carefully estimated. 

Note however that we could optimize on the plane in the right hand of \eqref{e:decay} to achieve 
\begin{equation}\label{e:decay_2}
\bE (T, \bB_\rho (p)) \leq 2^{-2+2\delta} \bE (T, \bB_{2\rho} (p))\, .
\end{equation}
In turn this latter estimate would imply that the assumption (d) of Theorem \ref{t:DG} holds also in the ball $\bB_\rho (p)$. Since all other assumptions are automatically satisfied at any scale smaller than $2\rho$ (the monotonicity formula plays a crucial role here), we could then iterate the argument to obtain the decay
\begin{equation}\label{e:decay_3}
\bE (T, \bB_r (p)) \leq C r^{2-2\delta}\, .
\end{equation}
However, if we choose $\varepsilon$ sufficiently small, this process can be applied replacing $\bB_r (p)$ with any ball $\bB_r (q)$ such that $q\in \bB_{\rho/2} (p)$ and $r\leq \rho$. We have thus achieved a ``Morrey-type'' estimate which easily shows that $\supp (T)\cap \bB_{r/2}$ is contained in the graph of a $C^{1,\alpha}$ function $u$, where $\alpha= 1-\delta$. The Constanty theorem (cf. \cite[Theorem 26.27]{Simon}) implies that $T$ must be (an integer multiple of) $\bG_u$ in $\bB_{r/2} (p)$. 

\begin{remark}\label{r:alfa=1}
As already mentioned, the argument sketched above leads to the conclusion that $\alpha$ in Theorem \ref{t:DG} can be taken arbitrarily close to $1$, at the price of making the threshold $\eps$ suitably small and the constant $C$ fairly large. For a detailed proof we refer the reader to \cite[Corollary 2.4 and Appendix A]{DS-cm}. 
The subsequent generalizations of  Almgren \cite{Alm3,Alm4}, Allard \cite{Allard} and other authors (cf. for instance \cite{Bombieri} and \cite{SS}) of De Giorgi's $\varepsilon$-regularity statement do not imply the same decay of the excess by simply halving the radius, but seems to require a more subtle adjustment of the parameters: a (small) sacrifice to flexibility, since the latter results can be applied to much more general objects and situations. 
\end{remark}

\subsection{First consequences of the $\eps$-regularity theorem} It is rather simple to see that the conditions (a), (c) and (d) will be met at a sufficiently small radius $\rho$ as soon as $p\in \supp (T)\setminus \supp (\partial T)$ and there is at least {\em one} flat tangent cone at $p$. However condition (b) discriminates severely between the codimension $1$ case ($\bar n=1$) and the higher codimensions. Indeed, Propostion \ref{p:currents-sets} shows that in codimension $1$ a current without boundary can be described as a ``superposition'' of boundaries of finitely many Caccioppoli sets $E_i$. In particular, we can reduce the regularity question for area-minimizing currents $T$ under the additional assumption that
\begin{equation}\label{e:bordo}
T = \partial \a{E} \qquad\qquad \mbox{for some measurable $E\subset \Sigma$ with $\cH^{m+1} (E)<\infty$.}
\end{equation}
Under the additional assumption \eqref{e:bordo} the existence of one flat tangent cone at a point $p$ guarantees that $\Theta (T, p)=1$, which
in turn implies condition Theorem \ref{t:DG}(b) for a sufficiently small $\rho$, because of the upper semicontinuity of the density.
We then achieve the following corollary (cf. \cite[Section 37]{Simon}). 

\begin{corollary}\label{c:DG}
If $T$ is an area minimizing current of dimension $m$ in a $C^2$ submanifold $\Sigma$ of dimension $m+1$, then any point $p$ at which there is a flat tangent cone is a regular point.
\end{corollary}

Indeed the ``sheeted structure'' given to codimension $1$ integral currents by Proposition \ref{p:currents-sets} has a much stronger consequence, which we record in the following

\begin{proposition}\label{p:convergence}
Let $T$, $\Sigma$, $T_k$ and $\Sigma_k$ be as in Theorem \ref{t:AM-cpt} and assume that the codimension $\bar n$ is $1$. If $T$ is regular in $\Omega$, then for any open set $\Gamma\subset\subset \Omega$, $T_k$ is regular in $\Gamma$ for $k$ large enough. 
\end{proposition} 

For the sake of our future discussions we will rephrase the proposition above in the following equivalent way, underlying that ``singularities persist in the limit'': we will stress later on that this persistence can be seen as the major difference between the
codimension $1$ and the higher codimension.

\begin{proposition}[Persistence of singularities in codimension 1]\label{p:persiste}
Let $\Omega$, $T$, $\Sigma$, $T_k$ and $\Sigma_k$, be as in Theorem \ref{t:AM-cpt} and assume that the codimension $\bar n$ is $1$. If $p_k\in \sing (T_k)$ and $p_k\to p\in \Omega$, then $p\in \sing (T)$.
\end{proposition}

In higher codimension there is no analog of Proposition \ref{p:currents-sets} which allows to bypass Assumption (b) in Theorem \ref{t:DG} and we will see that indeed Corollary \ref{c:DG} fails, together with the stronger propositions stated above. 

The upper semicontinuity of the density and an elementary argument implies only that Assumption (b) of Theorem \ref{t:DG} holds in a dense set, leading to the well-known

\begin{corollary}[Density of $\reg (T)$, cf. {\cite[Theorem 36.2]{Simon}}]\label{c:AE}
If $T$ is an area minimizing current of dimension $m$ in a $C^2$ submanifold $\Sigma$ of dimension {\em larger} than $m+1$, then $\reg (T)$ is dense in $\supp (T)\setminus \supp (\partial T)$.
\end{corollary}

Indeed this statement has been recently extended to any Hilbert space, cf. \cite{ADS}. Remarkably, Corollary \ref{c:AE} was the best regularity result available before the appearance of Almgren's manuscript \cite{Alm} with its proof of Theorem \ref{t:cod>1}. In fact this is the current situation for {\em stationary} integer rectifiable $m$-dimensional varifolds: Allard's theorem gives the regularity up to a meager closed set, even in codimension $1$, and this is up to now the best regularity result available in the literature for stationary objects. In particular it is not known that the singular set is $\cH^m$-negligible, not even in the simplest setting of stationary $2$-dimensional varifolds in $3$ dimensions. For stable hypersurfaces a rather satisfactory theory is instead available thanks to the pioneering works of Schoen - Simon - Yau \cite{SSY} and Schoen - Simon \cite{SS2}  and to the recent ones of Wickramasekera \cite{Wickramasekera}.

\subsection{Full regularity for $m\leq 6$ and $\bar{n}=1$: Simons' theorem} Let us now focus on the case of codimension $\bar{n}=1$. Corollary \ref{c:DG} naturally leads to discuss the existence of area minimizing (hyper-)cones which are not flat. 
The investigations upon this question were started by De Giorgi and Fleming who could show full regularity for $m=2$, cf. \cite{Fleming} and \cite{DeGiorgi5}. Moreover, De Giorgi showed that the problem of deciding whether every codimension $1$ area minimizing cone in $\R^{m+1}$ is flat is equivalent to decide whether any entire minimal (hyper-) graph in $\R^{m+2}$ is affine, the so-called Bernstein problem. The result of De Giorgi and Fleming was subsequently improved by Almgren (\cite{Alm2}, $m=3$) and finally by Simons in \cite{Simons} to show full regularity for $m\leq 6$.

Observe that if we know full regularity for area-minimizing (hyper)currents of dimension $m\leq m_0$ then there is no singular area minimizing hypercone in $\R^{m_0+1}$. As a consequence any area minimizing hypercone in $\R^{m_0+2}$ is necessarily regular except possibly at the origin: thus the cross-section is a minimal (i.e. stationary) embedded hypersurface of the standard $(m_0+1)$-dimensional sphere. In fact the cone $0\cone \a{\Gamma}$ over a closed submanifold $\Gamma$ of $\partial \bB_1 (0)$ is stationary in $\bB_1 (0)$ if and only if $\Gamma$ is stationary in $\partial \bB_1 (0)$. However the minimizing property is a stronger information: in particular $0\cone \a{\Gamma}$ is necessarily  
{\em stable}. The famous theorem of Simons is the following statement.

\begin{theorem}[Simons]\label{t:simons}
Let $2\leq m\leq 6$.
Any {\em stable} minimal hypersurface of $\R^{m+1}$ which is a cone over a minimal submanifold of $\partial \bB_1 (0)$ is necessarily an $m$-dimensional plane.
\end{theorem}

Theorem \ref{t:simons} and Corollary \ref{c:DG} settle the regularity of area minimizing currents of dimension $m\leq 6$ in codimension $1$.

\subsection{Simons' cone and the Theorem of Bombieri-De Giorgi-Giusti} In his celebrated paper \cite{Simons} Simons provided also an example which showed the optimality of his theorem. More precisely  he showed that the cone over $\mathbb S^3\times \mathbb S^3 \subset \mathbb S^7$, namely
\begin{equation}\label{e:simons}
S:= \{x\in \mathbb R^8 : x_1^2 +x_2^2 +x_3^2 + x_4^2 = x_5^2 + x_6^2 + x_7^2 + x_8^2\}
\end{equation}
is stationary and {\em stable}. The surface is usually called Simons' cone in the literature which followed \cite{Simons}. Later Bombieri, De Giorgi and Giusti in \cite{BDG} showed that $S$ is indeed an area minimizing cone and were thus able to settle the Bernstein problem in all dimensions.

\begin{theorem}[Bombieri-De Giorgi-Giusti]\label{t:BDG}
$S$ in \eqref{e:simons} is an area minimizing current in $\R^8$ and therefore for any $n\geq 8$ there are functions $u: \R^n \to \R$ which satisfy the minimal surface equation and are not affine.
\end{theorem}

For an elegant and simple proof we refer the reader to \cite{Guido}.

\subsection{Federer's reduction argument}\label{ss:Fedred}  In this section we review the basic idea behind Federer's partial regularity theorem for $m\geq 7$, namely Theorem \ref{t:cod1}(ii)\&(iii). The crucial (elementary!) observation is that, if $S$ is an area-mininizing cone (with vertex at the origin) in $\R^{m+\bar n}$ and $S_0$ is a tangent cone to $S$ at a point $p\neq 0$, then $S_0$ is invariant under translations in direction $p$. It is then not difficult to conclude that $S_0$ ``splits off a line''.

\begin{lemma}[Splitting off lines, cf. {\cite[Lemma 35.5 and proof of Theorem 35.3]{Simon}}]\label{l:split}
If $S_0$ is a tangent cone to an area minimizing $m$-dimensional cone $S$ in $\R^{m+\bar n}$ at a point $p\neq 0$, then $S_0 = \a{\R}\times Z$ for some $(m-1)$-dimensional area minimizing cone $Z$ in $\R^{m-1+\bar n}$. 
\end{lemma}

Thus, if we fix an area-minimizing cone $S$ of dimension $7$ in $\R^8$ we conclude from Theorem \ref{t:simons} that any cone tangent to $S_0$ at a point which is not the origin must necessarily be flat and Corollary \ref{c:DG} implies that $S$ is smooth except, possibly, at the origin. Let now $T$ be an area-minimizing current of dimension $7$ in a Riemannian submanifold $\Sigma\subset \R^{7+n}$ of dimension $8$ and fix a point $p\in \sing (T)$. 
If $p$ were not an isolated singularity, there would be a sequence of singular points $\{p_k\}$ converging to $p$ and, upon setting $r_k:= |p_k-p|$ we could consider the rescaled currents $T_{p, r_k}$. These currents would have a singularity at $0$ and a singularity at $q_k = \frac{p-p_k}{r_k}$. Up to subsequences we could assume that $q_k$ converges to some $q$ with $|q|=1$ and that $T_{p, r_k}$ converges to some tangent cone $S$. $S$ is an area-minimizing cone in $T_p \Sigma$, which is an $8$-dimensional euclidean space. The above discussion gives that $S$ is regular outside the origin, but Proposition \ref{p:persiste} would imply that the point $q$ is a singular point for $S$. 

A repeated application of Lemma \ref{l:split} allows to infer a similar conclusion for any $m\geq 8$. In particular, assume that for some $m\geq 8$ there is an area-minimizing current of dimension $m$ in a Riemannian manifold $\Sigma$ of dimension $m+1$ such that $\cH^{m-7+\alpha} (\sing (T)) > 0$. By the classical results on densities (cf. \cite[Theorem 3.2]{Simon}) there is a point $p\in \supp (T)\setminus \supp (\partial T)$ such that
\[
\limsup_{r\downarrow 0} \frac{\cH^{m-7+\alpha} (\sing (T)\cap \bB_r (p))}{r^{m-7+\alpha}} \geq 2^{-(m-7+\alpha)}\, .
\]
We can thus assume the existence of an area minimizing cone $S$ of dimension $m$ in $T_p \Sigma$ and of a subsequence of rescalings $T_{p, r_k}$ converging to it for which 
\begin{equation}\label{e:LB_haus}
\limsup_{k\to\infty} \cH^{m-7+\alpha} (\sing (T_{p, r_k})\cap \bar\bB_1)\geq 2^{-(m-7-\alpha)}\, .
\end{equation}
In addition, if $\Gamma$ is the limit, in the Hausdorff sense, of $\sing (T_{p, r_k}) \cap \bar \bB_1$, we conclude from Proposition \ref{p:persiste} that $\Gamma \subset \sing (S)$. It is well known that the Hausdorff measures are not upper semicontinuous under Hausdorff convergence and thus we cannot use \eqref{e:LB_haus} to conclude $\cH^{m-7+\alpha} (\sing (S)) >0$. On the other hand the $\cH^{\beta}_\infty$ measures are upper semicontinuous, (cf. \cite[Appendix A]{Simon}). Hence if we choose a point $\bar{p}$ such that
\[
\limsup_{r\downarrow 0} \frac{\cH_\infty^{m-7+\alpha} (\sing (T)\cap \bB_r (\bar p))}{r^{m-7+\alpha}} \geq 2^{-(m-7+\alpha)}\, 
\]
(which is possible, cf. \cite[Theorem 3.6]{Simon}) and argue as above with $\bar{p}$ in place of $p$, we then find an area-minimizing cone $S$ with the property that $\cH^{m-7+\alpha}_\infty (\sing (S))>0$. We can then repeat the procedure to find a second cone $S_0$, tangent to $S$ at some $\bar{q}\neq 0$, with the property that $\cH^{m-7+\alpha}_\infty (\sing (S_0))>0$. Since $S_0$ has the product structure given by Lemma \ref{l:split}, we infer the existence of an area-minimizing current $S_1$ of dimension $m-1$ in $\R^m$ such that $\cH^{(m-1)-7+\alpha}_\infty (\sing (S_1))>0$. Repeating this argument $(m-7)$ times we arrive at an area minimizing current $S_{m-7}$ of dimension $7$ in $\R^8$ which has $\cH^\alpha_\infty (\sing (S_{m-7}))>0$, contradicting our previous conclusion that $\sing (S_{m-7})$ is a discrete set. 

\subsection{Simon's rectifiability result} We complete our survey of the regularity results in codimension $1$ by mentioning Simon's spectacular achievement: combining his fundamental theorem about the uniqueness of tangent cones at isolated singularities with several additional innovative ideas, he has been able in the nineties to show that, when $\bar{n}=1$, $\sing (T)$ can be covered, up to a set of $\cH^{m-7}$-measure zero, by a countable collection of $C^1$ $(m-7)$-dimensional submanifolds, cf. \cite{Simon2}.  A new proof of Simon's theorem, which avoids the discussion of the uniqueness of tangent cones at isolated singularities, has been very recently found by Naber and Valtorta, see \cite{NV}. This is till now the best description available for the behavior of the singular set in codimension $1$.

\section{Federer's theorem and the failure of $\eps$-regularity in codimension $\bar{n}\geq 2$}\label{s:branching}

\subsection{Holomorphic subvarieties as area minimizing currents}
We start by recalling that holomorphic subvarieties of ${\mathbb C}^{k+j}$, namely zeros of holomorphic maps $u: {\mathbb C}^{k+j}\to {\mathbb C}^j$ ($k$ and $j$ being, respectively, the complex dimension and codimension of the variety) can be given a natural orientation. In what follows we identify $\mathbb C^{k+j}$ with $\R^{2k+2j}$ in the usual way: if $z_1, \ldots , z_{k+j}$ are complex coordinates and $x_j = {\rm Re}\, z_j$, $y_j = {\rm Im}\, z_j$, we let $x_1, y_1, \ldots , x_{k+j}, y_{k+j}$ be the standard coordinates of $\R^{2k+2j}$. 
Recall then that an holomorphic subvariety $\Gamma$ of $\mathbb C^{k+j}$ of complex dimension $k$ is a (real analytic) submanifold of $\R^{2k+2j}\setminus \sing (\Gamma)$ of (real) dimension $m = 2k$, where $\sing (\Gamma)$ is an holomorphic subvariety of complex dimension $k-1$.

Furthermore, at each point $p\in \Gamma\setminus \sing (\Gamma)$, the (real) tangent $2k$-dim. plane $T_p \Gamma$ can be identified with a complex $k$-dimensional plane of $\mathbb C^n$. If $v_1, \ldots, v_k$ is a complex basis of $T_p \Gamma$, we can then define a canonical orientation for $T_p \Gamma$ using the simple $2k$-vector
\[
{\rm Re}\, v_1 \wedge {\rm Im}\, v_1 \wedge \ldots \wedge {\rm Re}\, v_k\wedge {\rm Im}\, v_k\, .
\]
This allows us to define the current $\a{\Gamma}$ by integrating forms over the oriented submanifold $\Gamma\setminus \sing (\Gamma)$. It is also easy to check that $\partial \a{\Gamma} =0$, the reason being that  the ``singular set'' $\sing (\Gamma)$ is a set of (locally) finite $\cH^{2k-2}$ measure. 

\medskip

The discussion can be ``localized'' to holomorphic subvarieties in open subsets $\Omega$ of $\mathbb C^{k+j}$ (and more generally in complex hermitian manifolds). Note also that, if $\Omega'$ is a bounded open subset of the domain $\Omega$ where $\Gamma$ is defined, then $\a{\Gamma}$ has finite mass in $\Omega'$ and it is thus an integer rectifiable current. The following fundamental observation is due to Federer and is based on a classical computation of Wirtinger (\cite{Wirtinger}).

\begin{theorem}[Federer, cf. {\cite[Section 5.4.19]{Fed}}]\label{t:FW}
Let $\Gamma_1, \ldots \Gamma_N$ be holomorphic subvarities of complex dimension $k$ in $\Omega\subset \mathbb C^{k+j}$ and let $k_1, \ldots , k_N$ be positive integers. Then the current $T:= k_1 \a{\Gamma_1} + \ldots + k_N \a{\Gamma_N}$ is area minimizing in $\Omega$.
\end{theorem}

Indeed the above theorem holds in general K\"ahler manifolds, cf. \cite[5.4.19]{Fed}.

\subsection{Branching phenomena} Before giving an idea of why Theorem \ref{t:FW} holds we want to illustrate the deep consequences that it has in the regularity theory for area minimizing currents in codimension higher than $1$. Holomorphic subvarieties give easy counterexamples to Corollary \ref{c:DG} when $\bar{n}>1$: assumption (b) in Theorem \ref{t:DG} is absolutely crucial in this case. As a byproduct even Proposition \ref{p:persiste} fails and singularities might disappear in the limit when we deal with sequences of area minimizing currents in codimension higher than $1$: in the rest of this note we will see that the core difficulty in the proof of Theorem \ref{t:cod>1} is precisely this phenomenon of ``disappearance of singularities''. We illustrate these points with three explicit examples.

\begin{example}\label{e:bad}
Let $\delta>0$ be a small number and consider the holomorphic curve 
\[
\Gamma_\delta := \{(z,w)\in \mathbb C^2: z^2 = \delta w\}
\]
and the plane 
\begin{equation}\label{e:orizzontale}
\pi:= \{(z,w)\in \mathbb C^2 : z= 0\}\, .
\end{equation}
There is no neighborhood of $0$ where $\Gamma_\delta$ is the graph of a function $z= f (w)$, in spite of the fact that $\bE (\a{\Gamma_\delta}, \bB_1 (0), \pi)$ converges to $0$ as $\delta\downarrow 0$. In fact the conclusion of Theorem \ref{t:DG} does not apply: although each $\Gamma_\delta$ is smooth and it is graphical in $\bB_\rho (0)$ for any $\rho$, there is no uniform control of the $C^{1, \alpha}$ norm of the graph in terms of the excess. Observe that the $\Gamma_\delta$ do not satisfy the condition (c) in Theorem \ref{t:DG}, although they satisfy (a), (b) and (d).
\end{example}

\begin{example}\label{e:vbad}
Consider the holomorphic curve 
\[
\Gamma := \{(z,w)\in \mathbb C^2: z^2 = w^3\}\, .
\]
The origin belongs to $\sing (\a{\Gamma})$. On the other hand:
\begin{itemize}
\item The unique tangent cone at $0$ is given by $2\a{\pi}$ for $\pi$ as in \eqref{e:orizzontale}. 
\item The density of $\a{\Gamma}$ equals $2$ at $0$;
\item 
\[
\lim_{r\downarrow 0} \bE (\a{\Gamma}, \bB_r (0), \pi) = 0\, .
\]
\end{itemize}
Therefore:
\begin{itemize}
\item Corollary \ref{c:DG} is false for $2$-dimensional area minimizing currents in $\R^4$: $\Gamma$ is singular at the origin in spite of the existence of a flat tangent cone there.
\item Again Theorem \ref{t:DG} does not apply in any ball $\bB_{2\rho} (0)$. Note however that the only missing assumption is (b): the density $\Theta (\a{\Gamma}, p)$ equals $1$ at every point $p\in \Gamma \setminus \{0\}$ and equals $2$ at $p=0$.
\item Proposition \ref{p:persiste} fails for $2$-dimensional area minimizing currents in $\R^4$. Indeed $0$ is a singular point for $\a{\Gamma}_{0,r}$ for every positive $r>0$. On the other hand $\a{\Gamma}_{0,r}\to 2 \a{\pi}$ and thus $0$ is {\em not} a singular point of the limit: the singularity ``has disappeared''.
\end{itemize}
\end{example}

\begin{example}\label{e:vvbad}
Consider finally the holomorphic curve
\[
\Xi := \{(z,w)\in \mathbb C^2: (z- w^2)^2 = w^{2015}\}\, .
\]
All the considerations valid for the holomorphic curve $\Gamma$ of Example \ref{e:vbad} are also valid for $\Xi$. $\Xi$ does not add much for the moment to our discussion, but it will play a crucial role later: observe that $0$ is a singular point in spite of the fact that $\Xi$ is an almost imperceptible perturbation of the smooth current $2\a{\{z=w^2\}}$.
\end{example}

We close this section by remarking that Theorem \ref{t:FW} gives also a great abundance of singular area minimizing  cones in higher codimension: the zero set of any homogeneous polynomial $P (z_1, \ldots, z_{k+1})$ in $k+1$ complex variables is an area minimizing cone of dimension $2k$ in $\R^{2k+2}$. More generaly, for any projective subvariety of $\mathbb P^k \mathbb C$ with complex dimension $j$ we can construct a corresponding area-minimizing cone in $\R^{2k+2}$ of dimension $2j+2$. These cones are singular except when the corresponding algebraic subvarieties are affine. The easiest example of a singular area minimizing cone is thus the union of an arbitrary number of complex lines in $\mathbb C^2$. Such cones might however be considered ``mildly'' singular: in $\mathbb C^3$ the generic cone associated to a projective curve of $\mathbb P^2 \mathbb C$ has a singular set which behaves in rather complicated way. 

\subsection{Calibrations and the proof of Theorem \ref{t:FW}} We illustrate here the simple, yet deep, principle lying behind Theorem \ref{t:FW}. Recall first the notion of comass of a form, given in Definition \ref{d:comass}.

\begin{definition}[Calibrations, cf. \cite{HL}]\label{d:calibrations}
A {\em calibration} $\omega$ is a closed $m$-form such that $\|\omega\|_c\leq 1$. An integer rectifiable current $T$ is said to be calibrated by a calibration $\omega$ if $\langle \omega_p, \vec{T} (p)\rangle = 1$ for $\|T\|$-a.e. $p$. 
\end{definition}

Observe in particular that $\mass (T) \geq T (\omega)$ whenever $\omega$ is a calibration and that the equality sign holds if and only if $T$ is calibrated by $\omega$. The following is then a trivial fact

\begin{lemma}
If $T$ is calibrated by a calibration $\omega$, then $T$ is an area minimizing current.
\end{lemma}
\begin{proof}
Let $S$ be an $(m+1)$-dimensional integral current. Then
\begin{align*}
\mass (T)= & T (\omega) = T (\omega) + S(d\omega) = (T+\partial S) (\omega) \leq \mass (T+\partial S)\, .\qedhere
\end{align*}
\end{proof}

Holomorphic subvarieties are the primary example of calibrated currents and this observation dates back essentially to Wirtinger. More precisely, if $z_\ell = x_\ell + i y_\ell$ are the standard coordinates in $\mathbb C^{k+j}$, consider the K\"ahler form
\[
\omega := dx_1 \wedge dy_1 + \ldots + dx_{k+j} \wedge dy_{k+j}\, .
\]
Wirtinger's theorem can then be stated in the following form

\begin{theorem}[Wirtinger, cf. \cite{Wirtinger}]
If $\omega$ is the K\"ahler form and 
\[
\omega^k = \frac{1}{k!} \underbrace{\omega \wedge \ldots \wedge \omega}_{\mbox{$k$ times}}\, 
\]
then $\omega^k$ is a calibration. Moreover, $\langle \omega^k, v_1 \wedge \ldots \wedge v_{2k}\rangle = |v_1\wedge \ldots \wedge v_{2k}|$ if and only if $v_1, \ldots , v_{2k}$ is a positively oriented ($\R$-)base of a complex plane.
\end{theorem}

Calibrations and calibrated submanifolds are a rich source of interesting geometries: we refer the reader to \cite{HL} for several important examples.

\section{Almgren's stratification}\label{s:strat}

From now on we will focus on area minimizing currents in codimension $\bar{n}>1$ and we proceed in describing the proof of Theorem \ref{t:cod>1}. The first step in the analysis is the elegant generalization of Federer's reduction argument known as Almgren's stratification and widely used to analyze singularities in geometric analysis. 

Given an area minimizing $m$-dimensional cone $S$ we define its {\em spine} as the vector space $V$ of maximal dimension for which $S$ can be written as $S'\times \a{V}$, where $S'$ is an area minimizing cone of dimension $m- {\rm dim}\, (V)$. Equivalently, $V$ is the subset of those vectors $v$ such that $S$ is invariant under translations in direction $v$ and it is a simple exercise (using the monotonicity formula) to show that $V$ can be characterized as the subset of those points $p\in \supp (S)$ such that $\Theta (S, p) = \Theta (S, 0)$ or, equivalently, as the subset of those points $q$ such that $S_{q,1}$ is also an area minimizing cone (cf. \cite[Proof of Lemma 35.5]{Simon}). 

At the intuitive level it is clear that $S$ must have a certain ``asymmetry'' in the directions which are transversal to $V$. The dimension of the spine of $S$ is called the {\em building dimension of the cone $S$} (cf. \cite{White4}). Note
that such building dimension equals $m$ if and only if $S$ is an integer multiple of an $m$-dimensional plane, namely if and only if $S$ is flat.

\begin{definition}
A point $p\in \supp (T)$ belongs to the stratum $\mathcal{S}_k (T)$ if every tangent cone to $T$ at $p$ has building dimension at most $k$ and if there is {\em at least one} tangent cone to $T$ at $p$ with building dimension $k$.
\end{definition}

Almgren's generalization of Federer's argument can then be stated in the following theorem.

\begin{theorem}[Almgren's stratification, cf. {\cite{White4}}]\label{t:strat}
For any given area minimizing current $T$ the stratum $\mathcal{S}_k (T)$ has Hausdorff dimension at most $k$
and $\mathcal{S}_0 (T)$ is a discrete set.
\end{theorem} 

Observe that the discussion of Section \ref{ss:Fedred} proves that:
\begin{itemize}
\item[(F)] Given any area minimizing cone $S$ of codimension $\bar{n}=1$, either such cone is a multiple of an $m$-dimensional plane, or its building dimension is at most $m-7$.
\end{itemize}
As a corollary we conclude that for $\bar{n}=1$ the strata 
\[
\mathcal{S}_{m-1} (T), \mathcal{S}_{m-2} (T),\quad \ldots\quad  , \mathcal{S}_{m-6} (T)
\] 
are all empty. Next, at any point $p$ in the top stratum $\mathcal{S}_m (T)$ there is a flat tangent cone and thus, by Corollary \ref{c:DG}, we actually know that $\mathcal{S}_m (T) = \reg (T)$ (we stress again that this holds only under the assumption that $\bar{n}=1$: Example \ref{e:vbad} shows its failure as soon as $\bar{n}=2$ and $m=2$). We therefore conclude that $\sing (T) = \mathcal{S}_0 (T) \cup \ldots \cup \mathcal{S}_{m-7} (T)$ and thus Theorem \ref{t:cod1}(ii)\&(iii) is a corollary of Theorem \ref{t:strat}.

\medskip

Unfortunately from Section \ref{s:branching} we know that the identity $\reg (T) = \mathcal{S}_m (T)$ does not hold anymore when the codimension $\bar{n} > 1$. On the other hand we surely have $\reg (T) \subset \mathcal{S}_m (T)$. We could call ``branch points'' for $T$ those points $p\in \mathcal{S}_m (T)\setminus \reg (T)$. The major concern in the rest of the note will be to estimate the Hausdorff dimension of $\mathcal{S}_m (T)\setminus \reg (T)$. A simple consequence of Theorem \ref{t:strat} is that, in order to prove Theorem \ref{t:cod>1}(ii), the ``only'' concern is truly to bound the Hausdorff dimension of the set of branch points by $m-2$, because of the following lemma.

\begin{lemma}[cf. {\cite[Theorem 35.3]{Simon}}]\label{l:strat}
The stratum $\mathcal{S}_{m-1} (T)$ is empty in any codimension $\bar{n}$. 
\end{lemma}

The proof is very elementary: given the definition of the strata we just need to show that the building dimension of an area-minimizing $m$-dimensional cone can never be $m-1$: in turn this is equivalent to show that $1$-dimensional area-minimizing cones are never singular, namely they are always single straight lines counted with integer multiplicity. The latter statement has a rather trivial proof. 

\medskip

The proof of Theorem \ref{t:strat} is relatively elementary and ``soft''. In spite of this the idea is powerful and can be applied to several different problems in geometric analysis; for instance, we refer the reader to Simon's work on the singularities of harmonic maps, \cite{Simon6}, to White's far-reaching generalization of Theorem \ref{t:strat} and its applications to the mean-curvature flow, \cite{White4}, and to recent results about Riemannian manifolds with one-sided curvature bounds, see for instance \cite{CN1}. Recently, in a series of works (cf. \cite{CN2, CHN, CNV}), the method of Almgren has been extended to deal with the Minkowski content, see also \cite{FMS} for an abstract general version of this.

We finally mention that the cones with building dimension $m-2$ can be actually further characterized: it is not difficult to see that such cones are necessarily unions of multiples of $m$-dimensional planes. The spines of such cones are $(m-2)$-dimensional subspaces. Due to the remarkable work of White, \cite{White2}, when $m=2$ there is one such unique tangent cone at every point $p\in \mathcal{S}_0 (T)$. However, for $m\geq 3$ the same uniqueness result is not yet proved and in fact it is not even known whether at points $p\in \mathcal{S}_m (T)\setminus \reg (T)$ the flat tangent cone is the unique one!

\section{Multiple valued functions and the Dirichlet energy}\label{s:multi_1}

As already noticed, in codimension $1$ the regularity in a neighborhood of a point with integer multiplicity $Q$
where at least one tangent cone is flat can be reduced to the case of multiplicity $Q=1$, whereas the discussions of Section \ref{s:branching}
show that this reduction is impossible in codimension
larger than $1$. Indeed, in the Examples \ref{e:vbad} and \ref{e:vvbad} even the starting point of De Giorgi's strategy as described in Section \ref{ss:DG} fails
dramatically: no matter how small is the neighborhood $U$ of the origin that we choose, it is simply not possible to approximate efficiently the corresponding current $T$ in $U$ with the graph of a (single valued) function. However, in each of these examples the current turns out to be a ``multivalued'' graph, where the number of values is in fact determined by the multiplicity $Q=\Theta (T, 0)$.
This discussion motivates the starting idea of Almgren's monograph:
in order to go beyond an Allard's type statement (namely regularity in a dense relatively open subset of $\supp (T)\setminus \supp (\partial T)$) we need to develop an efficient theory for  ``multiple valued functions'' minimizing a suitable generalization of the Dirichlet energy, where we can (and we will) consider the multiplicity to be a constant preassigned positive integer $Q$.

\subsection{The metric space of unordered $Q$-tuples}
The obvious model case to keep
in mind is the following. Given two integers $k, Q$ with ${\rm MCD}
(k, Q)=1$, look at the set valued map which assigns to each point $z\in \C$
the set $M (z):=\{w^k: w^Q=z\}\subset \C$. Obviously for each
$z$ we can choose some arbitrary ordering $\{u_1 (z), \ldots, u_Q (z)\}$ of the elements of the set $M(z)$.
However, it is not possible to do it in such a way
that the resulting ``selection maps'' $z\mapsto u^i (z)$ are continuous: even at the local level, this is impossible in
every neighborhood of the origin.

Our example motivates the following definition. Given an integer $Q$ we 
define a $Q$-valued map from a set $E\subset \R^{m}$ into $\R^{n}$
as a function which to each point $x\in E$ associates an unordered
$Q$-tuple of vectors in $\R^{n}$. Following Almgren, we consider the group
$\Pe_Q$ of permutations of $Q$ elements and we let $\Iqs$ be
the set $(\R^{n})^Q$ modulo the equivalence relation
$$
(v_1, \ldots, v_Q) \equiv (v_{\pi (1)}, \ldots, v_{\pi (Q)})
\qquad \forall \pi\in \Pe\, .
$$
Hence a multiple valued map is simply a map taking values in $\Iqs$.
There is a fairly efficient formulation
of this definition which will play a pivotal role in our discussion, because the set $\Iqs$ can be naturally identified with a subset of the set
of measures (cf. \cite{Alm} and \cite[Definition 0.1]{DS1}).

\begin{definition}[Unordered $Q$-tuples]\label{d:IQ}  
Denote by $\a{P_i}$ the Dirac mass in $P_i\in \R^{n}$. Then,
\begin{equation*}
\Iqs :=\left\{\sum_{i=1}^Q\a{P_i}\,:\,P_i\in\R^{n}\;\textrm{for every  }i=1,\ldots,Q\right\}\, .
\end{equation*}
\end{definition}

Observe that with this definition each element of $\Iqs$ is in fact a $0$-dimensional integral current.
This set has also a natural metric structure; cf. \cite{Alm} and
\cite[Definition 0.2]{DS1} (the experts will
recognize the well-known Wasserstein $2$-distance, cf. \cite{villani}).

\begin{definition}\label{d:FeG} For every
$T_1, T_2\in\Iqs$, with $T_1=\sum_i\a{P_i}$ 
and $T_2=\sum_i\a{S_i}$, we set
\begin{equation}\label{e:defG}
\cG(T_1,T_2)\;:=\;\min_{\sigma\in\Pe_Q}\sqrt{\sum_i\abs{P_i-S_{\sigma(i)}}^2}\, .
\end{equation}
\end{definition}

\begin{remark}\label{r:modulo}
Since we will often need to compute $\cG (T, Q\a{0})$ we introduce the special notation $|T|$ for the latter quantity. Observe, however, that $\Iqs$ is 
{\em not} a linear space except for the special case $Q=1$: the map $T\to |T|$ {is not a norm}. 
\end{remark}

\subsection{$Q$-valued maps} Using the metric structure on $\Iqs$ one defines obviously
measurable, Lipschitz and H\"older maps from subsets of $\R^{m}$
into $\Iqs$. One important point to be made is about the existence of ``selections''. A selection for a $Q$-valued function $u$ is given by $Q$ classical single valued functions $u_1, \ldots , u_Q$ such that $u (x) = \sum_{i=1}^Q \a{u_i (x)}$, cf. \cite[Definition 1.1]{DS1}. If the $u_i$ are measurable, continuous, Lipschitz, etc. the selection will be called measurable, continuous, Lipschitz, etc. It is rather easy to show that a measurable selection exists for any measurable $u$, cf. \cite[Proposition 0.4]{DS1}. Incidentally, this will be used repeatedly as we write
\[
\sum_i \a{u_i}
\]
for any given measurable $Q$-valued map $u$, tacitly assuming to have chosen some measurable selection.

However continuous maps (resp. Sobolev, Lipschitz) do not possess in general selections which are continuous (resp. Sobolev, Lipschitz): the primary examples are the maps 
stemming from holomorphic subvarieties already discussed at length. If, however, they do, the corresponding selection will be called {\em regular}. Only maps defined on $1$-dimensional intervals are a notable exception, since they always have regular selections.

\subsection{The generalized Dirichlet energy: geometric definitions} If we want to approximate area minimizing
currents with multiple valued functions and ``linearize'' the area
functional in the spirit
of De Giorgi, we need to define a suitable concept of 
Dirichlet energy. We will now show how this can be done naturally, proposing three different approaches. 

Consider the model case of $Q=2$ and assume $u:\Omega \to 
\Is{2}$ is a Lipschitz map. If, at some point $x$, $u(x) = \a{P_1} + \a{P_2}$ is
``genuinely $2$-valued'', i.e. 
$P_1\neq P_2$, then there exist obviously a ball $B_r (x)\subset \Omega$
and a regular (Lipschitz) selection, namely Lipschitz classical maps $u_1, u_2: B_r (x) \to \R^{n}$ such that
$u (y) = \a{u_1 (y)} + \a{u_2 (y)}$ for every $y\in B_r (x)$. On the other hand on the closed set where the values of $u$ are ``collapsed'' we can find a single Lipschitz map $v$ such that $u = 2 \a{v}$. It is easy to generalize this to the $Q$-valued case and to maps defined on a manifold $\Sigma$:

\begin{lemma}[Decomposition, {cf. \cite[Lemma 1.1]{DS2}}]\label{l:chop}
Let $M \subset \Sigma$ be measurable and $F: M\to \Iqs$ Lipschitz. Then 
there are a countable partition of $M$ in bounded measurable subsets $M_i$ ($i\in \N$)
and Lipschitz functions $f^j_i: M_i\to \R^{n}$ ($j\in \{1, \ldots, Q\}$) such that
\begin{itemize}
\item[(a)] $F|_{M_i} = \sum_{j=1}^Q \a{f^j_i}$ for every $i\in \N$ and $\Lip (f^j_i)\leq
\Lip (F)$ $\forall i,j$;
\item[(b)] $\forall i\in \N$ and $j, j' \in \{1, \ldots ,Q\}$,
either $f_i^j \equiv f_i^{j'}$ or
$f_i^j(x) \neq f_i^{j'}(x)$ $\forall x \in M_i$. 
\end{itemize}
\end{lemma}

The Dirichlet energy can be defined for Lipschitz maps $F$ as above by 
\begin{equation}\label{e:Dir_1}
\D (F, \Sigma) := \sum_{i,j} \int_{M_i} |Df^j_i|^2\, .
\end{equation}
$W^{1,2}$ maps and their Dirichlet energy can then be defined by relaxation: assuming that $\Omega$ is open
$W^{1,2} (\Omega, \Iqs)$ consists of those
measurable maps $v: \Omega \to \Iqs$ for which there is a sequence
of Lipschitz maps $u_k$ converging to $v$ a.e. and
enjoying a uniform bound $\D (\Omega, u_k) \leq C$.
The Dirichlet energy $\D (\Omega, v)$ is the infimum
of all constants $C$ for which there is a sequence with the
properties above.

\medskip

Another possible definition of the Dirichlet energy follows more closely Almgren's original
idea: for $Q$-valued maps we can introduce a notion of differentiability in the following way

\begin{definition}[$Q$-valued differential, {cf. \cite[Definition 1.9]{DS1}}]\label{d:diff}
Let $f:\Om\to\Iq$ and $x_0\in\Om$. We say that $f$ is differentiable
at $x_0$ if there exist $Q$ matrices $L_i$ satisfying:
\begin{itemize}
\item[(i)]$\cG(f(x),T_{x_0} f)=o(\abs{x-x_0})$, where
\begin{equation}\label{e:taylor1st}
T_{x_0} f(x):=\sum_i\a{L_i\cdot(x-x_0)+f_i(x_0)};
\end{equation}
\item[(ii)] $L_i=L_j$ if $f_i(x_0)=f_j(x_0)$.
\end{itemize}
The $Q$-valued map $T_{x_0} f$
will be called the {\em first-order approximation} of $f$ at $x_0$.
The point $\sum_i \a{L_i} \in \Iq (\R^{n\times m})$
will be called the differential of $f$ at $x_0$ and is
denoted by $Df (x_0)$. 
\end{definition}

A Rademacher's type theorem shows that Lipschitz $Q$-valued maps are differentiable almost everywhere and that the Dirichlet
energy defined above corresponds to the integral of $|Df| = \cG (Df, Q\a{0})$ (our notation is consistent, cf. Remark \ref{r:modulo}). 
The proof is in fact a straightforward corollary of
Lemma \ref{l:chop} and elementary measure theory.

\begin{proposition}[$Q$-valued Rademacher, {cf. \cite[Theorem 1.13]{DS1} and \cite[Lemma 1.1]{DS2}}]
Let $f:\Om\to\Iq$ be a Lipschitz function. Then, $f$ is differentiable
almost everywhere in $\Om$ and 
\[
\D (f, \Omega) = \int_\Omega |Df|^2\, ,
\]
where the left hand side is understood in the sense of \eqref{e:Dir_1}. In particular the expression in \eqref{e:Dir_1}
is independent of the decomposition given by Lemma \ref{l:chop}.
\end{proposition}

Almgren's definition of Sobolev map does not follow, however, a ``relaxation procedure'' but uses the (biLipschitz) embedding of $\Iqs$ in 
a large Euclidean space, see below.

\subsection{$W^{1,2}$ and the generalized Dirichlet energy: metric analysis definition} Although the definition above is certainly very natural and gives a good geometric
intuition for the Dirichlet energy, it turns out that it is rather complicated to work with it, in
particular if one wants to recover the usual statements of the Sobolev space theory for
classical functions.

Instead, a rather efficient way to achieve such statements is to rely on a more
abstract definition of Dirichlet energy and Sobolev functions, as proposed in \cite{DS1}. A very general
theory has been developed in the literature for Sobolev maps taking values in abstract metric spaces, following the
pioneering works  of Ambrosio \cite{ambrosio} and Reshetnyak \cite{reshetnyak,reshetnyak2}. The careful
reader will notice, however, that there is a crucial difference between
the definition of Dirichlet energy in \cite{reshetnyak} and the one
given below.

\begin{definition}[Sobolev $Q$-valued functions, cf. {\cite[Definition 0.5]{DS1}}]
\label{d:W1p}
A measurable 
$f:\Omega\to\Iq$ is in the Sobolev class
$W^{1,p}$ ($1\leq p\leq\infty$) if there exist $m$ functions
$\varphi_j\in L^p(\Omega;\R^{+})$
 such that 
\begin{itemize}
\item[(i)] $x\mapsto\cG (f(x),T)\in W^{1,p}(\Omega)$ for all $T\in \Iq$;
\item[(ii)] $\abs{\de_j\, \cG (f, T)}\leq\varphi_j$ a.e. in $\Omega$
for all $T\in \Iq$ and
for all $j\in\{1, \ldots, m\}$.
\end{itemize}
\end{definition}

It is not difficult to show the existence 
of minimal functions $\tilde\varphi_j$ fulfilling (ii),
i.e. such that, for any other $\varphi_j$ satisfying (ii),
$\tilde\varphi_j\leq\varphi_j$ a.e. (cf. \cite[Proposition 4.2]{DS1}).
Such ``minimal bounds'' will be denoted by $|\de_j f|$ and we note that
they are characterized by the following
property
(see again \cite[Proposition 4.2]{DS1}): for every countable
dense subset $\{T_i\}_{i\in\N}$ of $\Iq$ and
for every $j=1,\ldots,m$,
\begin{equation}\label{e:def|D_jf|}
\abs{\de_j f}\;=\;\sup_{i\in\N}\abs{\de_j\,\cG(f,T_i)}
\quad\textrm{almost everywhere in } \Om.
\end{equation}

We are now ready to give an abstract characterization of the Dirichlet energy.

\begin{proposition}[Cf. {\cite[Proposition 2.17]{DS1}}]\label{d:dirichlet}
If $u\in W^{1,2}$ is Lipschitz then 
\begin{equation}\label{e:identita}
|\partial_j u|^2 (x) = \sum_i |Du_i (x) \cdot e_j|^2
\end{equation}
at a.e. point $x$ of differentiability of $u$, where $\sum_i \a{Du_i (x)}$ is the $Q$-valued differential in the sense of Definition \ref{d:diff} and $|\partial_j u|$ is as in \eqref{e:def|D_jf|}.
\end{proposition}

The main feature of the above proposition is that essentially all
the conclusions of the usual Sobolev space theory for single valued functions can be now reduced to
routine modifications of the usual arguments: among them we mention Sobolev and Morrey embeddings, compact embeddings,
Poincar\'e inequalities, semicontinuity results, trace properties (cf. \cite[Chapter 4]{DS1}). 

We list here some of these facts and refer to \cite{DS1} for their proof.

\begin{definition}[Trace of Sobolev $Q$-functions]
\label{d:Dirichlet problem}
Let $\Omega\subset \R^{m}$ be a Lipschitz bounded open set
and $f\in W^{1,p} (\Omega, \Iq)$. A function $g$ belonging to
$L^p (\partial \Omega , \Iq)$ is said to be the trace of $f$ at $\partial \Omega$
(and we denote it by  $f|_{\partial\Omega}$) if,
for every $T\in \Iq$, the trace of the real-valued Sobolev function
$\cG (f, T)$ coincides with $\cG (g, T)$.
\end{definition}

\begin{definition}[Weak convergence]\label{d:weak convergence}
Let $f_k, f\in W^{1,p}(\Om,\Iq)$. We say 
that $f_k$ converges weakly 
to $f$ for $k \to \infty$, 
(and we write $f_k\rightharpoonup f$) in $W^{1,p}(\Om,\Iq)$,
if
\begin{itemize}
\item[(i)] $\int\cG(f_k,f)^p\to 0$, for $k\to\infty$;
\item[(ii)] there exists a constant
$C$ such that $\int |Df_k|^p\leq C<\infty$ for every $k$.
\end{itemize}
\end{definition}

\begin{proposition}[Weak sequential closure]\label{p:trace}
Let $f\in W^{1,p}(\Om,\Iq)$. Then, there is a unique function
$g\in L^p(\de \Om,\Iq)$ such that $f\vert_{\de\Om}=g$ 
in the sense of Definition 
\ref{d:Dirichlet problem}.
Moreover, $f\vert_{\de\Om}=g$ if and only if
$\cG (f,T)\vert_{\de\Om}= \cG (g, T)|_{\partial \Omega}$ for every $T$
in the usual sense, and the set of mappings
\begin{equation}\label{e:trace g}
W^{1,2}_g(\Omega, \Iq) := 
\left\{ f\in W^{1,2} (\Omega, \Iq)\,:\,
f|_{\partial \Omega}
= g\right\}
\end{equation}
is sequentially weakly closed in $W^{1,2}$.
\end{proposition}

\begin{proposition}[Sobolev Embeddings]\label{p:Sembeddings}
The following embeddings hold:
\begin{itemize}
\item[$(i)$] if $p<m$, then $W^{1,p}(\Om,\Iq)
\subset L^{q}(\Om,\Iq)$ for
every $q\in [1,p^*]$, and the inclusion is compact
when $q<p^*$;
\item[$(ii)$] if $p=m$, then $W^{1,p}(\Om,\Iq)
\subset L^{q}(\Om,\Iq)$, for
every $q\in [1,+\infty)$, with compact inclusion; 
\item[$(iii)$] if $p>m$, then $W^{1,p}(\Om,\Iq)
\subset C^{0,\alpha}(\Om,\Iq)$, for
$\alpha=1-\frac{m}{p}$, with compact inclusion.
\end{itemize}
\end{proposition}

\begin{proposition}[Poincar\'e inequality]\label{p:poincare'}
Let $M$ be a connected
bounded Lipschitz open set of an $m$-dimensional
Riemannian manifold and let $p<m$.
There exists a constant $C=C(p,m,n,Q,M)$ with the following
property:
for every $f\in W^{1,p}(M,\Iq)$,
there exists a point $\overline f\in \Iq$ such that
\begin{equation}\label{e:poincare'}
\left(\int_{M} \cG \big(f, \overline f\big)^{p^*}\right)^{\frac{1}{p^*}}
\leq C \left(\int_{M}|Df|^{p}\right)^{\frac{1}{p}}.
\end{equation}
\end{proposition}

\begin{proposition}[Campanato-Morrey]\label{p:Campanato Morrey}
Let $f\in W^{1,2}(B_1,\Iq)$, with $B_1\subset \R^m$, and $\alpha\in(0,1]$ be such that
\begin{equation*}%\label{e:decay}
\int_{B_r (y)}|Df|^2
\leq A\; r^{m-2+2\alpha}\quad
\mbox{for every $y\in B_1$ and a.e. $r\in ]0,1-|y|[$.}
\end{equation*}
Then, for every $0<\delta<1$,
there is a constant $C=C(m,n,Q,\delta)$ with
\begin{equation}\label{e:Campanato Morrey}
\sup_{x,y\in
\overline{B_{\delta}}}\frac{\cG(f(x),f(y))}{\abs{x-y}^\alpha}
=: \left[f\right]_{C^{0,\alpha}(\overline{B_{\delta}})}
\leq C\,\sqrt{A}.
\end{equation}
\end{proposition}

\begin{lemma}[Interpolation Lemma]\label{l:technical}
There is a constant $C=C(m,n,Q)$ with the following property.
Let $r>0$, $g\in W^{1,2}(\de B_r,\Iq)$ and 
$f\in W^{1,2}(\partial B_{r(1-\eps)},\Iq)$.
Then, there exists $h\in W^{1,2}(B_r\setminus B_{r(1-\eps)},\Iq)$ 
such that $h|_{\de B_r} = g$, $h|_{\de B_{r(1-\eps)}}=f$ and 
\begin{multline}
\D(h,B_r\setminus B_{r(1-\eps)})\leq
 C\,\eps\,r\, 
\big[\D(g,\de B_r)+\D(f,\de B_{r(1-\eps)})\big]+\\
+\frac{C}{\eps\,r}\int_{\de B_r}
\cG\left(g(x),f\left((1-\eps)\,x\right)\right)^2 dx.
\label{e:joining}
\end{multline}
\end{lemma}

\subsection{Lipschitz approximation and approximate differentiability} An important feature of classical Sobolev maps is the existence of suitable smooth approximations. Since the space $\Iqs$ is itself rather singular and lacks any linear structure, the usual approximation results are indeed much more subtle. However a robust way to approximate Sobolev maps is to ``truncate them'' along the level sets of the Hardy-Littlewood maximal function of the modulus of their gradient. This is possible in the setting of $Q$-valued maps as well and will play a crucial role in the sequel. 

\begin{proposition}[Lipschitz approximation, {cf. \cite[Proposition 4.4]{DS1}}]\label{p:lipapprox metric}
There exists a constant $C=C(m,\Om,Q)$ with the
following property. For every $f\in W^{1,p}(\Om,\Iq)$
and every $\lambda>0$, there exists a $Q$-function
$f_\lambda$ such that ${\rm Lip}\, (f_\lambda)\leq C \,\lambda$,
\begin{equation}\label{e:approx2}
|E_\lambda| = \abs{\big\{x\in\Om\,:\,f(x)\neq
f_\lambda(x)\big\}}\leq \frac{C\||Df|\|^p_{L^p}}{\lambda^p}\, .
\end{equation}
$\Omega\setminus E_\lambda$ can be assumed to contain $\{x\in \Omega : M (|Df|) \leq \lambda\}$, where $M$ is the maximal function operator. 
\end{proposition}

A simple corollary of the previous proposition is that Sobolev maps are ``approximate differentiable'' in the following sense:

\begin{definition}[Approximate Differentiability]\label{d:appdiff}
A $Q$-valued function $f$ is approximately differentiable in
$x_0$ if there exists a measurable subset $\tilde{\Omega}
\subset \Omega$ containing $x_0$
such that $\tilde{\Omega}$ has
density $1$ at $x_0$ and $f|_{\tilde{\Omega}}$ is differentiable
at $x_0$. 
\end{definition}

\begin{corollary}\label{c:appdiffae}
Any $f\in W^{1,p} (\Omega ,\Iq)$ is approximately differentiable a.e.
\end{corollary}

\subsection{Chain rule formulas: $Q$-valued calculus} The latter property is very useful to extend classical computations like the chain rule
to Sobolev maps. Indeed, it is rather easy to extend such formulas to Lipschitz maps using the multivalued differentiability: Proposition \ref{p:lipapprox metric} can then be used to routinely justify the same formulas for general Sobolev maps. 

Consider a function $f:\Om \to \Iqs$.
For every $\Phi: \tilde{\Omega}\to \Omega$,
the right composition $f\circ \Phi$ defines
a $Q$-valued function on $\tilde{\Omega}$. 
On the other hand, given a map $\Psi: \Omega \times \R^{n} \to \R^{k}$, we
can consider the left composition,
$x \mapsto \sum_i \a{\Psi (x,f_i (x))}$,
which defines a $Q$-valued function denoted, with
a slight abuse of notation, by $\Psi (x,f)$. 

\begin{proposition}[Chain rules, {cf. \cite[Proposition 1.12]{DS1}}]\label{p:chain}
Let $f:\Omega\to \Iqs$ be differentiable at $x_0$.
\begin{itemize}
\item[$(i)$] Consider $\Phi: \tilde{\Omega}\to \Omega$
such that $\Phi (y_0)=x_0$ and assume that
$\Phi$ is differentiable at $y_0$. 
Then, $f\circ \Phi$ is differentiable at $y_0$ and
\begin{equation}\label{e:interna}
D (f\circ \Phi) (y_0) = \sum_i 
\a{D f_i (x_0)\cdot D\Phi (y_0)}.
\end{equation}
\item[$(ii)$] Consider $\Psi: \Omega_x\times \R^{n}_u\to \R^{k}$
such that $\Psi$ is differentiable at $(x_0, f_i (x_0))$
for every $i$. Then, $\Psi (x,f)$ is differentiable at $x_0$ and
\begin{equation}\label{e:esterna}
D\Psi (x,f)) (x_0)=\sum_i 
\a{D_u\Psi (x_0, f_i (x_0))\cdot Df_i (x_0)
+ D_x \Psi (x_0, f_i (x_0))}.
\end{equation}
\end{itemize}
\end{proposition}

\subsection{Almgren's extrinsic maps}\label{ss:extrinsic}
The metric $\cG$ on $\Iqs$ is ``locally euclidean'' at most of the points.
Consider for instance the model case $Q=2$ and a point $P= \a{P_1} 
+ \a{P_2}$ with $P_1\neq P_2$. Then, obviously, in a sufficiently small
neighborhood of $P$, the metric space $\Is{2}$ is isometric to the
Euclidean space $\R^{2n}$. This fails instead in any neighborhood
of a point of type $P= 2\a{P_1}$. On the
other hand, if we restrict our attention to the closed subset
$\{2 \a{X} : X\in \R^{n}\}$, we obtain the metric structure of 
$\R^{n}$. 

A remarkable observation of Almgren is that
$\Iqs$ is biLipschitz equivalent to a deformation retract of the Euclidean
space (cf. \cite[Section 1.3]{Alm}). For a simple presentation of 
this fact we refer the reader to \cite[Section 2.1]{DS1}.

\begin{theorem}[Almgren's embedding and retraction]\label{t:xi}
There exists $N=N(Q,n)$ and an injective
$\xii:\Iqs\to\R^{N}$ such that:
\begin{itemize}
\item[$(i)$] $\Lip(\xii)\leq1$;
\item[$(ii)$] if $\cQ= \xii(\Iq)$, then
$\Lip(\xii^{-1}|_\cQ)\leq C(n,Q)$.
\end{itemize}
%\begin{propos}\label{p.rho} Let $\cQ$ be the set of Theorem
%\ref{t:xi}. 
Moreover there exists a Lipschitz map $\ro:\R^{N}\to\cQ$
which is the identity on $\cQ$.
%\end{propos}
\end{theorem}

In fact much more can be said: the set $\cQ$ is  
a cone and a polytope. On each separate face of the
polytope the metric structure induced by $\cG$ is euclidean,
essentially for the reasons outlined a few paragraphs above
(cf. again \cite[Section 1.3]{Alm} or \cite[Section 6.1]{DS3}). A simple, yet important, observation of White
is that the map $\xii$ can be easily constructed so that the Dirichlet energy of $\xii\circ u$ (as clssical Euclidean map) coincides with that
of $u$ (as multivalued map) for any $u\in W^{1,2}$. 

\medskip

Sobolev maps were initially defined by Almgren using the map $\xii$. With this artifact a lot of the theory outlined in the previous paragraphs can be recovered from the usual (single valued) theory using the maps $\xii$ and $\ro$. Presently such maps could be avoided for essentially all the arguments. However a more complicated version of the map $\ro$ will play a rather important
role at a certain point later. As already mentioned, for $Q>1$ the space $\Iqs$ is not linear and
we cannot regularize $Q$-valued maps by convolution. Nonetheless we will need a way to smooth $W^{1,2}$ maps suitably
with a procedure which retains some of the basic estimates available for convolutions with a standard mollifier (in particular when computing the energy of the regularizations). A possible approach is to smooth the euclidean map $\xii \circ u$
and then ``project'' it back onto $\cQ$ using $\ro$. However, projecting back might be rather costly in terms of the energy since the Lipschitz constant
of $\ro$ is indeed rather far from $1$.

To bypass this problem,
we follow Almgren and prove the existence of 
``almost'' projections, denoted by $\ro^\star_\delta$, which are
$(1+\mu)$-Lipschitz in the 
$\delta$-neighborhood of $\xii (\Iqs)$. These maps cannot be the identity on $\cQ$, but they are at a uniform distance $\eta$ from it. Almgren's original proof is rather complicated. In \cite[Proposition 6.2]{DS3} we have proposed a different proof which uses heavily Kirszbraun's extension theorem and seems to yield a better estimate of $\mu$ and $\eta$ in terms of $\delta$ (in particular in the version of \cite{DS3} these are suitable positive powers of $\delta$). 

\begin{proposition}[Almost projection, {cf. \cite[Proposition 6.2]{DS3}}]\label{p:ro*}
For every $\bar{n}, Q\in \N\setminus \{0\}$ there are geometric constants $\delta_0, C>0$ with
the following property. 
For every $\delta\in ]0, \delta_0[$
there is $\ro^\star_\delta:\R^{N(Q,\bar n)}\to \cQ=\xii(\Iq(\R^{\bar n}))$ such that $|\ro^\star_\delta(P)-P|\leq C\,\delta^{8^{-\bar n Q}}$ for all $P \in \cQ$
and, for every $u\in W^{1,2}(\Omega,\R^{N})$, the following holds:
\begin{equation}\label{e:ro*1}
\int |D(\ro^\star_\delta\circ u)|^2\leq
\left(1+C\,\delta^{8^{-\bar{n}Q-1}}\right)\int_{\left\{\dist(u,\cQ)\leq \delta^{\bar nQ+1}\right\}} |Du|^2+
C\,\int_{\left\{\dist(u,\cQ)> \delta^{\bar nQ+1}\right\}} |Du|^2\, .
\end{equation}
\end{proposition}

\section{$\D$-minimizers and their regularity}

We are now ready to state the main results in the theory of $\D$-minimizing maps. In what follows, $\Om$ is always
assumed to be a bounded open set with a sufficiently regular
boundary.

\begin{theorem}[Existence for the Dirichlet Problem, cf. {\cite[Theorem 0.8]{DS1}}]
\label{t:existence}
Let $g\in W^{1,2} (\Om;\Iq)$. Then there exists a $\D$-minimizing
$f\in W^{1,2}(\Om;\Iq)$ such that $f|_{\partial\Omega}=g|_{\partial\Omega}$.
\end{theorem}

\begin{theorem}[H\"older regularity, cf. {\cite[Theorem 0.9]{DS1}}]\label{t:hoelder} 
There is a positive constant $\alpha=\alpha (m,Q)$ with the
following property.
If $f\in W^{1,2} (\Om;\Iq)$ is $\D$-minimizing,
then $f\in C^{0,\alpha} (\Omega')$ 
for every $\Omega'\subset\subset\Omega\subset \R^{m}$.
For two-dimensional domains, we have the explicit
constant $\alpha (2, Q)=1/Q$.
\end{theorem}

For the second regularity theorem we need the definition of
the singular set of $f$.

\begin{definition}[Regular and singular points, cf. {\cite[Definition 0.10]{DS1}}]\label{d:regular}
A $\D$-minimizing $f$ is regular at a point
$x\in \Omega$ if there exists a neighborhood $B$ of $x$
and $Q$ analytic functions $f_i:B\to \R^{n}$
such that
\begin{equation}\label{e:reg}
f (y)\;=\; \sum_i \a{f_i (y)}
\qquad \mbox{for every $y\in B$}
\end{equation}
and either $f_i (y)\neq f_j (y)$ for every $y\in B$, or $f_i\equiv f_j$.
The singular set $\sing (f)$ is the complement
of the set of regular points.
\end{definition}

\begin{theorem}[Estimate of the singular set, cf. {\cite[Theorem 0.11]{DS1}}]\label{t:structure}
Let $f$ be $\D$-minimizing. Then, the singular set $\sing (f)$
is relatively closed in $\Omega$. Moreover, if $m=2$,
then $\sing (f)$ is at most countable, and if
$m\geq 3$, then the Hausdorff dimension of $\sing (f)$ is at most
$m-2$. 
\end{theorem}

Note in particular the striking similarity between the estimate
of the size of the singular set in the case of multiple valued $\D$-minimizers
and in that of area minimizing currents. It will be discussed later that,
even in the case of $\D$-minimizers, there are singular solutions
(which are no better than H\"older continuous).

\medskip

Complete and self-contained proofs of these theorems can be found
in \cite{DS1}. The key tool for the estimate of the dimension of the singular set is the celebrated
frequency function (cp. with \cite[Section 3.4]{DS1}), which has been indeed used in a variety of different
contexts in the theory of unique continuation of elliptic partial differential
equations (see for instance the papers \cite{GL1}, \cite{GL2}). This is the central tool of our proofs as well. However,
our arguments manage much more efficiently the technical 
intricacies of the problem and some aspects of the theory
are developed in further details. For instance, we present in \cite[Section 3.1]{DS1}
the Euler-Lagrange conditions derived from first variations in a
rather general form. This is to our knowledge the first time that
these conditions appear somewhere in this generality.

\medskip

Largely following ideas of \cite{Chang} and of White, we improved the second
regularity theorem to the following optimal statement for planar maps.

\begin{theorem}[Improved $2$-dimensional estimate, cf. {\cite[Theorem 0.12]{DS1}}]
\label{t:finite}
Let $f$ be $\D$-minimizing and $m=2$. Then $\sing (f)$ is discrete.
\end{theorem}

This result was announced in \cite{Chang}. However, to our knowledge the
proof has never appeared before \cite{DS1}. 

A new addition to the regularity theory, which will have a lot of importance in the
subsequent discussions, is the following higher integrability result.

\begin{theorem}[Higher integrability of $\D$-minimizers, cf. {\cite[Theorem 5.1]{DS3}}]
\label{t:hig fct}
Let $\Omega'\subset\subset\Omega \subset\subset\R^{m}$ be open domains.
Then, there exist $p>2$ and $C>0$ such that
\begin{equation}\label{e:hig fct}
\norm{Du}{L^p(\Omega')}\leq C\,\norm{Du}{L^2(\Omega)}\quad
\text{$\forall\, \D$-minim. }\, u\in W^{1,2}(\Omega,\Iqs).
\end{equation}
\end{theorem}

We believe that several intricate 
arguments and complicated constructions
in the third chapter of Almgren's monograph can be reinterpreted as rather particular
cases of this key observation (see for instance \cite[Section 3.20]{Alm}).
Surprisingly, this higher integrability can be proved in a very
simple way by deriving
a suitable reverse H\"older inequality and using a (nowadays) very
standard version of the classical Gehring's Lemma. 

Theorem \ref{t:hig fct} has been stated and proved for the first time in \cite{DS3}. The relevant reverse
H\"older inequality has been derived using a comparison argument
and hence relying heavily on the minimality of the ${\rm Dir}$-minimizers.
A second proof, exploiting the Euler-Lagrange conditions to give
a Caccioppoli-type inequality, has been given in \cite{spadaro}. This last proof still
uses the regularity theory for $\D$-minimizers. However, this occurs only at one step:
one could hope to remove this restriction
and generalize the higher integrability
to ``critical'' points of the Dirichlet energy.

In \cite{spadaro} a yet different proof for the planar case is proposed, yielding
the optimal range of exponents $p$ for which \eqref{e:hig fct} holds.
The optimality of this result, as well as the optimality of 
Theorems \ref{t:hoelder} and \ref{t:finite}, is shown by another remarkable observation
of Almgren. Besides giving area minimizing currents, holomorphic varieties
are locally graphs of $\D$ minimizing $Q$-valued functions.
In \cite[Section 2.20]{Alm} Almgren proves this statement appealing to his powerful approximation
results for area minimizing current. 
However this is unnecessary and
a rather elementary proof can be found in \cite{spadaro}.

\subsection{H\"older continuity of $\D$-minimizers}\label{ss:Hoelder} The entire Section \ref{s:frequency} will be dedicated to the proof of Theorem
\ref{t:structure}, since it contains, in a simplified setting, several of the themes of the proof of Theorem \ref{t:cod>1}. 

In this paragraph we will instead discuss briefly the ideas behind the proof of Theorem \ref{t:hoelder}. We first assume that $u$ is a classical
(single valued) function and, for simplicity, that $m\geq 3$ (the case $m=2$ is somewhat special and can be handled in a
simpler way). Let $u\in W^{1,2} (B_1, \Iqs)$ be a minimizer and compare its energy to the energy of
the $0$-homogeneous extension $v$ of its trace on $\partial B_1 (x)$: we achieve the following crude inequality
\begin{equation}\label{e:comp-0}
\int_{B_1 (x)} |Du|^2 \leq \int_{B_1 (x)} |Dv|^2 \leq \frac{1}{m-2} \int_{\partial B_1 (x)} |Du|^2\, .
\end{equation}
A scaling-invariant version of the above inequality can be combined with Fubini's theorem to give the following differential inequality:
\begin{equation}\label{e:comp-1}
\int_{B_r(x)} |Du|^2 \leq \frac{r}{m-2} \frac{d}{dr} \int_{B_r(x)} |Du|^2\, ,
\end{equation}
which in turn gives the bound $\D (u, B_r(x)) \leq C r^{m-2}$. If we could improve the constant in \eqref{e:comp-1} to $\frac{1}{m-2+2\varepsilon}$,
the same reasoning would give the estimate $\D (u, B_r(x)) \leq C r^{m-2 + 2 \varepsilon}$ and Proposition \ref{p:Campanato Morrey} would imply the $\eps$-H\"older continuity of $u$. Now, for a single valued function $u$ the first inequality in \eqref{e:comp-0} is certainly strict, since $v$ does not satisfy the Euler-Lagrange conditions of a minimizer. It is not difficult to see that the very same conclusion can be drawn in the multivalued setting. The problem is to gain, in the factor of the right hand side of \eqref{e:comp-0}, a constant $\eps>0$ which is independent of the function (and, more importantly, of the central point $x$). 

\medskip

We can therefore focus on improving the constant in the right hand side of \eqref{e:comp-0} and without loss of generality we can assume $x=0$.
It is easy to see that we can assume, again without loss of generality, that the Dirichlet energy in $B_1 (0)$ is normalized to $1$.
When $u$ is single valued we can also assume that $u$ has average $0$ after subtracting a second suitable constant: the ``uniform gain'' from $\frac{1}{m-2}$ to $\frac{1}{m-2+2\varepsilon}$ in \eqref{e:comp-1} is then a simple consequence of the standard compactness of Sobolev maps
(via Poincar\'e inequality). 
However, the multivalued Poincar\'e inequality (namely Proposition \ref{p:poincare'}) does not give the same normalization property. The only well defined operation is the subtraction of the {\em same} value $p$ from {\em all} $Q$ sheets, namely given $u= \sum_i \a{u_i}$ we can set
\[
v(x) = \sum_{i=1}^Q \a{u_i (x) -p}\, .
\]
But the ``generalized mean'' of Proposition \ref{p:poincare'} is just a generic point in $\Iqs$. 
In particular we cannot expect compactness when we only control the Dirichlet energy: for a general Sobolev map some sheets might  be very far apart on a large subset and be very close on another, very small, subset. However, it can be shown that if the average separation between some sheets of a $\D$-minimizer $v$ is too large compared to its Dirichlet energy on a given domain, then $v$ must split into simpler functions in a smaller domain. This allows to prove that there is a uniform gain in the constant of the right hand side of the inequality of \eqref{e:comp-0}. The gain will depend upon $Q$, but this is not an artifact of the proof: as already pointed out, the H\"older exponent in Theorem \ref{t:hoelder} does deteriorate to $0$ as $Q\to\infty$. 

\medskip

The above discussion suggests that, although a control on the Dirichlet energy is not enough to ensure compactness for a general sequence of $Q$-valued $W^{1,2}$ functions, a suitable form of compactness can be recovered after modifying the sequence on a small
set. This principle plays a key role in our future discussions and has been formalized in \cite{DS3} in the following way.

\begin{definition}[Translating sheets, cf. {\cite[Definition 3.1]{DS3}}]\label{d:pacchetti}
Let $\Omega\subset\R^m$ be a bounded open set. A sequence of maps $\{h_k\}_{k\in \N}\subset W^{1,2}(\Omega,\Iqs)$ is called a sequence of {\em translating sheets} if there are:
\begin{itemize}
\item[(a)] integers $J\geq1$ and $Q_1, \ldots, Q_J\geq1$
satisfying $\sum_{j=1}^J Q_j = Q$,
\item[(b)] vectors $y^j_k\in \R^{n}$ (for  $j\in \{1, \ldots, J\}$ and $k\in \N$) with
\begin{equation}\label{e:separazione}
\lim_k |y^j_k - y^i_k|= + \infty\qquad \forall i\neq j,
\end{equation}
\item[(c)] and maps $\zeta^j\in W^{1,2}(\Omega, \I{Q_j})$ for $j\in \{1, \ldots, J\}$,
\end{itemize}
such that $h_k=\sum_{j=1}^J\llbracket \zeta_j - y^j_k\rrbracket$.
\end{definition}

\begin{proposition}[Concentration compactness, cf. {\cite[Proposition 3.3]{DS3}}]\label{p:cc}
Let $\Omega\subset \R^m$ be a Lipschitz bounded open set
and $(g_k)_{k\in\N} \subset W^{1,2}(\Omega,\Iq)$ a sequence of functions
with $\sup_k \int_\Omega |Dg_k |^2 < \infty$.
Then, there exist a subsequence (not relabeled) and a sequence of translating sheets $h_k$
such that $\norm{\cG(g_k,h_k)}{L^2}\to 0$ and the following inequalities
hold for every open $\Omega'\subset \Omega$ and any sequence of measurable sets
$J_k$ with $|J_k|\to 0$:
\begin{gather}
\liminf_{k\to+\infty} \left(\int_{\Omega'\setminus J_k} |Dg_k|^2 -\int_{\Omega'} |Dh_k|^2\right)\geq 0\label{e:cc2}\\
\limsup_{k\to+\infty} \int_\Omega
\left(|Dg_k|-|Dh_k|\right)^2 \leq \limsup_k \int_\Omega \left(|Dg_k|^2 - |D h_k|^2\right)\, .\label{e:cc3}
\end{gather}
\end{proposition}

\section{The frequency function}\label{s:frequency}

In this section we review the ideas behind the proof of Theorem \ref{t:structure}. As already mentioned the argument will serve as a prototype for the argument of Theorem \ref{t:cod>1} and for this reason we will be quite detailed. 

\subsection{First variations} There are two natural types of variations that can
be used to perturb $\D$-minimizing $Q$-valued
functions. The first ones, which we
call inner variations, are generated by right compositions
with diffeomorphisms of the domain.
The second, which we call
outer variations, correspond to ``left compositions''.
More precisely, let $f$ be a $\D$-minimizing
$Q$-valued map. 
\begin{itemize}
\item[(IV)] Given $\varphi\in C^\infty_c (\Omega,
\R^{m})$, for $\eps$ sufficiently small,
$x\mapsto \Phi_\eps (x)=x +\eps \varphi (x)$ is
a diffeomorphism of $\Omega$ which leaves $\partial
\Om$ fixed. Therefore,
\begin{equation}\label{e:IV}
0= \left.\frac{d}{d\eps}\right|_{\eps=0}
\int_\Omega |D (f \circ \Phi_\eps)|^2.
\end{equation}
\item[(OV)] Given $\psi\in C^\infty (\Omega\times
\R^{n}, \R^{n})$ such that $\supp(\psi)\subset \Om'\times \R^{n}$ for
some $\Om'\subset\subset\Om$, we set $\Psi_\eps (x)
=\sum_i \a{f_i (x) + \eps \psi (x, f_i (x))}$ and
derive
\begin{equation}\label{e:OV}
0= \left.\frac{d}{d\eps}\right|_{\eps=0}
\int_\Omega |D\, \Psi_\eps|^2.
\end{equation}
\end{itemize}
Using the multivalued chain rules we can turn the conditions \eqref{e:IV} and \eqref{e:OV} into the following identities:

\begin{proposition}[First variations. cf. {\cite[Proposition 3.1]{DS1}}]\label{p:first}
Let $f:\Omega\to\Iqs$ be $\D$-minimizing.
For every $\varphi\in C^\infty_c (\Omega, \R^{m})$, we have
\begin{equation}\label{e:IVbis}
2\int \sum_i \big\langle Df_i  : 
Df_i \cdot D\varphi \big\rangle\,
- \int |Df|^2 \,{\rm div}\, \varphi = 0.
\end{equation}
For every $\psi\in C^\infty (\Omega_x\times \R^{n}_u,
\R^{n})$ such that
$$
\supp(\psi)\subset \Om'\times \R^{n}\quad
\mbox{for some $\Om'\subset\subset\Om$,}
$$
and
\begin{equation}\label{e:hp}
|D_u\psi|\leq C<\infty\quad\textrm{and}\quad
|\psi|+|D_x\psi|\leq C\left(1+|u|\right)\,,
\end{equation}
we have
\begin{align}
&\int \sum_i \big\langle D f_i (x) : 
D_x \psi (x, f_i (x))\big\rangle\, dx
+ \int \sum_i \big\langle Df_i (x) : 
D_u \psi (x, f_i (x))\cdot
Df_i (x)\rangle\, d x= 0. \label{e:OVbis}
\end{align}
\end{proposition}

\subsection{The monotonicity of the frequency function}\label{ss:monot_freq}
\eqref{e:IVbis} and \eqref{e:hp} give particularly interesting identities when tested with functions which depend on $|x|$. In what follows, $\nu$ will always denote the outer unit normal
on the boundary $\partial B$ of a given ball. The following proposition gives the relevant identities when we test with the singular functions $\varphi (y) = {\bf 1}_{B_r (x)} (y) y$ and $\psi (x,u) = u {\bf 1}_{B_r (x)} (y)$ (the proof follows from a standard regularization of these $\varphi$ and $\psi$). 

\begin{proposition}[cf. {\cite[Proposition 3.1]{DS1}}]\label{p:import}
Let $x\in \Omega$ and $f:\Omega \to \Iqs$ be $\D$-minimizing.
Then, for a.e. $0< r<
\dist(x, \partial\Om)$, we have
\begin{equation}\label{e:cono}
(m-2) \int_{B_r (x)} |Df|^2 =
r \int_{\partial B_r (x)} |Df|^2 - 2\,
r \int_{\partial B_r (x)} \sum_i |\partial_\nu f_i|^2,
\end{equation}
\begin{equation}\label{e:perparti}
\int_{B_r (x)} |Df|^2 = \int_{\partial B_r (x)}
\sum_i \langle \partial_\nu f_i, f_i\rangle.
\end{equation}
\end{proposition}

We next introduce Almgren's frequency function and
state his celebrated monotonicity estimate, which is a straightforward consequence of the identities \eqref{e:cono} and \eqref{e:perparti}. Recall the notation $|f|$ for the function $\cG (f, Q\a{0})$. 

\begin{definition}[The frequency function, cf. {\cite[Definition 3.13]{DS1}}]\label{d:frequency}
Let $f$ be a $\D$-minimizing function, $x\in \Omega$
and $0<r<\dist (x, \partial \Omega)$. We define the
functions
\begin{equation}\label{e:frequency}
D_{x,f} (r)=\int_{B_r (x)} |Df|^2, \quad
H_{x,f} (r)=\int_{\partial B_r} |f|^2\quad\textrm{and}\quad
I_{x,f} (r)=\frac{r D_{x,f} (r)}{H_{x,f} (r)}. 
\end{equation}
$I_{x,f}$ is called the {\em frequency function}.
\end{definition}

When $x$ and $f$ are clear from the context, we will often
use the shorthand notation $D(r)$, $H(r)$ and $I(r)$.

\begin{theorem}[Monotonicity of the frequency function, cf. {\cite[Theorem 3.15]{DS1}}]\label{t:frequency}
Let $f$ be $\D$-minimizing and $x\in \Omega$.
Either there exists $\varrho > 0$ such that $f|_{B_\varrho (x)}\equiv 0$
or $I_{x,f} (r)$ is an absolutely continuous nondecreasing positive function
on $]0, \dist (x, \partial \Omega)[$. This function takes a constant value $\alpha$ if and only
if $f (y)$ is $\alpha$-homogeneous in $y-x$.
\end{theorem}

This monotonicity is the main ingredient in the proof of both Theorems \ref{t:cod>1} and \ref{t:structure}. An important observation, which was first made in \cite{DS5}, is that the frequency function can be thought as a ``singular limit'' of smoother objects, i.e. of regularized frequency functions, which are also monotone. This simple remark (which is not present in Almgren's monograph) gives an important advantage: the regularized frequency functions enjoy better continuity properties in terms of $f$. 

\begin{definition}[Regularized requency functions]\label{d:frequency_r}
Assume $\phi$ is a Lipschitz nonnegative nonincreasing compactly supported function on $[0,1[$ which is constant and positive in a neighborhood of $0$ and define
\begin{align*}
\bD_{0,f} (r) := &\int \phi\left(\frac{|x|}{r}
\right)\,|D f|^2(x)\, dx\\
\bH_{0,f} (r) := &- \int \phi'\left(\frac{|x|}{r}\right)\,\frac{|f|^2(x)}{|x|}\, dx\\
\bI_{0,f} (r) := &\frac{r\,\bD_{0,f} (r)}{\bH_{0,f} (r)}\, . 
\end{align*}
\end{definition}

\begin{theorem}\label{t:frequency2}
Let $f$ be $\D$-minimizing and $0\in \Omega$.
Either there exists $\varrho$ such that $f|_{B_\varrho (0)}\equiv 0$
or $\bI_{0,f} (r)$ is an absolutely continuous nondecreasing positive function
on $]0, \dist (x, \partial \Omega)[$. This function takes a constant value $\alpha$ if and only
if $f (x)$ is $\alpha$-homogeneous in $x$.
\end{theorem}

We do not have a reference for the latter theorem, which follows from a straightforward adaption of the arguments used in the proof of Theorem \ref{t:frequency}. A special case of Theorem \ref{t:frequency2}, namely for a special choice of the cut-off $\phi$, is hidden in the computations of
\cite[Theorem 3.2]{DS5} (cf. in particular \cite[Eq. (3.13)]{DS5}).

\subsection{The two fundamental consequences of the monotonicity formula}\label{ss:fundamental} Theorem \ref{t:frequency} has two crucial consequences, when ``blowing-up'' a given $\D$-minimizing function. More precisely, consider a $\D$-minimizing $f$ taking $Q>1$ values and a point $p$ in its domain. Without loss of generality we can assume that $p=0$. If the support of $f (0)$ contains two different points, then, by continuity, in a neighborhood $U$ of $0$ $f$ splits into two separate functions $u_1$ and $u_2$ which are both $W^{1,2}$ and continuous. It is simple to see that both must be minimizers of the Dirichlet energy in $U$. $0$ is then a good point, where we have reduced the complexity of the problem. For instance, if $Q$ were $2$ we would know that $u_1$ and $u_2$ are two classical (single valued) harmonic functions. The ``problematic points'' are then those $p$ where $f (p) = Q\a{q}$. 

\medskip

We can therefore assume that $f (0) = Q\a{q}$ for some $q\in \R^n$. Now, according to our definition of the singular set $\sing (f)$, we have two possibilities:
\begin{itemize}
\item[(a)] $f$ equals $Q$ copies of a classical harmonic function in a neighborhood of $0$;
\item[(b)] $0$ is a singular point for $f$.
\end{itemize}
In general, an interesting object to look at is the average of the sheets of $f= \sum_i \a{f_i}$, namely $\frac{1}{Q} \sum_i f_i$. For this average we fix the notation $\etaa\circ f$. It is not difficult to see that $\etaa\circ f$ is a classical harmonic function. Indeed, if we define
\[
\bar{f} := \sum_i \a{f_i -\etaa \circ f}\, ,
\]
it is immediate to see that $\D (f) = \D (\bar{f}) + Q \D (\etaa\circ f)$. In particular it is not difficult to conclude that $\bar{f}$ is also a $\D$-minimizer, cf. \cite[Lemma 3.23]{DS1}. Looking at the latter function we can thus restate the alternative as: either $\bar{f}\equiv Q \a{0}$ in a neighborhood of the origin, or $0$ is a singular point for $\bar{f}$ (and thus a singular point of $f$!).

\medskip

The discussion above leads to the consideration that, without loss of generality, we can assume $\etaa\circ f \equiv 0$. Assume further that the (more interesting!) alternative (b) above holds. Then $f$ does not vanish identically and therefore both $D_{0,f} (r)$ and $H_{0,f} (r)$ are positive for some $r$.
Using Theorem \ref{t:hoelder} it is not difficult to see that, under the assumption $f (0)= Q \a{0}$, we have a uniform bound of the form
\begin{equation}\label{e:poinc}
H_{0,f} (r)\leq C r D_{0,f} (r)\qquad\qquad \forall r\in \left]0, \textstyle{\frac{\dist (0, \partial \Omega)}{2}}\right[\, ,
\end{equation}
where the constant $C$ is independent of $f$.
The obvious consequence of Theorem \ref{t:frequency} is that there is also a reverse control
\begin{equation}\label{e:reverse}
r D_{0,f} (r) \leq \bar{C} H_{0,f} (r)
\end{equation}
although the latter constant $\bar{C}$ depends upon the point ($0$ in this case) and the function $f$. Indeed such constant approaches, for $r\downarrow 0$,  the limit
$I_0 (f) := \lim_{\rho\downarrow 0} I_{0,\rho} (f)$, which by \eqref{e:poinc} is bounded away from $0$ and by Theorem \ref{t:frequency} is finite: on the other hand we have no explicit (neither universal!) upper bound: we insist that $I_0 (f)$ depends upon $f$ and the particular point ($0$ in this case) where we are ``blowing-up''. 

\medskip

Consider now the rescaled functions $f_{0,r} (x) := f (r x)$ and their renormalized versions
\[
u_{0,r} (x) := \frac{f_{0,r}}{\D (f_{0,r}, B_1)^{\sfrac{1}{2}}}\, .
\]
In particular the energy of $u_{0,r}$ is $1$. However the $L^2$ norm of $|u_{0,r}|$ is also under control because of \eqref{e:poinc}.  We then have compactness for the family $\{u_{0,r}\}_r$. Fix a map $\bar{u}$ which is the limit of any subsequence $u_{0, r_k}$ with $r_k\downarrow 0$. It is not difficult to see that a sequence of minimizers with such uniform controls converge {\em strongly} in $W^{1,2}$ in any compact subset: namely the Dirichlet energy of the limiting function is the limit of the Dirichlet energy of the corresponding functions on any subdomain $\Omega$ which is compactly contained in $B_1 (0)$, cf. \cite[Proposition 3.20]{DS1}. However the minimizing property alone does not guarantee strong convergence on the {\em whole} domain $B_1 (0)$. 

To understand the latter statement, consider for instance the planar (single valued!) harmonic functions
\[
f_k (x_1, x_2) = {\rm Re}\, (x_1+ix_2)^k
\]
and their normalizations 
\[
u_k:= f_k/ \D (f_k, B_1 (0))\, .
\] 
It is very elementary to see that $u_k$ converges to $0$ in $B_1 (0)$: in fact most of the Dirichlet energy of $u_k$ lies in a thin layer around the boundary $\partial B_1 (0)$. For $k$ large the layer becomes thinner and thinner and all the energy is ``pushed'' towards the boundary $\partial B_1 (0)$. On the other hand it is easy to see that the ratio
\[
\frac{D_{0,u_k} (1)}{H_{0,u_k} (1)} = \frac{1}{H_{0, u_k} (1)}
\]
explodes, namely that the $L^2$ norm of $u_k$ on $\partial B_1 (0)$ converges to $0$.

\medskip

This highlights the first important consequence of the frequency function: the ``reverse Poincar\'e'' inequality \eqref{e:reverse} excludes that the energy of $u_{0,r}$ concentrates towards the boundary. Any limit $\bar{u}$ of a sequence $u_{0,r_k}$ must therefore have energy equal to $1$. Since Theorem \ref{t:hoelder} guarantees uniform convergence, we also conclude that $\bar{u} (0) = Q\a{0}$. Moreover, $\etaa\circ \bar u \equiv 0$ because $\etaa\circ u_{0,r}\equiv 0$. 

Thus $0$ must be a singular point of $u$ as well: the only way $\bar{u}$ could be regular around $0$ would be to take the value $Q\a{0}$ identically in a neighborhood of $0$. 
However notice that $I_{0, \bar{u}} (r) = I_{0,f} (0) =: \alpha$ for every $r$. But then Theorem \ref{t:frequency} implies that $\bar{u}$ is $\alpha$-homogeneous, and if $\bar{u}$ would vanish in a neighborhood of $0$, then it would vanish on the entire ball $B_1 (0)$, contradicting the fact that the Dirichlet energy of $u$ is indeed $1$.

The conclusion is that the singularity has {\em persisted} in the limit. Recalling that our main concern in proving Theorem \ref{t:cod>1} was the disappearance of singular points along sequences of converging currents, the reader will understand why the monotonicity of the
frequency function is such an exciting discovery. It must also be noticed that the monotonicity of the frequency function was unknown even for classical single valued harmonic functions before \cite{Alm}: the shear observation that Almgren was able to discover a new fundamental fact for classical harmonic functions around 1970 gives in my opinion the true measure of his genius.

\medskip

The second fundamental consequence of the monotonicity of the frequency function is that $I_{0, \bar{u}} (r)$ is indeed constant in $r$ and equals $\alpha := I_{0,f} (0)$, which, as already noticed, gives that We have $\bar{u}$ is $\alpha$-homogeneous. In particular when the domain is $2$-dimensional, it is not difficult to classify all $\alpha$-homogeneous $\D$-minimizers and to show that their only singularity is at the origin, cf. \cite[Proposition 5.1]{DS1}. 

\medskip

The careful reader will recognize the formal analogy with the two ingredients of Federer's reduction argument illustrated in Section \ref{ss:Fedred}: pretty much the same reasoning gives the proof of Theorem \ref{t:structure}. There is however one important difference: for a $Q$-valued minimizer $f$ on a $2$-dimensional connected domain we do not conclude the discreteness of $\sing (f)$, but rather the weaker statement that 
\begin{itemize}
\item either  ``multiplicity $Q$ points'' of $f$ are isolated;
\item or $f$ collapses to $Q\a{\etaa\circ f}$.
\end{itemize}
Only in the case $Q=2$ the statement above is equivalent to discreteness of the singular set of $f$. When, for instance, $Q=3$, we have not ruled out that singular points with ``$2$-sheeted branching'' could converge towards a singular point with a ``3-sheeted'' branching. 

Thus, the argument sketched above gives, in the $2$-dimensional case, that $\sing (f)$ is countable, but it does not imply its discreteness. The proof of Theorem \ref{t:finite} needs much more work and in particular it passes through the important additional conclusion that the tangent functions $\bar{u}$ analyzed above are unique, namely the renormalized blown-up functions $u_{0,r}$ have a unique limit as $r\downarrow 0$, cf. \cite[Theorem 5.3]{DS1}. At present this uniqueness is an open problem when the dimension of the domain is higher than $2$.

\section{Approximation with multiple valued graphs}

Following the intuition that a ``sufficiently flat'' area minimizing current is  close to the graph of a $\D$-minimizing multivalued function, we wish now to use the theory above to infer some interesting informations upon area minimizing currents in a region where they are rather flat, i.e. the tangent planes are almost parallel to a given one (at least in an average sense). For this reason we introduce the notion of cylindrical excess. We will denote by $\bC_r (p)$ the cylinder $B_r (x)\times \R^n$ when $p= (x,y)\in \R^m\times \R^n$. In fact in the future we wish to consider cylinders with bases parallel to different $m$-dimensional planes: having fixed an $m$-dimensional plane $\pi$, we set $B_r (p, \pi):= \bB_r (p) \cap (p+\pi)$ and $\bC_r (p, \pi) = B_r (p, \pi) + \pi^\perp$. The notation $\pi_0$ will be reserved for the ``horizontal plane'' $\R^m\times \{0\}$ and we will use $\p_{\pi}$ and $\p^\perp_\pi$ for the orthogonal projections onto $\pi$ and $\pi^\perp$. 

\begin{definition}[Cylindrical excess]\label{d:c_excess}
Given an integer rectifiable $m$-dimensional current $T$ in $\mathbb R^{m+n}$ with finite mass and compact support and $m$-planes $\pi, \pi'$, we define the {\em excess} of $T$ in the cylinder $\bC_r (x, \pi)$ compared to $\pi'$ as 
\begin{align}
\bE (T, \bC_r (x, \pi), \pi') &:= \left(2\omega_m\,r^m\right)^{-1} \int_{\bC_r (x, \pi)} |\vec T - \vec \pi'|^2 \, d\|T\|\, .
\end{align}
If $\pi= \pi'$, then we write $\bE (T, \bC_r (x, \pi))$.

The {\it height function} in a set $A \subset \R^{m+m}$ with respect to an $m$-dimensional plane $\pi$ is
\[
\bh(T,A,\pi) := \sup_{x,y\,\in\,\supp(T)\,\cap\, A} |\p_{\pi^\perp}(x)-\p_{\pi^\perp}(y)|\, .
\]
\end{definition}

\subsection{Multivalued push-forwards, graphs, the area formula and the Taylor expansion of the mass} One first technical detail that we have to tackle concerns the currents which are naturally induced by multivalued maps. Assume therefore to have fixed a Lipschitz map $F: \R^m \supset \Omega \to \Iq (\R^{m+n})$ on a bounded open set $\Omega$. Consider the regions $M_i$ and the functions $f_i^j$ of Lemma \ref{l:chop}. Through them we can define the ``multivalued'' pushforward 
\[
\bT_F := \sum_i (f_i^j))_\sharp \a{M_i}\, .
\]
In a similar fashion we can define multivalued pushforwards when the domain $\Omega$ is a Riemannian manifold with finite volume. Moreover the current naturally carried by the graph of a multivalued function $u$ can be defined using the pushforward through the map $x\mapsto \sum_i \a{(x, u_i (x))}$. The corresponding current will be denoted by $\bG_u$, whereas for the set-theoretic objects we will use the notation $\im (F)$ and $\gr (u)$. 

The currents introduced above are well-defined (namely they are independent of the decomposition chosen in Lemma \ref{l:chop}) and in fact the assumption of boundedness of $\Omega$ and finiteness of the volume of $M$ can be removed if $F$ satisfies a suitable ``properness'' assumption, cf. \cite[Definition 1.2 \& Definition 1.3]{DS2}. Moreover, the usual formulas and conclusions valid in the classical-valued setting holds in the multivalued case as well. We record here some important conclusions.

\begin{lemma}[Bilipschitz invariance, {cf. \cite[Lemma 1.8]{DS2}}]\label{l:biLipschitz_inv} Let $F: \Sigma\to \Iqs$ 
be a Lipschitz and proper map, $\Phi: \Sigma'\to \Sigma$ a
bilipschitz homeomorphism and $G:= F\circ \Phi$. Then, $\mathbf{T}_F = \mathbf{T}_G$.
\end{lemma}

\begin{lemma}[$Q$-valued area formula, {cf. \cite[Lemma 1.9]{DS2}}]\label{l:area}
Let $\Sigma \subset\R^N$ be a Lipschitz oriented submanifold, $M\subset \Sigma$ a
measurable subset and $F:M\to\Iqs$ a proper Lipschitz map.
For any bounded Borel function $h: \R^n \to [0, \infty[$, we have
\begin{gather}\label{e:mass}
\int h (p)\,  d \|\mathbf{T}_F\| (p) \leq \int_M \sum_j h (F^j (x))\, \bJ F^j (x) \, d\cH^m(x)  \, ,
\end{gather}
where
\[
\bJ F^j (x) = \left|DF^j (x)_\sharp \vec{e}\,\right| = \sqrt{\det ((DF^j(x))^T \cdot DF^j (x))}
\]
Equality holds in \eqref{e:mass} if there is a set $M'\subset M$ 
of full measure for which 
\begin{equation}\label{e:no_cancellation}
\langle DF^j (x)_\sharp \vec{e} (x), DF^i (y)_\sharp \vec{e} (y) \rangle \geq 0 \qquad 
\forall x,y\in M' \; \mbox{and}\; i, j \;\mbox{with}\; F^i (x) = F^j (y)\, .
\end{equation}
If \eqref{e:no_cancellation} holds the formula is valid also for
bounded {\em real}-valued Borel $h$ with compact support.  
\end{lemma}

\begin{corollary}[Area formula for $Q$-graphs, {cf. \cite[Corollary 1.11]{DS2}}]\label{c:massa_grafico}
Let $\Sigma = \R^m$, $M\subset \R^m$ and $f: \Sigma \to \Iqs$ be a proper Lipschitz map. Then,
for any bounded compactly supported Borel $h: \R^{m+n}\to \R$, we have
\begin{equation}\label{e:massa_grafico}
\int h(p)\, d\|\bG_f\| (p) = \int_M \sum_i h (x, f_i (x)) 
\Big(1 + \sum_{k=1}^m \sum_{A \in M^k (Df^i)} (\det A)^2\Big)^{\frac{1}{2}}\,dx .
\end{equation}
\end{corollary}

\begin{theorem}[Boundary of the push-forward, {cf. \cite[Theorem 2.1]{DS2}}]\label{t:commute}
Let $\Sigma$ be a Lipschitz submanifold of $\R^N$ with Lipschitz boundary,
$F:\Sigma \to \Iqs$ a proper Lipschitz function and $f= F\vert_{\de\Sigma}$.
Then, $\partial \mathbf{T}_F = \mathbf{T}_f$.
\end{theorem} 

One crucial point in our discussions is the Taylor expansion of the mass of a multivalued graph. 

\begin{corollary}[Expansion of $\mass (\bG_f)$, {cf. \cite[Corollary 3.3]{DS2}}] \label{c:taylor_area}
Assume $\Omega\subset \R^m$ is an open set with bounded measure and $f:\Omega
\to \Iqs$ a Lipschitz map with $\Lip (f)\leq \bar{c}$. Then,
\begin{equation}\label{e:taylor_grafico}
\mass (\mathbf{G}_f) = Q |\Omega| + \frac{1}{2} \int_\Omega |Df|^2 + \int_\Omega \sum_i \bar{R}_4 (Df_i)\, ,
\end{equation}
where $\bar{R}_4\in C^1$ satisfies $|\bar{R}_4 (D)|= |D|^3\bar{L} (D)$ for $\bar{L}$ with $\Lip (\bar{L})\leq C$ and
$\bar L(0) = 0$.
\end{corollary}

\subsection{The main approximation theorem}
If $T$ is an integral current without boundary in $\bC_r (0)$, a Lipschitz $u: B_r (0)\to \Iqs$ 
is an efficient approximation if $\mass (T- \bG_u)$ is small (compared to $r^m$). Since $\bG_u$ is, in a ``loose'' sense a $Q$-fold cover of $B_r (0)$, namely $(\p_{\pi_0})_\sharp \bG_u = Q \a{B_r (0)}$, this condition must hold for a well-approximated current $T$ as well. We are now ready to summarize the main assumptions under which we wish to find a good Lipschitz multivalued approximation, which is achieved in Theorem \ref{t:app_main}. To simplify our notation $\p_{\pi_0}$ and $\p^\perp_{\pi_0}$ will be denoted by $\p$ and $\p^\perp$. 

\begin{ipotesi}\label{ipotesi_base_app}
$\Sigma\subset\R^{m+n}$ is a $C^2$
submanifold of dimension $m + \bar n = m + n - l$, which is the graph of an entire
function $\Psi: \R^{m+\bar n}\to \R^l$ and satisfies the bounds
\begin{equation}\label{e:Sigma}
\|D \Psi\|_0 \leq c_0 \quad \mbox{and}\quad \bA := \|A_\Sigma\|_0
%\| D^2\bPsi \|_{L^\infty}
\leq c_0,
\end{equation}
where $c_0$ is a positive (small) dimensional constant.
$T$ is an integral current of dimension $m$ with bounded support contained in $\Sigma$ and which, for some open cylinder $\bC_{4r} (x)$ 
(with $r\leq 1$)
and some positive integer $Q$, satisfies
\begin{equation}\label{e:(H)}
\p_\sharp T\res \bC_{4r} (x) = Q\a{B_{4r} (x)}\quad\mbox{and}\quad
\de T \res \bC_{4r} (x) =0\, .
\end{equation}
\end{ipotesi}

\begin{theorem}[Strong approximation, cf. {\cite[Theorem 1.4]{DS3}}]\label{t:app_main}
There exist constants $C, \gamma_1,\eps_1>0$ (depending on $m,n,\bar n,Q$)
with the following property. Assume that $T$ is area minimizing, satisfies Assumption \ref{ipotesi_base_app} in the cylinder $\bC_{4\, r} (x)$
and 
$E =\bE(T,\bC_{4\,r} (x)) < \eps_1$.
Then, there is a map $f: B_r (x) \to \Iqs$, with $\gr (f) \subset \Sigma$,
and a closed set $K\subset B_r (x)$ such that
\begin{gather}\label{e:main(i)}
\Lip (f) \leq C E^{\gamma_1} + C \bA r, \\
\label{e:main(ii)}
\bG_f\res (K\times \R^n)=T\res (K\times\R^{n})\quad\mbox{and}\quad
|B_r (x)\setminus K| \leq
 C \, E^{\gamma_1} \left(E+ r^2\,\bA^2\right)\, r^m,\\
\label{e:main(iii)}
\left| \|T\| (\bC_{{\sigma\,r}} (x)) - Q \,\omega_m\,({\sigma\,r})^m -
{\textstyle{\frac{1}{2}}} \int_{B_{{\sigma\,r}} (x)} |Df|^2\right| \leq
 C \, E^{\gamma_1} \left(E+ r^2\,\bA^2\right)\, r^m \quad \forall\,{0<\sigma \leq 1}.
\end{gather}
If in addition $\bh (T, \bC_{4r} (x), \pi_0) \leq r$, then
\begin{equation}\label{e:L-infty_est}
{\rm osc}\, (f) \leq C \bh (T, \bC_{4r} (x), \pi_0) + C (E^{\sfrac{1}{2}} + r\,\bA)\, r\, ,
\end{equation}
where ${\rm osc}\, (f) := \sup \{|p-q| : p\in \supp (f(x)), q\in \supp (f(y)), x,y\in B_r (x)\}$. 
\end{theorem}

We note that the theorem is scaling invariant and thus it suffices to prove it in the case $r=1$. Moreover, for simplicity we will mostly ignore $\Sigma$ and often assume that $T$ is area minimizing in the whole euclidean space: this will be of great help in illustrating the main ideas behind the proof, avoiding some technicalities.

An elementary computation shows that, under Assumption \ref{ipotesi_base_app}, 
\[
\bE(T,\bC_r (x))= \frac{\|T\| (\bC_r (x))}{\omega_m r^m} - Q\, .
\]
It is then natural to introduce the following ``excess measure'':

\begin{definition}[Excess measure, cf. {\cite[Definition 1.2]{DS3}}]\label{d:excess_meas}
For a current $T$ as in Assumption \ref{ipotesi_base_app} we define
the \textit{excess measure} $\e_T$ and its {\em density} $\bd_T$:
\begin{gather*}
\e_T (A) := \|T\| (A\times\R^{n})  - Q\,|A| \qquad \text{for every Borel }\;A\subset B_r (x),\\%\label{e:ex_measure}\\
\bd_T(y) := \limsup_{s\to 0} \frac{\e_T (B_s (y))}{\omega_m\,s^m}= \limsup_{s\to 0} \bE (T,\bC_s (y))\, .%\label{e:ex_density}
\end{gather*}
\end{definition}

\subsection{$BV$ estimate for slices and first approximation}\label{ss:BVslice} It is rather clear that the smallness of the cylindrical excess prevents the tangent plane to $T$ at $p$ to have negative intersection with $\{\p (p)\}\times \R^n$ at most points $p$ in $\supp (T)$. In fact this is a simple measure-theoretic fact: even without assuming that $T$ is area minimizing, it remains true that, under Assumption \ref{ipotesi_base_app}, most slices $\langle T, \p, y\rangle$ will be elements of $\Iqs$. The points $y$ which violate the latter property will form a set of small measure. 

It is instructive to see what happens if $Q=1$ and $T$ is assumed to be a-priori the graph of a classical map $v$, assuming a Lipschitz bound like $\Lip (v)\leq 1$. 
The cylindrical excess $E$ is then comparable, up to constants, to the $L^2$ norm of $Dv$. It is a classical statement for a (single valued) Sobolev map that a Lipschitz control holds on the restriction of the map on a fairly large closed set, cf. for instance \cite[Section 6.6.3]{EG}. Indeed a way to identify a good set on which such Lipschitz bound holds is to look at those points where the Hardy-Littlewood maximal function of $|Dv|$ is suitably small. Under our idealized situation, $|Dv|^2$ is comparable to the excess density $\bd_T$ introduced above. This motivates the introduction of a maximal function in our setting

\begin{definition}[Maximal function of the excess measure, cf. {\cite[Definition 2.1]{DS3}}]\label{d:maximal}
Given a current $T$ as in Assumption \ref{ipotesi_base_app} we introduce the ``non-centered'' maximal function of $\e_T$:
\begin{equation*}
\bmax\be_T (y) := \sup_{y \in B_s (w)\subset B_{4r} (x)} \frac{\be_T (B_s (w))}{\omega_m\, s^m} = \sup_{y\in B_s (w)\subset B_{4r} (x)}
\bE (T,\bC_s (w)) .
\end{equation*}
\end{definition}

Going on with our Sobolev space analogy, if we denote by $E$ the square of the $L^2$ norm of $|Dv|$ (normalized by $r^m$) and we let $K$ be the set where the maximal function of $|Dv|^2$ lies below the threshold $E^{2\gamma_1}$, then the restriction of $v$ to $K$ will have Lipschitz constant $E^{\gamma_1}$ and the size of the complement of $K$ can be estimated with $r^m E^{1-2\gamma_1}$. Of course we can then extend $v|_K$ outside $K$ to a Lipschitz function with essentially the same Lipschitz bound. Neglecting the effect of $\Psi$, it is then clear that, only relying on Assumption \ref{ipotesi_base_app} we can hope for estimate \eqref{e:main(i)} if we replace the superlinear $E^{1+\gamma_1}$ in \eqref{e:main(ii)} and \eqref{e:main(iii)} with, respectively, $E^{1-2\gamma_1}$ and $E$. 

\medskip

This heuristic discussion can be in fact made rigorous in a very direct way relying on some recent developments in geometric measure theory. Regarding the slicing map $\langle T, \p, \cdot\rangle$ as a map taking values into the space of $0$-dimensional currents (endowed with a suitable metric) and using the formalism introduced by Ambrosio in \cite{ambrosio} for BV maps with metric targets, Jerrard and Soner have given in \cite{JS} a rather elementary way to prove that such map is a function of bounded variation, with norm which can be controlled with the mass of $T$ and the mass of its boundary. Ambrosio and Kirchheim used then this idea in \cite{AK} to develop part of their general theory of metric currents and give a rather efficient  and general approach to the Federer-Fleming compactness theorem.  The resulting computations must be suitably adjusted to our setting. However the theory allows a quite direct proof of the following 

\begin{proposition}[Lipschitz approximation, cf. {\cite[Proposition 2.2]{DS3}}]\label{p:max}
There exists a constant $C>0$ with the following property.
Let $T$ and $\Psi$ be as in Assumption \ref{ipotesi_base_app} in the cylinder $\bC_{4s} (x)$.
Set $E=\bE(T,\bC_{4s}(x))$, let $0<\delta_{11}<1$ be such that $16^m E < \delta_{11}$,
and define 
\[
K := \big\{\bmax\be_T<\delta_{11}\big\}\cap B_{3s}(x)\, .
\]
Then, there is $u\in \Lip (B_{3s}(x), \Iqs)$ such that
$\gr (u) \subset \Sigma$ and
\[
\Lip (u)\leq C\,\big(\delta_{11}^{\sfrac{1}{2}} + {\|D\Psi\|_0}\big),
\qquad {\rm osc}\, (u) \leq C {\bh (T, \bC_{4s} (x), \pi_0)} + C s \|D \Psi\|_0\, ,
\]
\[
\bG_u \res (K\times \R^{n})= T\res (K\times \R^{n}),
\]
\begin{equation}\label{e:max1}
|B_{r}(x)\setminus K|\leq \frac{10^m}{\delta_{11}}\,\e_T \Big(\{\bmax\be_T > 2^{-m} \delta_{11} \}\cap B_{r+r_0s}(x)\Big) \quad\forall\;
r\leq 3\,s ,
\end{equation}
where $r_0 = 16 \sqrt[m] {E/\delta_{11}} < 1$.
\end{proposition}

From Proposition \ref{p:max} one derives immediately a version of Theorem \ref{t:app_main} where the bound \eqref{e:main(i)} is correct, whereas in the bound \eqref{e:main(ii)} the factor $E^{1+\gamma_1}$ must be replaced by $E^{1-2\gamma_1}$ and in the bound \eqref{e:main(iii)} $E^{1+\gamma_1}$ must be replaced by $E$. In the rest of this section we will discuss why the area minimizing assumption, which so far we have not yet used, allows to improve the bounds to achieve Theorem \ref{t:app_main}. 

\subsection{Superlinear gain} Going back to our heuristic idea, in which $T$ is replaced by the graph of a single valued function $v$ and the excess by the square of the  $L^2$ norm, the ``maximal function truncation'' described in the previous paragraph would deliver the desired superlinear estimates  if we knew that the $L^{2+\beta}$ norm of $Dv$ were controlled by $E^{\sfrac12}$, namely the $L^2$ norm of $Dv$, for some $\beta>0$. This amounts to a reverse H\"older inequality of the form
\begin{equation}\label{e:reverse-0}
\|Dv\|_{L^p} \leq C \|Dv\|_{L^2} \qquad \mbox{for some $p>2$}\, .
\end{equation}
In our setting one possible translation would be: the excess measure $\be_T$ is absolutely continuous and its density  $\bd_T$ enjoys the estimate 
\begin{equation}\label{e:reverse-1}
\|\bd_T \|_{L^{1+\varepsilon} (B_{2r} (x))} \leq C \be_T (B_{4r} (x)) \leq C r^m E\, .
\end{equation}
This is certainly not correct under the only Assumption \ref{ipotesi_base_app}: it is clear that in order to hope for such a bound we need to use the hypothesis that $T$ is area minimizing. We do not know whether \eqref{e:reverse-1} is correct under the additional assumption that $T$ be area minimizing: even if it is, we expect that its proof is rather difficult, see the discussion below. However, the cornerstone to our approach to Theorem \ref{t:app_main} is that the following slightly weaker form of \eqref{e:reverse-1} is correct and can be achieved with a moderate effort. 

\begin{theorem}[Gradient $L^p$ estimate, cf. {\cite[Theorem 1.3]{DS3}}]\label{t:higher1}
There exist constants $p_1 >1$ and $C, \eps_{10}>0$ (depending on $m,n,\bar n, Q$)
with the following property. Let $T$ be as in Assumption \ref{ipotesi_base_app} in the cylinder $\bC_4$.
If $T$ is area minimizing and $E=\bE (T,\bC_4)< \eps_{10}$, then
\begin{equation}\label{e:higher1}
\int_{\{\bd\leq1\}\cap B_2} \bd_T^{p_1} \leq C\, E^{p_1-1} \left(E + \bA^{2}\right).
\end{equation}
\end{theorem}

From Theorem \ref{t:higher1} and Proposition \ref{p:max} we cannot conclude directly Theorem \ref{t:app_main} because we lack control on the set where $\bd_T$ is rather high (and on the singular part of the measure $\be_T$!). We would rather need an estimate for the regions where the tangent to $T$ has high slope, compared to $\pi_0$. Theorem \ref{t:higher1} can be indeed used to prove something of that kind:

\begin{theorem}[Almgren's strong excess estimate, cf. {\cite[Theorem 6.1]{DS3}}]\label{t:higher}
There are constants $\eps_{11},\gamma_{11}, C> 0$ (depending on $m,n,\bar n,Q$)
with the following property.
Assume $T$ satisfies Assumption \ref{ipotesi_base_app} in $\bC_4$ and is area minimizing.
If $E =\bE(T,\bC_4) < \eps_{11}$, then
\begin{equation}\label{e:higher2}
\e_T (A) \leq C\, \big(E^{\gamma_{11}} + |A|^{\gamma_{11}}\big) \left(E+\bA^2\right)
\quad \text{for every Borel }\; A\subset B_{{\frac98}}.
\end{equation}
\end{theorem}

Actually, in the case of a classical single valued map Theorem \ref{t:higher} could be concluded directly by comparing the mass of the current $T$ with that of a suitable convolution of the approximating Lipschitz map. The effect of the convolution is to smear high gradients and show that they are energetically not favorable. A surrogate of this regularization process can be achieved using the map $\ro^\star$ of Proposition \ref{p:ro*}: to regularize a map $u$ we can first regularize $\xii \circ u$ by convolution and then gain a multivalued map composing the regularization with $\xii^{-1} \circ \ro^\star$. We can compare the Dirichlet energy of this map
to that of the regularization of $\xii\circ u$. Note however that if we want a small error we have to choose the parameter $\delta$ in Proposition \ref{p:ro*} small and yet hope that the second integrand in the right hand side of \eqref{e:ro*1} remains small. Theorem \ref{t:higher1} allows to control that dangerous term.

\subsection{Higher integrability and harmonicity} Going back to our analogy, we know that if $T$ were the graph of a function, the minimality assumption and the smallness of the excess should imply that $v$ is close to an harmonic function. Of course for single valued harmonic functions then the reverse H\"older inequality \eqref{e:reverse-0} is true for any exponent $p>2$. On the other hand we already discussed that, for a suitable choice of $p$, the same reverse H\"older inequality does hold in the multivalued setting as well, cf. Theorem \ref{t:higher}. This suggests that in order to prove Theorem \ref{t:higher1} we could first show that the Lipschitz map of Proposition \ref{p:max} is almost $\D$-minimizing. Looking at Theorem \ref{t:app_main} it is rather intuitive that the ``almost $\D$-minimality'' of $f$ should correspond to have a $o (E)$ in place of $E^{1+\gamma_1}$ in \eqref{e:main(iii)}, where $o (E)$ is any function of $E$ which vanishes faster than $E$ at $0$. Now, using an energetic comparison, such a gain would correspond to show that 
\[
\int_{B_r (x)\setminus K} |Df|^2 = r^m o (E)\, .
\]
If this were not true we could run a contradiction argument over a sequence of currents $T_k$ with vanishing excess $E_k$ and look at the normalized  approximations $u_k:= f_k/ E_k^{\sfrac{1}{2}}$. We could also rescale the corresponding balls to have radius $1$ and center $0$. The $m$-dimensional volume of the corresponding bad sets $B_1\setminus K_k $ is converging to $0$ and in spite of that 
\[
\liminf_k \int_{B_1\setminus K_k} |Du_k|^2 \geq \eta 
\] 
for some positive $\eta$. If we assume that $u_k$ is converging in $L^2$ to some $u$, the Dirichlet energy of $u$ would then satisfy
\[
\liminf_k \int_{B_1} |Du_k|^2 \geq \int_{B_1} |Du|^2 + \eta\, .
\]
But then the graph of $E_k^{\sfrac{1}{2}} u$ must have less mass than $T_k$ and we could hope to modify it and gain a comparison current which would contradict the minimality of $T_k$, at least for $k$ sufficiently large. 

Recalling Section \ref{ss:Hoelder} there is a delicate point to address, namely that for multiple valued functions a uniform control on the Dirichlet energy of a sequence does not imply directly compactness, since the separation between sheets could explode along the sequence. This is achieved with Proposition \ref{p:cc}. Incidentally the analysis above shows also that the approximation of Theorem \ref{t:app_main} is close to a $\D$-minimizer, which we record in the following theorem. 

\begin{theorem}[Harmonic approximation]\label{t:harmonic_final}
Let $\gamma_1$ be the constant of Theorem \ref{t:app_main}.
Then, for every $\bar{\eta}, \bar{\delta}>0$, there is a positive constant $\bar{\eps}_1$ with the following property. 
Assume that $T$ is as in Theorem \ref{t:app_main},
$E := \bE(T,\bC_{4\,r}(x)) < \bar{\eps}_1$
and $r \bA \leq E^{\sfrac{1}{4}+\bar\delta}$. If $f$ is the map in Theorem~\ref{t:app_main} and we fix suitable coordinates, then there exists a $\D$-minimizing function $u: B_r (x) \to \Iq (\R^{\bar{n}})$ such that $w := (u, \Psi (y, u))$ satisfies 
\begin{equation}\label{e:harmonic_final}
r^{-2} \int_{B_r (x)} \cG (f, w)^2 + \int_{B_r(x)} \left(|Df|-|Dw|\right)^2 + \int_{B_r(x)} |D(\etaa\circ f)-
D(\etaa \circ w)|^2\leq \bar{\eta}\, E \, r^m\, .
\end{equation}
\end{theorem}

\section{A first attempt to prove Theorem \ref{t:cod>1}}\label{s:tentativo}

In this section we summarize what we have achieved so far and propose a first strategy to show Theorem \ref{t:cod>1}. After resolving the first important issues, we will have to face a major obstacle: more than half of Almgren's monograph is in fact dedicated to overcome this point and even in the proof given by \cite{DS1,DS2,DS3,DS4,DS5} the same obstacle is responsible for roughly one quarter of the combined length, namely paper \cite{DS4}.

\medskip

 The strategy to prove Theorem \ref{t:cod>1} starts similarly to the Federer's reduction argument. Assume that there is an area minimizing current $T$ of dimension $m\geq 2$, in a sufficiently smooth Riemannian manifold $\Sigma$, which has a large singular set $\sing (T)$: more precisely we assume that, for some $\alpha>0$, $\cH^{m-2+\alpha} (\sing (T))>0$.

From Theorem \ref{t:strat} and Lemma \ref{l:strat} we conclude immediately that at $\cH^{m-2+\alpha}$-a.e. $p\in \sing (T)$ there is one flat tangent plane and the multiplicity is integral. Let us introduce the notation $\rD_Q (T)$ for those points in $\supp (T)\setminus \supp (\partial T)$ where the density of $T$ is the positive integer $Q$. Similarly, we set $\sing_Q (T) := \rD_Q (T)\cap \sing (T)$. We then know that $\sing_1 (T)$ is empty. Indeed the assumptions (a), (b) and (c) in Theorem \ref{t:DG} follow from the monotonicity formula when $\rho$ is sufficiently small. The second assumption in (d) is also fulfilled: since we can assume that the second fundamental form of $\Sigma$ is bounded, for $\rho$ sufficiently small we obviously have $\rho\bA < \varepsilon$. It would remain to prove that the excess with respect to some plane is suitably small at a sufficiently small scale. This is however not difficult since all tangent cones at a point $p$ with $\Theta (T, p)=1$ must be necessarily flat: it can be shown that the only area minimizing $m$-dimensional cones $S$ with $\|S\| (\bB_1 (0)) = \omega_m$ are $m$-dimensional planes counted with multiplicity $1$. 

\medskip

We stop for a moment to observe the following interesting consequence of the above discussion. Let $p$ be a point in $\supp (T)$ where the multiplicity is $Q$ and assume that the surrounding points in $\supp (T)$ have the same multiplicity at a sufficiently small scale, say in $\bB_\rho (p)$. Then $S:= T/Q$ is a well defined integer-rectifiable area minimizing current  in $\bB_\rho (p)$ and moreover $p\in \rD_1 (S)$. Thus $S$ is regular in a neighborhood of $p$. We summarize the outcome of the latter discussion in the following

\begin{corollary}\label{c:densita}
If $S$ is an area minimizing cone with $\Theta (S, 0) =1$, then $S$ is a flat plane with multiplicity $1$. 

Let $T$ be an area minimizing current in a $C^2$ Riemannian manifold $\Sigma$. If $p\in \rD_Q (T)$ and there is a neighborhood $U$ of $p$ where the density is $Q$ at $\|T\|$-a.e. point, then $p$ is a regular point.
\end{corollary}

We next recover our discussion and look at the current $T$ which should contradict Theorem \ref{t:cod>1}. We infer from Corollary \ref{c:densita} that there must be an integer $Q>1$ such that $\sing_Q (T)$ has positive $\cH^{m-2+\alpha}$-measure. Now, recalling the approach of Federer's reduction argument, we know that for $\cH^{m-2+\alpha}_\infty$-a.e. $p\in \sing_Q (T)$ we have
\begin{equation}\label{e:successione_1}
\liminf_{r\downarrow 0} \frac{\cH^{m-2+\alpha}_\infty (\sing (T) \cap \bB_r (p))}{r^{m-2+\alpha}} > 0\, .
\end{equation}

Moreover, by Theorem \ref{t:strat} we can assume that at least one tangent cone at $p$ is flat. We thus have a sequence of rescalings $T_{p, s_k}$ which are converging to a flat plane and a sequence of rescalings $T_{p, r_k}$ for which (by \eqref{e:successione_1})
\begin{equation}\label{e:successione_2}
\lim_{k\uparrow \infty} \cH^{m-2+\alpha}_\infty (\sing (T_{p, r_k})\cap \bB_1 (0)) = \eta >0\, .
\end{equation}
Of course the sequence $\{s_k\}$ does not necessarily coincide (or is comparable to) $\{r_k\}$. However, it can be shown that, w.l.o.g., the two sequences can be assumed to coincide (cf. \cite{DS5} and \cite[Section 4]{Ema}). More precisely

\begin{proposition}[cf. {\cite[Proposition 1.3]{DS5}}]\label{p:successione}
If Theorem \ref{t:cod>1} were false then there would be an area minimizing current $T$ in a smooth Riemannian manifold $\Sigma$, a point $p\in \sing_Q (T)$ and a sequence of rescalings $T_{p,r_k}$ converging to a flat plane of multiplicity $Q$ and such that \eqref{e:successione_2} holds. 
\end{proposition}

We will see in a moment the simple idea behind Proposition \ref{p:successione}. Taking it for granted, one could at this point hope to carry the following program:
\begin{itemize}
\item[(A)] We apply Theorem \ref{t:app_main} to construct a sequence of Lipschitz maps $f_k$ whose graphs approximate efficiently $T_{p, r_k}$;
\item[(B)] After normalizing suitably $f_k$, we apply Theorem \ref{t:harmonic_final} and, up to extraction of a further subsequence, assume that it converges to a $\D$-minimizing multivalued map $f_\infty$;
\item[(C)] We finally use \eqref{e:successione_2} to show that $f_\infty$ has a singular set of positive $\cH^{m-2+\alpha}$ measure: this would contradict Theorem \ref{t:structure}.
\end{itemize}
(C) is again a ``persistence of singularity'' statement. Unfortunately it is not difficult to see that it is false in this form and thus the rough strategy outlined above must be suitably adjusted. After dealing with Proposition \ref{p:successione} we will discuss in detail why (C) fails and propose therefore a new strategy to prove Theorem \ref{t:cod>1}.

\subsection{The existence of a good sequence} The proof of Proposition \ref{p:successione} is still a suitable modification of Federer's reduction argument. By the discussion above, we first choose a point $p\in \sing_Q (T)$ and a sequence $r_k\downarrow 0$ where \eqref{e:successione_2} holds. Assume without loss of generality that $p=0$. If $T_{0, r_k}$ converges to a flat plane of multiplicity $Q$ we are done. Otherwise we can assume that it converges to some tangent cone $S$, which is singular and such that $\Theta (S, 0) = Q$. We now wish to show that $\cH^{m-2+\alpha}_\infty (\sing_Q (S)) > 0$. 

First of all, by the monotonicity formula, $\cH^{m-2+\alpha}_\infty (\rD_Q (S))>0$. Now, if 
\begin{equation}\label{e:no_bound}
\cH^{m-2+\alpha}_\infty (\sing_Q (S)) = 0
\end{equation}
then many of the points in $\rD_Q (S)$ should be regular: let us denote by $\reg_Q (S)$ the set of such points. Note that $\reg_Q (S)$ is relatively open. If $S\res (\reg_Q (S))$ has nonempty boundary, then the support of such boundary consists of elements in $\sing_Q (S)$ and it has positive $\cH^{m-1}$ measure. The latter statement is not compatible with \eqref{e:no_bound} for $\alpha\leq 1$. Thus $S' = S\res (\reg_Q (S))$ has no boundary. So it is an area minimizing cone with multiplicity $Q$ at every $p\in \supp (S')\setminus \{0\}$ and with multiplicity no larger than $Q$ at the tip $0$ (because $\|S'\|\leq \|S\|$). On the other hand the upper semicontinuity of the density implies that $\Theta (S', 0) \geq Q$: thus $\|S\| (\bB_1 (0)) = Q\omega_m = \|S'\| (\bB_1 (0))$. Hence $S$ and $S'$ must coincide. We thus conclude that Corollary \ref{c:densita} is applicable to $S$, which must be flat, contrary to our assumption.

\medskip

Having found that $S$ is another area minimizing current with large $\sing_Q (S)$, we can apply the discussion above to some point $p\in \supp (S)\setminus \{0\}$. We thus find a sequence $S_{p, r_k}$ such that $\lim_k \cH^{m-2+\alpha}_\infty (\sing_Q (S_{p, r_k})) > 0$. As above, $S_{p, r_k}$ can be assumed to be converging to some tangent cone $Z$: if it is flat, we then have achieved the conclusion of Proposition \ref{p:successione}. Otherwise $\cH^{m-2+\alpha} (\sing_Q (Z)) >0$ and we can restart with $Z$ in place of $S'$: this time, however, $Z$ splits off a line. Iterating this procedure we keep ``splitting off'' lines, until eventually we must reach a sequence as in the statement of Proposition \ref{p:successione}.

\subsection{Persistency of multiplicity $Q$ points} Having proved Proposition \ref{p:successione}, we are now in the position to carry on our first program. Point (A) is obvious and we have to face point (B). Let us fix a sequence as in Proposition \ref{p:successione} that it is converging to $Q\a{\pi_0}$ where $\pi_0 = \R^m\times \{0\}$. Thus, for a sufficiently large $k$, Theorem \ref{t:app_main} applies to $T_{0, r_k}$ in the cylinder $\bC_4 (0)$. Let $f_k$ be the corresponding approximating maps $f_k: B_1 (0, \pi_0) \to \Iqs$. It is not obvious, apriori, that we can apply Theorem \ref{t:harmonic_final}, since the excess $E_k := \bE (T_{0, r_k}, \bC_4 (0))$ might converge to zero too fast compared to $r_k \bA$. Let us leave this technical problem aside: we then could assume that $u_k := f_k/E_k^{\sfrac{1}{2}}$ is converging to a $\D$-minimizing map $u_\infty$. 

\medskip

Next, we can intuitively expect that $u_\infty$ has many points of multiplicity $Q$, in particular all the ones which are limits of sequences lying in $\p_{\pi_0} (\rD_Q (T_{0, r_k}))$. Namely, we expect that points in $\p_{\pi_0} (\rD_Q (T_{0, r_k}))$ cluster towards points where $u_\infty = Q\a{\etaa\circ u_\infty}$. This intuition is correct. In fact we can first prove the following 

\begin{theorem}[Persistence of $Q$-points, cf. {\cite[Theorem 1.7]{DS3}}]\label{t:persistence}
For every $\hat{\delta}, C^\star>0$, there is $\bar{s}\in ]0, \frac{1}{2}[$ such that, for every $s<\bar{s}$, there exists $\hat{\eps} (s,C^*,\hat{\delta})>0$ with the following property. If $T$ is as in Theorem \ref{t:app_main}, $E := \bE(T,\bC_{4\,r} (x)) < \hat{\eps}$,
$r^2 \bA^2 \leq C^\star E$ and $\Theta (T, (p,q)) = Q$ at some $(p,q)\in \bC_{r/2} (x)$, then the approximation
$f$ of Theorem~\ref{t:app_main} satisfies
\begin{equation}\label{e:persistence}
\int_{B_{sr} (p)} \cG (f, Q \a{\etaa\circ f})^2 \leq \hat{\delta} s^m r^{2+m} E\, .
\end{equation}
\end{theorem}

For the proof of the latter theorem we refer to \cite{DS3}. The argument relies on the improved monotonicity formula of Theorem \ref{t:monot_imp}.

Looking back at our $u_\infty$, which is the strong $L^2$ limit of $u_k = f_k/E_k^{\sfrac{1}{2}}$, when $p$ is a point in the domain of $u_\infty$ which is the limit of (the projections onto $\pi_0$ of) a sequence of $(p_k, q_k)\in \supp (T_k)$ with $\Theta (T_k, (p_k, q_k))= Q$ , we then have
\[
\lim_{r\to 0} \mint_{B_r (p)} \cG (u_\infty, Q \a{\etaa\circ u_\infty})^2 = 0\, .
\]
Since $u_\infty$ is $\D$-minimizing and, hence, continuous, we have $u_\infty (p) = Q \a{\etaa\circ u_\infty (p)}$. Now, we must have a set of points $p$ with positive $\cH^{m-2+\alpha}$ measures where this occurs.
Since Theorem \ref{t:structure} tells us that the singular set of $u_\infty$ has dimension at most $m-2$, the only alternative left is that $u_\infty$ is a classical harmonic single valued function counted $Q$ times. That is, once again the singularities have failed to survive in the limit. If we could exclude this disappearence of the singularities, we would have reached a contradiction and hence proved Theorem \ref{t:cod>1}.

\medskip

Let us look of what happens if we apply the analysis above when the current $T$ is the holomorphic curve of Example \ref{e:vvbad} in a neighborhood of $0$. It is obvious that (in complex coordinates) the procedure above will deliver the map $u_\infty (z) = 2 \a{z^2}$: although the currents $T_{0,r}$ are singular at the origin, $u_\infty$ is regular. If we compare our situation with the proof of Theorem \ref{t:structure} outlined in Section \ref{s:frequency}, it is quite obvious why we failed to capture the singularity in the limit: we have not subtracted the ``average of the sheets'', namely the regular part of our multiple valued function. The latter has much higher energy than the branching singularity, which is a very small perturbation: if we do not normalize our approximations in some way, we fail to capture the singular behavior in the limit. 

\section{The center manifold: construction algorithm}\label{s:cm}

Summarizing the discussion of the previous section, we are confronted with the following problem. Given a $Q$-valued $\D$-minimizing function $u= \sum_i \a{u_i}$, the average of its sheets, namely $\etaa \circ u := \frac{1}{Q} \sum_i u_i$, is a classical harmonic function and after subtracting it from $u$ we find a new $\D$-minimizing $Q$-valued function
\[
\bar{u} := \sum_i \a{u_i - \etaa\circ u}\, .
\] 
When $\bar{u}$ is nontrivial, a point $p$ where $\bar{u} (p) = Q\a{0}$ is necessarily singular. Loosely speaking we could say that $\bar{u}$ is the ``well-balanced part'' of $u$.
If an area minimizing current $T$ satisfies Assumption \ref{ipotesi_base_app} we would like to have a replacement for the average of the sheets $\etaa\circ u$ and a replacement for the well-balanced $Q$-valued map $\bar{u}$. 

One possibility would be to apply Theorem \ref{t:app_main}: we then gain a corresponding approximating Lipschitz map $f$: the average of its sheets, namely $\etaa\circ f$, and its well-balanced part $\bar f$ are both well defined. 
However, we wish to use these objects in a blow-up procedure: obviously $\etaa\circ f$ and $\bar{f}$ do not serve our purposes, since $f$ is a good approximation of the current only at the scale of a certain given cylinder

\medskip

We would rather like to localize the idea above. This is obviously only possible in those regions (and those scales) where the current is sufficiently flat. On the other hand we might not worry about those portions, or those scales, at which the current is not sufficiently flat: in the blow-up procedure we wish to capture the limiting behavior of the current around those points belonging to $\sing_Q (T)$ and we already saw in the previous section, namely in the proof of Proposition \ref{p:successione}, that when a lot of points $\rD_Q (T)$ cluster at a certain scale, the current is necessarily rather flat.

\subsection{Preliminary considerations and notation}\label{ss:constr_whit} Localizing the basic idea above is a very delicate issue, which involves several parameters. First of all, to fix ideas we will assume that our center manifold will be constructed in a cylinder $\bC$ of size comparable to $1$ (namely the radius will be a fixed geometric constant, certainly larger than $1$), centered at the origin and with basis parallel to $\pi_0 := \R^m \times \{0\}$. We will assume that in the cylinder $\bC$ the ambient manifold $\Sigma$ is very close to be flat: this ``almost flatness'' is measured in a suitable norm.

\begin{ipotesi}\label{ipotesi}
$\eps_0\in ]0,1]$ is a fixed constant and
$\Sigma \subset \bB_{7\sqrt{m}} \subset R^{m+n}$ is a $C^{3,\eps_0}$ $(m+\bar{n})$-dimensional submanifold with no boundary in $\bB_{7\sqrt{m}}$. We moreover assume that, for each $p\in \Sigma$,  $\Sigma$ is the graph of a
$C^{3, \eps_0}$ map $\Psi_p: T_p\Sigma\cap \bB_{7\sqrt{m}} \to T_p\Sigma^\perp$. We denote by $\mathbf{c} (\Sigma)$ the
number $\sup_{p\in \Sigma} \|D\Psi_p\|_{C^{2, \eps_0}}$. 
$T^0$ is an $m$-dimensional integral current of $\mathbb R^{m+n}$ with support in $\Sigma\cap \bar\bB_{6\sqrt{m}}$ and finite mass. It is area-minimizing and moreover
\begin{gather}
\Theta (0, T^0) = Q\quad \mbox{and}\quad \partial T^0 \res \B_{6\sqrt{m}} = 0,\label{e:basic}\\
\quad \|T^0\| (\B_{6\sqrt{m} \rho}) \leq \big(\omega_m Q (6\sqrt{m})^m + \eps_2^2\big)\,\rho^m
\quad \forall \rho\leq 1,\label{e:basic2}\\
\bE\left(T^0,\B_{6\sqrt{m}}\right)=\bE\left(T^0,\B_{6\sqrt{m}},\pi_0\right),\label{e:pi0_ottimale}\\
\bmo := \max \left\{\mathbf{c} (\Sigma)^2, \bE\left(T^0,\B_{6\sqrt{m}}\right)\right\} \leq \eps_2^2 \leq 1\, .\label{e:small ex}
\end{gather}
$\eps_2$ is a positive number whose choice will be specified in each statement.
\end{ipotesi}

In what follows we set $l:= n - \bar{n}$. To avoid discussing domains of definitions it is convenient to extend $\Sigma$ so
that it is an entire graph over all $T_p \Sigma$. Moreover
we will often need to parametrize $\Sigma$ as the graph of a map $\Psi: \mathbb R^{m+\bar n}\to \mathbb R^l$. Although
we do not assume that $\mathbb R^{m+\bar n}\times \{0\}$ is tangent to $\Sigma$ at any $p$, it is clear that, when $\eps_2$ is sufficiently small, $\Sigma$ is the graph of a function $\Psi$ over $\R^{m+\bar n}$ with $\|D\Psi\|_{C^{2,\eps_0}}\leq C_0 \bmo^{\sfrac{1}{2}}$.

The next lemma is a standard consequence of the theory of area-minimizing currents.

\begin{lemma}[Non-empty vertical slices, {cf. \cite[Lemma 1.6]{DS4}}]\label{l:tecnico1}
There are positive constants $C_0 (m,n, \bar{n},Q)$ and $c_0 (m,n,\bar{n}, Q)$ with the following property.
If $T^0$ is as in Assumption \ref{ipotesi}, $\eps_2 < c_0$ and $T:= T^0 \res \bB_{23\sqrt{m}/4}$, then:
\begin{align}
&\partial T \res \bC_{11\sqrt{m}/2} (0, \pi_0)= 0\, ,\quad (\p_{\pi_0})_\sharp T\res \bC_{11 \sqrt{m}/2} (0, \pi_0) = Q \a{B_{11\sqrt{m}/2} (0, \pi_0)}\label{e:geo semplice 1}\\
&\quad\qquad\qquad\mbox{and}\quad\bh (T, \bC_{5\sqrt{m}} (0, \pi_0)) \leq C_0 \bmo^{\sfrac{1}{2m}}\, .\label{e:pre_height}
\end{align}
In particular for each $x\in B_{11\sqrt{m}/2} (0, \pi_0)$ there is  a point
$p\in \supp (T)$ with $\p_{\pi_0} (p)=x$.
\end{lemma}

From now we will always work with the current $T$ of Lemma \ref{l:tecnico1}.
We specify next some useful notation.
For each $j\in \N$, $\sC^j$ denotes the family of closed cubes $L$ of $\pi_0$ of the form 
\begin{equation}\label{e:cube_def}
[a_1, a_1+2\ell] \times\ldots  \times [a_m, a_m+ 2\ell] \times \{0\}\subset \pi_0\, ,
\end{equation}
where $2\,\ell = 2^{1-j} =: 2\,\ell (L)$ is the side-length of the cube, 
$a_i\in 2^{1-j}\Z$ $\forall i$ and we require in
addition $-4 \leq a_i \leq a_i+2\ell \leq 4$. 
To avoid cumbersome notation, we will usually drop the factor $\{0\}$ in \eqref{e:cube_def} and treat each cube, its subsets and its points as subsets and elements of $\mathbb R^m$. Thus, for the {\em center $x_L$ of $L$} we will use the notation $x_L=(a_1+\ell, \ldots, a_m+\ell)$, although the precise one is $(a_1+\ell, \ldots, a_m+\ell, 0, \ldots , 0)$.
Next we set $\sC := \bigcup_{j\in \N} \sC^j$. 
If $H$ and $L$ are two cubes in $\sC$ with $H\subset L$, then we call $L$ an {\em ancestor} of $H$ and $H$ a {\em descendant} of $L$. When in addition $\ell (L) = 2\ell (H)$, $H$ is {\em a son} of $L$ and $L$ {\em the father} of $H$.

\begin{definition}\label{e:whitney} A Whitney decomposition of $[-4,4]^m\subset \pi_0$ consists of a closed set $\bGam\subset [-4,4]^m$ and a family $\mathscr{W}\subset \sC$ satisfying the following properties:
\begin{itemize}
\item[(w1)] $\bGam \cup \bigcup_{L\in \mathscr{W}} L = [-4,4]^m$ and $\bGam$ does not intersect any element of $\mathscr{W}$;
\item[(w2)] the interiors of any pair of distinct cubes $L_1, L_2\in \mathscr{W}$ are disjoint;
\item[(w3)] if $L_1, L_2\in \mathscr{W}$ have nonempty intersection, then $\frac{1}{2}\ell (L_1) \leq \ell (L_2) \leq 2\, \ell (L_1)$.
\end{itemize}
\end{definition}

Observe that (w1) - (w3) imply 
\begin{equation}\label{e:separazione}
{\rm sep}\, (\bGam, L) := \inf \{ |x-y|: x\in L, y\in \bGam\} \geq 2\ell (L)  \quad\mbox{for every $L\in \mathscr{W}$.}
\end{equation}
However, we do {\em not} require any inequality of the form 
${\rm sep}\, (\bGam, L) \leq C \ell (L)$, although this would be customary for what is commonly 
called a Whitney decomposition in the literature.

We next set up a refinement procedure: starting with an initial grid of dyadic cubes, denoted by $\sC^{N_0}$, with sidelength $2^{1-N_0}$, we subdivide each of them into $2^m$ cubes $H$ of sidelength $2 \ell (H) = \ell (L)$. To this purpose we will need to associate to each cube $L$ a corresponding ball $\bB_L$ which is centered on a point lying ``vertically'' over the center of the cube. Its definition will be specified later. We next proceed to give a rough outline of the refinement procedure.

Now, given our starting hypothesis, we know that both the excess and the height in $\bB_L$ are small compared to $\pi_0$ whenever $L$ belongs to the initial grid of dyadic cubes, which are taken with side-length comparable to $1$. We wish to keep refining the cubes $H$ as long as the height and the excess in the corresponding balls $\bB_H$ keep sufficiently small. However, the reference plane might tilt as we refine the scales and we wish to keep track of this. For this reason, we denote by $\pi_L$ a given plane for which $\bE (T, \bB_L) = \bE (T, \bB_L, \pi_L)$ and we define the height of $T$ in the ball $\bB_L$ as $\bh (T, \bB_L) := \bh (T, \bB_L, \pi_L)$. 

We then stop the refining procedure at some dyadic cube $L$ if
\begin{itemize}
\item[(EX)] The excess has become too large in $\bB_L$;
\item[(HT)] The height has become too large in $\bB_L$, although the excess has remained small;
\item[(NN)] The refinement stopped at some cube $J$ which touches $L$ and has double sidelength, although in $L$ itself both the height and the excess would be small enough to keep refining.
\end{itemize}
The reader familiar with Whitney's (or Calderon-Zygmund) decompositions will recognize that the latter condition is enforced to guarantee that, after all the steps of the refinement procedure  have been carried on, all nearby cubes have comparable sides. Unfortunately the conditions (EX) and (HT) taken alone do not guarantee this outcome (the troubles are indeed caused by condition (EX)) and the extra (NN) is a source of a few technical complications.

\subsection{Parameters} In order to detail the conditions (EX), (HT) and (NN) we need to introduce several parameters.

\begin{ipotesi}\label{parametri}
$C_e,C_h,\beta_2,\delta_2, M_0$ are positive real numbers and $N_0$ a natural number for which we assume
always
%\footnote{{\color{red} Q: A number of times later we encounter $\beta_2/4$; why not replace these with $\delta_2$?}
%{\color{blue} A: because it is just an accident that those constants equal $\delta_2$.  In particular if we choose $\delta_2$ smaller those statements would still hold with $\beta_2/4$.}}
\begin{gather}
\beta_2 = 4\,\delta_2 = \min \left\{\frac{1}{2m}, \frac{\gamma_1}{100}\right\}, \quad
\mbox{where $\gamma_1$ is the constant of \cite[Theorem~1.4]{DS3},}\label{e:delta+beta}\\
M_0 \geq C_0 (m,n,\bar{n},Q) \geq 4\,  \quad
\mbox{and}\quad \sqrt{m} M_0 2^{7-N_0} \leq 1\, . \label{e:N0}
\end{gather}
\end{ipotesi}

As we can see, $\beta_2$ and $\delta_2$ are fixed. The other parameters are not fixed but are subject to further restrictions in the various statements, respecting the following ``hierarchy''. As already mentioned, ``geometric constants'' are assumed to depend only upon $m, n, \bar{n}$ and $Q$. The dependence of other constants upon the various parameters $p_i$ will be highlighted using the notation $C = C (p_1, p_2, \ldots)$.

\begin{ipotesi}[Hierarchy of the parameters]\label{i:parametri}
In all the statements of the paper
\begin{itemize}
\item[(a)] $M_0$ is larger than a geometric constant (cf. \eqref{e:N0}) or larger than a costant $C (\delta_2)$;
\item[(b)] $N_0$ is larger than $C (\beta_2, \delta_2, M_0)$;
\item[(c)] $C_e$ is larger than $C(\beta_2, \delta_2, M_0, N_0)$;
\item[(d)] $C_h$ is larger than $C(\beta_2, \delta_2, M_0, N_0, C_e)$;
\item[(e)] $\eps_2$ is smaller than $c(\beta_2, \delta_2, M_0, N_0, C_e, C_h)$ (which will always be positive).
\end{itemize}
\end{ipotesi}

The functions $C$ and $c$ will vary in the various statements: 
the hierarchy above guarantees however that there is a choice of the parameters for which {\em all} the restrictions required in the statements of the next propositions are simultaneously satisfied. In fact it is such a choice which is then made in \cite{DS5}. To simplify our exposition, for smallness conditions on $\eps_2$ as in (e) we will use the sentence ``$\eps_2$ is sufficiently small''.

\subsection{The Whitney decomposition} 
Thanks to Lemma \ref{l:tecnico1}, for every $L\in \sC$,  we may choose $y_L\in \pi_L^\perp$ so that $p_L := (x_L, y_L)\in \supp (T)$ (recall that $x_L$ is the center of $L$). $y_L$ is in general not unique and we fix an arbitrary choice.
A more correct notation for $p_L$ would be $x_L + y_L$. This would however become rather cumbersome later, when we deal with various decompositions of the
ambient space in triples of orthogonal planes. We thus abuse the notation slightly in using $(x,y)$ instead of $x+y$ and, consistently, $\pi_0\times \pi_0^\perp$ instead of $\pi_0 + \pi_0^\perp$.

\begin{definition}[Refining procedure, {cf. \cite[Definition 1.10]{DS4}}]]\label{d:refining_procedure}
For $L\in \sC$ we set $r_L:= M_0 \sqrt{m} \,\ell (L)$ and 
$\B_L := \bB_{64 r_L} (p_L)$. We next define the families of cubes $\sS\subset\sC$ and $\sW = \sW_e \cup \sW_h \cup \sW_n \subset \sC$ with the convention that
$\sS^j = \sS\cap \sC^j, \sW^j = \sW\cap \sC^j$ and $\sW^j_{\square} = \sW_\square \cap \sC^j$ for $\square = h,n, e$. We define $\sW^i = \sS^i = \emptyset $ for $i < N_0$. We proceed with $j\geq N_0$ inductively: if no ancestor of $L\in \sC^j$ is in $\sW$, then 
\begin{itemize}
\item[(EX)] $L\in \sW^j_e$ if $\bE (T, \B_L) > C_e \bmo\, \ell (L)^{2-2\delta_2}$;
\item[(HT)] $L\in \sW_h^j$ if $L\not \in \mathscr{W}_e^j$ and $\bh (T, \B_L) > C_h \bmo^{\sfrac{1}{2m}} \ell (L)^{1+\beta_2}$;
\item[(NN)] $L\in \sW_n^j$ if $L\not\in \sW_e^j\cup \sW_h^j$ but it intersects an element of $\sW^{j-1}$;
\end{itemize}
if none of the above occurs, then $L\in \sS^j$.
We finally set
\begin{equation}\label{e:bGamma}
\bGam:= [-4,4]^m \setminus \bigcup_{L\in \sW} L = \bigcap_{j\geq N_0} \bigcup_{L\in \sS^j} L.
\end{equation}
\end{definition}
Observe that, if $j>N_0$ and $L\in \sS^j\cup \sW^j$, then necessarily its father belongs to $\sS^{j-1}$.

\begin{proposition}[Whitney decomposition, {cf. \cite[Proposition 1.11]{DS4}}]\label{p:whitney}
Let Assumptions \ref{ipotesi} and \ref{parametri} hold and let $\eps_2$ be sufficiently small.
Then $(\bGam, \mathscr{W})$ is a Whitney decomposition of $[-4,4]^m \subset \pi_0$.
Moreover, for any choice of $M_0$ and $N_0$, there is $C^\star := C^\star (M_0, N_0)$ such that,
if $C_e \geq C^\star$ and $C_h \geq C^\star C_e$, then 
\begin{equation}\label{e:prima_parte}
\sW^{j} = \emptyset \qquad \mbox{for all $j\leq N_0+6$.}
\end{equation}
Moreover, the following 
estimates hold with $C = C(\beta_2, \delta_2, M_0, N_0, C_e, C_h)$:
\begin{gather}
\bE (T, \B_J) \leq C_e \bmo^{}\, \ell (J)^{2-2\delta_2} \quad \text{and}\quad
\bh (T, \B_J) \leq C_h \bmo^{\sfrac{1}{2m}} \ell (J)^{1+\beta_2}
\quad \forall J\in \sS\,, \label{e:ex+ht_ancestors}\\
 \bE (T, \B_L) \leq C\, \bmo^{}\, \ell (L)^{2-2\delta_2}\quad \text{and}\quad
\bh (T, \B_L) \leq C\, \bmo^{\sfrac{1}{2m}} \ell (L)^{1+\beta_2}
\quad \forall L\in \sW\, . \label{e:ex+ht_whitney}
\end{gather}
\end{proposition}

\subsection{Interpolating functions} What do we expect from the decomposition above? Fix a point $x\in [-1, 1]^m$. Only two situations might occur:
\begin{itemize}
\item[(NS)] $x$ does not belong to any cube where the refinement procedure stopped;
\item[(S)] $x\in L$ for some cube $L$ where the refinement procedure stopped.
\end{itemize}
Let $\bGam$ be the set of points as in (NS). If the stopping conditions (EX) and (HT) are sufficiently severe, we can expect that $T\res (\bGam \times \pi_0^\perp)$ is a $Q$ multiple of a portion of a (rather smooth) single valued graph. 

Consider next a point $x\in L$ as in (S). Although we stopped the refinement at $L$, at a slightly larger scale we still have a very small excess and a very small height: both $\bE (T, \bB_L)$ and $\bh (T, \bB_L)$ are still rather small. We thus can hope to apply the approximation Theorem \ref{t:app_main} in a suitable cylinder $\bC_{32 r_L} (p_L, \pi_L)$: we can then construct a good Lipschitz $Q$-valued approximation $f_L: B_{8r_L} (p_L, \pi_L) \to \Iq (\pi_L^\perp)$, which will be called the {\em $\pi_L$-approximation in $\bB_L$}. Finally we take its average $\etaa \circ h_L$ and smooth it by convolution with a kernel (which we take to be radial, although the importance of this assumption will become clear only in the next section).

\medskip

In order to carry on our program we fix two important functions $\vartheta,\varrho: \R^m \to \R$.

\begin{ipotesi}\label{mollificatore}
$\varrho\in C^\infty_c (B_1)$ is radial, $\int \varrho =1$ and $\int |x|^2 \varrho (x)\, dx = 0$. For $\lambda>0$ $\varrho_\lambda$ 
denotes, as usual, $x\mapsto \lambda^{-m} \varrho (\frac{x}{\lambda})$.
$\vartheta\in C^\infty_c \big([-\frac{17}{16}, \frac{17}{16}]^m, [0,1]\big)$ is identically $1$ on $[-1,1]^m$.
\end{ipotesi}

$\varrho$ will be used as convolution kernel for smoothing maps $z$ defined on $m$-dimensional planes $\pi$ of $\mathbb R^{m+n}$. In particular, having fixed an isometry $A$ of $\pi$ onto $\mathbb R^m$, the smoothing will be given by $[(z \circ A) * \varrho] \circ A^{-1}$. Observe that since $\varrho$ is radial, our map does not depend on the choice of the isometry
and we will therefore use the shorthand notation $z*\varrho$.

\begin{definition}[$\pi$-approximations, {cf \cite[Definition 1.13]{DS4}}]\label{d:pi-approximations}
Let $L\in \sS\cup \sW$ and $\pi$ be an $m$-dimensional plane. If $T\res\bC_{32 r_L} (p_L, \pi)$ fulfills 
the assumptions of \cite[Theorem 1.4]{DS3} in the cylinder $\bC_{32 r_L} (p_L, \pi)$, then the resulting map $f: B_{8r_L} (p_L, \pi)  \to \Iq (\pi^\perp)$ given by \cite[Theorem 1.4]{DS3} is a {\em $\pi$-approximation of $T$ in $\bC_{8 r_L} (p_L, \pi)$}. 
The map $\hat{h}:B_{7r_L} (p_L, \pi) \to \pi^\perp$ given
%\footnote{{\color{red} Q: Maybe $\hat{h}$ should be $\hat{h}_L$?} {\color{blue} A: No, because $\hat{h}_L$ will be used when $\pi = \pi_L$, whereas the reference plane $\pi$ above might be different.}} 
by $\hat{h}:= (\etaa\circ f)* \varrho_{\ell (L)}$ will be called the {\em smoothed average of the $\pi$-approximation}, where we recall the notation $\etaa \circ f(x) := Q^{-1} \sum_{i=1}^Q f_i(x)$ for any $Q$-valued map $f = \sum_i \a{f_i}$. 
\end{definition}

\subsection{Gluing the interpolations}
For each $L\in \sS\cup \sW$ we let $\hat\pi_L$ be an optimal plane in $\bB_L$, namely so that $\bE (T, \bB_L, \pi_L) = \bE (T, \bB_L)$ and $\bh (T, \bB_L) = \bh (T, \bB_L, \pi_L)$, and choose an $m$-plane $\pi_L\subset T_{p_L} \Sigma$ which minimizes $|\hat\pi_L-\pi_L|$.

We wish to patch the graphs of the various tilted interpolating functions $h_L$ in a single submanifold (and we also hope that this submanifold will glue smoothly with the portion of the current lying over $\bGam$!). However, since the graphs of $h_L$ are relative to different systems of coordinates, we need to parametrize them as graphs in a common coordinate system. Given condition (EX), we can hope that along the refinement procedure the planes $\pi_L$ did not tilt much. If this is the case the graph of each $h_L$ can be described by the graph of some $g_L$ in the ``original coordinates'' $\pi_0\times \pi_0^\perp$. We could then glue the various $g_L$ together using a partition of unity. 

\begin{lemma}[{Cf. \cite[Lemma 1.15]{DS4}}]\label{l:tecnico2}
Let the assumptions of Proposition \ref{p:whitney} hold and assume $C_e \geq C^\star$ and $C_h \geq C^\star C_e$ (where $C^\star$ is the
constant of Proposition \ref{p:whitney}). For any choice of the other parameters,
if $\eps_2$ is sufficiently small, then $T\res \bC_{32 r_L} (p_L, \pi_L)$ satisfies the assumptions of
\cite[Theorem 1.4]{DS3} for any $L\in \sW\cup \sS$. 
Moreover, if $f_L$ is a $\pi_L$-approximation, denote by $\hat{h}_L$ its smoothed average and by $\bar{h}_L$ the map $\p_{T_{p_L}\Sigma} (\hat{h}_L)$,
which takes value in the plane $\varkappa_L := T_{p_L} \Sigma \cap \pi_L^\perp$, i.e. the orthogonal complement of $\pi_L$ in $T_{p_L} \Sigma$.
If we let $h_L$ be the map $x\mapsto h_L (x):= (\bar{h}_L (x), \Psi_{p_L} (x, \bar{h}_L (x)))\in \varkappa_L \times T_{p_L} \Sigma^\perp$, 
then there is a smooth map $g_L: B_{4r_L} (p_L, \pi_0)\to \pi_0^\perp$ such that 
$\bG_{g_L} = \bG_{h_L}\res \bC_{4r_L} (p_L, \pi_0)$.
\end{lemma}

We are now ready to finally specify how we patch the interpolating functions. The center manifold will then be a graph which is obtained by a suitable limiting process. 

\begin{definition}[Interpolating functions, {cf. \cite[Definition 1.16]{DS4}}]\label{d:glued}
The maps $h_L$ and $g_L$ in Lemma \ref{l:tecnico2} will be called, respectively, the
{\em tilted $L$-interpolating function} and the {\em $L$-interpolating function}.
For each $j$ let $\sP^j := \sS^j \cup \bigcup_{i=N_0}^j \sW^i$ and
for $L\in \sP^j$ define $\vartheta_L (y):= \vartheta (\frac{y-x_L}{\ell (L)})$. Set
\begin{equation}
\hat\varphi_j := \frac{\sum_{L\in \sP^j} \vartheta_L g_L}{\sum_{L\in \sP^j} \vartheta_L} \qquad \mbox{on $]-4,4[^m$},
\end{equation}
let $\bar{\varphi}_j (y)$ be the first $\bar{n}$ components of $\hat{\varphi}_j (y)$ and define
$\varphi_j (y) = \big(\bar{\varphi}_j (y), \Psi (y, \bar{\varphi}_j (y))\big)$.
$\varphi_j$ will be called the {\em glued interpolation} at the step $j$.
\end{definition}

\begin{theorem}[Existence of the center manifold, {cf. \cite[Theorem 1.17]{DS4}}]\label{t:cm}
Assume that the hypotheses of Lemma \ref{l:tecnico2} hold and
let $\kappa := \min \{\eps_0/2, \beta_2/4\}$. For any choice of the other parameters,
if $\eps_2$ is sufficiently small, then
\begin{itemize}
\item[(i)] $\|D\varphi_j\|_{C^{2, \kappa}} \leq C \bmo^{\sfrac{1}{2}}$ and $\|\varphi_j\|_{C^0}
\leq C \bmo^{\sfrac{1}{2m}}$, with $C = C(\beta_2, \delta_2, M_0, N_0, C_e, C_h)$.
\item[(ii)] if $L\in \sW^i$ and $H$ is a cube concentric to $L$ with $\ell (H)=\frac{9}{8} \ell (L)$, then $\varphi_j = \varphi_k$ on $H$ for any $j,k\geq i+2$. 
\item[(iii)] $\varphi_j$ converges in $C^3$ to a map $\phii$ and $\cM:= \gr (\phii|_{]-4,4[^m}
)$ is a $C^{3,\kappa}$ submanifold of $\Sigma$.
\end{itemize}
\end{theorem}

\begin{definition}[Center manifold and contact set {cf. \cite[Definition 1.18]{DS4}}]\label{d:cm}
The manifold $\cM$ in Theorem \ref{t:cm} is called
{\em a center manifold of $T$ relative to $\pi_0$} and 
$(\bGam, \sW)$ the {\em Whitney decomposition associated to $\cM$}. 
Setting $\Phii(y) := (y,\phii(y))$, we call
$\Phii (\bGam)$ the {\em contact set}.
\end{definition}

\section{The $C^3$ estimate for the center manifold}\label{s:cm2}

Theorem \ref{t:cm} is probably the most complicated part of the proof of Theorem \ref{t:cod>1}. However it is important to notice that Theorem \ref{t:cm} alone does not encode  the full strength of the construction described above: we will see in the next section that the graph of $\phii$ is indeed a very good substitute for the ``average of the sheets of a $Q$-valued graph''. For instance, the algorithm can be applied even under the assumption that the density of $T$ equals $Q$ $\|T\|$-almost everywhere in $\bC$: in this case the refinement procedure never stops, $\bGam$ equals $[-1,1]^m$ and finally $\supp (T)\cap [-1,1]^m\times \pi_0^\perp \subset \gr (\phii)$. Namely, by the constancy theorem $T\res ([-1,1]^m\times \pi_0^\perp) = Q \bG_\phii$ and thus we have gained two more derivative in the conclusion of Theorem \ref{t:DG}.

The latter surprising conclusion, namely that with a ``purely geometric construction'' it is possible to improve the classical $\eps$-regularity theorem by $2$ derivatives, is already observed in the introduction of Almgren's monograph. It is however possible to find a rather fast shortcut to this conclusion since multiple valued functions are not needed. A self-contained ``elementary'' proof has in fact been given in \cite{DS-cm}. The latter reference contains, in a highly simplified setting, the most important estimates which hide behind Theorem \ref{t:cm}. In this paragraph we will give a rough idea of the $C^3$ regularity of $\phii$. 

\subsection{Propagation of the $C^k$ estimates along the scales} It is obvious that in order to show Theorem \ref{t:cm} we need to prove a uniform $C^{3,\kappa}$ estimate on $g_L$ for any $L\in \sW$. This alone will not suffice: as it is simple to conclude from the usual properties of a generic partition of unity, we also need to show that, when $L$ and $H$ are two neighboring cubes in $\sW$, then $\|D^l (g_L- g_H)\|_{C^0} \leq C \ell (H)^{3-l+\kappa}$, for every $l\in \{0,1,2,3\}$. However, leaving aside the ``interaction'' between nearby cubes, let us focus on $\|g_L\|_{C^3}$ and, to simplify the matter even further, let us in fact look at $\|h_L\|_{C^3}$. It is rather plausible that if we could prove a uniform bound on $\|h_L\|_{C^3}$ for the tilted interpolating functions $h_L$, this will not be destroyed by the change to the coordinates $\pi_0\times \pi_0^\perp$. 

\medskip

Let us therefore fix $L\in \sW$ and let $L= L_i \subset L_{i+1} \subset \ldots \subset L_{N_0}$ be a chain of dyadic cubes where $L_{j-1}$ is the father of $L_j$ for every $j$. Now, if we ignore the Riemannian manifold $\Sigma$ and assume that it coincides with $\R^{m+n}$, $h_{L_{N_0}}$ is the convolution of a Lipschitz function at a scale which is obviously comparable with $1$ (since $\ell (L_{N_0}) = 2^{1-N_0}$ and $N_0$ is a fixed ``constant'', although rather large). Thus $\|D^k h_{L_{N_0}}\|_{C^0}$ is in fact bounded a-priori with a constant depending only on $k$ (and on all the other parameters of the construction). We next want to study how the norm $\|D^k h_{L_j}\|_{C^0}$ might increase compared to $\|D^k h_{L_{j-1}}\|_{C^0}$: the hope is that this can be bounded by some power of $\ell (L_j)$, leading in turn to a convergent geometric series when $k\in \{0,1,2,3\}$. This would then give a uniform bound on $\|h_L\|_{C^3}$.

Ideally we would like to compute $\|D^k (h_{L_j} - h_{L_{j-1}})\|_{C^0}$. This is however not really possible, since the two functions are defined according to two different coordinate systems (namely $\pi_{L_j}\times \pi_{L_j}^\perp$ and $\pi_{L_{j-1}}\times \pi_{L_{j-1}}^\perp$). Let us however assume, for the sake of argument, that $\pi_{L_j} = \pi_{L_{j-1}} =: \pi$. Moreover, to simplify the notation let us denote $L_j$ by $J$ and $L_{j-1}$ by $H$.

Under this assumption, $h_J$ is the convolution of $\etaa\circ f_J$ and $h_H$ the convolution of $\etaa\circ f_H$. We know that both $f_H$ and $f_J$ approximate very accurately the area minimizing current $T$, at two scales which are comparable by a factor $2$. Thus, for both graphs $\bG_{f_J}$ and $\bG_{f_H}$ the first variation is close to $0$, which in turn, given the smallness of the excess of the current at that scale, should imply that both $\etaa\circ f_J$ and $\etaa\circ f_H$ are almost harmonic. 

\medskip

The latter discussion is correct but must be quantified. It is not difficult to see that it can be translated into an estimate for $\Delta (\etaa\circ f_J)$ and $\Delta (\etaa\circ f_H)$ in some negative Sobolev space (more precisely we use the $W^{-1,1}$ norm, cf. \cite[Proposition 5.1]{DS4}). To simplify the matter even further, let us assume that both $\etaa\circ f_H$ and $\etaa\circ f_J$ are in fact harmonic. The regularization by convolution will then not change the two functions, because the convolution kernel is radial. Thus we would have $\|D^k (h_H - h_J)\|_{C^0} =
\|D^k (\etaa\circ f_H - \etaa\circ f_J)\|_{C^0}$. On the other hand, again by the mean-value formula for harmonic function, we could estimate
\begin{equation}\label{e:media}
\|D^k (h_H - h_J)\|_{C^0} \leq \frac{C}{\ell (H)^{m+k}} \|\etaa\circ f_H - \etaa\circ f_J\|_{L^1}\, .
\end{equation}
Let us however recall that the graph of $f_H$ coincides with the current $T$, except for a set of measure $C \ell (H)^m \bE (T, \bB_H)^{1+\gamma_1}$. In turn the excess $\bE (T, \bB_H)$ is of size $\ell (H)^{2-2\delta_1}$. Since an analogous consideration holds for $J$ in place of $H$, we conclude that the two maps $f_H$ and $f_J$ coincides except for a set of measure at most $\ell (H)^{m+(2-2\delta_1)(1+\gamma_1)}$. On the other hand the ``heights of both maps'' is of order $\ell (H)^{1+\beta_2}$, i.e. the available estimate for the heights $\bh (T, \bB_H)$ and $\bh (T, \bB_L)$. Combining these observations, we conclude that   
\[
\|\etaa\circ f_H - \etaa\circ f_J\|_{L^1} \leq C \ell (H)^{m+3+\kappa}\, .
\]
Inserting the latter inequality in \eqref{e:media} we would then conclude 
\[
\|D^k (h_H-h_J)\|_{C^0} \leq C \ell (H)^{3+\kappa -k}\, .
\]

\subsection{The key PDE estimate} In order to carry on rigorously the program outlined in the previous paragraph, rather than comparing directly $h_{HJ}$ and $h_J$, we introduce an intermediate map constructed in the reference cube $L$ relative to a system of coordinates which is suitable for anbother $J$. The following proposition is the technical tool which ensures the existence of the relevant maps.

\begin{proposition}[Existence of interpolating functions, {cf. \cite[Proposition 4.2]{DS4}}]\label{p:gira_e_rigira}
Assume that the hypotheses of Assumptions \ref{ipotesi} and \ref{parametri} hold, that
$C_e \geq C^\star$ and $C_h \geq C^\star C_e$, where $C^\star (M_0, N_0)$ is a suitable constant.
The following facts are true provided $\eps_2$ is sufficiently small. 
Let $H, L\in \sW\cup \sS$ be such that either $H\subset L$ or $H\cap L \neq \emptyset$ and
$\frac{\ell (L)}{2} \leq \ell (H) \leq \ell (L)$. Then,
\begin{itemize}
\item[(i)] for $\pi= \pi_H, \hat{\pi}_H$, $(\p_{\pi})_\sharp T\res \bC_{32r_L} (p_L, \pi) = Q \a{B_{32r_L} (p_L, \pi))}$ and $T$ satisfies the assumptions of \cite[Theorem 1.4]{DS3} in the cylinder $\bC_{32 r_L} (p_L, \pi)$;
\item[(ii)] Let $f_{HL}$ be the $\pi_H$-approximation of $T$ in $\bC_{8 r_L} (p_L, \pi_H)$ and $h := (\etaa\circ f_{HL})*\varrho_{\ell (L)}$ be its smoothed average. Set $\varkappa_H := \pi_H^\perp \cap T_{p_H} \Sigma$ and consider the maps 
\begin{equation*}
\begin{array}{lll}
x\quad \mapsto\quad \bar{h} (x)  &:=  \p_{T_{p_H}\Sigma} (h)&\in \varkappa_H\\ 
x\quad \mapsto\quad h_{HL} (x) &:= (\bar{h} (x), \Psi_{p_H} (x, \bar{h} (x)))&\in \varkappa_H \times (T_{p_H} (\Sigma))^\perp\, .
\end{array}
\end{equation*} 
Then there is a smooth
$g_{HL} :  B_{4r_L} (p_L, \pi_0)\to \pi_0^\perp$ s.t. $\bG_{g_{HL}} = \bG_{h_{HL}}\res \bC_{4r_L} (p_L, \pi_0)$.
\end{itemize}
\end{proposition}

\begin{definition}\label{d:mappe_h_HL}%\label{d:interpolations 2}
$h_{HL}$ and $g_{HL}$ will be called, respectively, {\em tilted $(H,L)$-interpolating function} and {\em $(H,L)$-interpolating function}.
\end{definition}

Observe that the tilted $(L,L)$-interpolating function and the $(L,L)$-interpolating function correspond to the tilted $L$-interpolating function and to the $L$-interpolating function 
of Definition~\ref{d:glued}. Obviously, Lemma \ref{l:tecnico2} is just a particular case of Proposition \ref{p:gira_e_rigira}.
We are now ready to detail the main analytical reason behind the estimates claimed in the previous paragraph: the point is that the vanishing of the first variation of $T$ can be translated into a suitable elliptic system for the interpolating functions. However, since the current $T$ is constrained to belong to a given Riemannian submanifold $\Sigma$, the elliptic system necessarily involves only some
components of the maps

\begin{definition}[Tangential parts, {cf. \cite[Definition 5.1]{DS4}}]\label{d:tangential}
Having fixed $H\in \sP^j$ and $\pi:= \pi_H\subset T_{p_H} \Sigma$,
we let $\varkappa$ be the
orthogonal complement of $\pi$ in $T_{p_H} \Sigma$. 
For any given point $q\in \R^{m+n}$, any set $\Omega\subset \pi$ and any map 
$\xi: q+\Omega \to \pi^\perp$, 
 the map $\p_{\varkappa} \circ \xi$ will be called the {\em tangential part of $\xi$} and usually denoted by $\bar{\xi}$.
Analogous notation and terminology will be used for multiple-valued maps.
\end{definition}

\begin{proposition}[Elliptic system, {cf. \cite[Theorem 5.2]{DS4}}]\label{p:pde}
Assume $\eps_2$ is sufficiently small and the assumptions of Proposition \ref{p:gira_e_rigira} hold.
Let $H\in \sW^j\cup \sS^j$ and $L$ be either an ancestor or a cube of
$\sW^j\cap \sS^j$ with $H\cap L\neq \emptyset$
(possibly also $H$ itself). Let $f_{HL}: B_{8r_L} (p_L, \pi_H) \to \Iq (\pi_H^\perp)$ be the 
$\pi_H$-approximation of $T$ in $\bC_{8r_L} (p_L, \pi_H)$, $h_{HL}$ the tilted $(H,L)$-interpolating function
and $\bar{f}_{HL}$ and $\bar{h}_{HL}$ their tangential parts, according to Definition \ref{d:tangential}.
Then, there is a matrix $\bL$, which depends on $\Sigma$ and $H$ but not on $L$, such that
$|\bL|\leq C_0 \bA^2 \leq C_0 \bmo$ for a geometric constant $C_0$ and
(for $C= C(\beta_2, \delta_2, M_0, N_0, C_e, C_h)$) 
\begin{equation}\label{e:pde}
\left| \int \big( D (\etaa \circ \bar f_{HL}) : D \zeta + (\p_\pi (x-p_H))^t \cdot \bL \cdot\zeta\big)\right|
\leq C \bmo \, r_L^{m+1+\beta_2} \big(r_L\,\|\zeta\|_{C^1} + \|\zeta\|_{C^0}\big)\, 
\end{equation}
for every $\zeta\in C^\infty_c (B_{8r_L} (p_L, \pi_L), \varkappa)$.
Moreover (for $C= C(\beta_2, \delta_2, M_0, N_0, C_e, C_h)$)
\begin{equation}\label{e:L1_est}
\|\bar h_{HL} - \etaa \circ \bar f_{HL}\|_{L^1 (B_{7r_L} (p_L, \pi_L))} \leq C \bmo \, r_L^{m+3+\beta_2}\, .
\end{equation}
\end{proposition}

Observe that, if we ignor the effect of the ambient Riemannian manifold, \eqref{e:pde} is an estimate on the $W^{-1,1}$ norm of $\Delta (\etaa\circ f_{HL})$.

\section{The approximation on the normal bundle of the center manifold}\label{s:N}

To carry on our program for proving Theorem \ref{t:cod>1} by ``blow-up'', we now need to approximate again our area minimizing current in a cylinder where the excess is small, taking advantage of the center manifold. Let $\cM = \gr (\phii)$ be the center manifold constructed in the previous sections and let us make some first considerations. 

\medskip

First of all, by the $C^{3,\kappa}$ estimates, we know that in a sufficiently small neighborhood $\bU$ of $\cM$ there is a $C^{2,\kappa}$ orthogonal projection $\p: \bU \to \cM$ which to each $p\in \bU$ assigns the unique point $q = \p (p)$ such that $p-q$ is normal to $T_q \cM$. In fact, since $\|D\phii\|_{C^{2,\kappa}} \leq C \bmo^{\sfrac{1}{2}}$ and $\bmo$ can be chosen arbitrarily small, the ``thickness'' of $\bU$ can be assumed to be of any given size, say $1$. In turn, since the height of $T$ in $\bC$ is of order $\bmo^{\sfrac{1}{2m}}$, we can certainly assume that $\supp (T\res \bC_{1/2} (0))\subset \bU$. It is also not difficult to see that $T$ is a $Q$-fold covering of $\cM$, namely $(\p_\sharp (T\res \bC_{1/2}))\res \bC_{1/4}  = Q \a{\cM\cap \bC_{1/4}}$ (we need to restrict slightly the radius of the cylinder to avoid ``boundary effects''). We could define a ``curvilinear excess'' compared to $\cM$ with the following procedure: at each point $p\in \supp (T)$ we measure the distance between $\vec{T} (p)$ and the oriented tangent plane to $\cM$ at the projection $\p (p)$. We then integrate the square of this quantity over $\supp (T)\cap \bC_{1/2}$. The corresponding formula is
\[
\int_{\bC_{1/2}} |\vec{T} (p) - \vec{T_{\p (p)} \cM}|^2 \, d\|T\| (p)\, .
\]
It is no surprise that the latter is controlled by the ``straight excess'' in the cylinder $\bC$, simply because the tilt between $T_q \cM$ and $\pi_0$ is controlled by $\|D\phii\|_{C^0}$, for which in turn we have the bound $C \bmo^{\sfrac{1}{2}}$. Thus, as it happens with the ``straight excess'' we can expect to be able to approximate $T$ efficiently with a multivalued map $N$ defined on $\cM$ and taking values in the normal bundle of $\cM$. 

\subsection{Preliminaries} In order to state the corresponding theorem precisely we need however to introduce some suitable objects.

\begin{ipotesi}\label{intorno_proiezione}
We fix the following notation and assumptions.
\begin{itemize}
\item[(U)] $\bU :=\big\{x\in \R^{m+n} : \exists !\, y = \p (x) \in \cM \mbox{ with $|x- y| <1$ and
$(x-y)\perp \cM$}\big\}$.
%\footnote{{\color{red} Q: Why not let $\bU$ be the set of points $x\in \mathbb R^{m+n}$ such that {\em either} $x\in \cM$ {\em or} there exists $y\in \cM$ such that $|x-y|<1$ and $\bB_{|x-y| }(x) \cap \cM = \{y\}$.} {\color{blue} A: our definition includes automatically $\cM$ in $\bU$ without distinguishing two cases, provided $\eps_2$ is sufficiently small, see (R). Besides your definition picks more points when $x$ is close to the boundary of $\cM$.}}
\item[(P)] $\p : \bU \to \cM$ is the map defined by (U).
\item[(R)] For any choice of the other parameters, we assume $\eps_2$ to be so small that
$\p$ extends to $C^{2, \kappa}(\bar\bU)$ and
$\p^{-1} (y) = y + \overline{B_1 (0, (T_y \cM)^\perp)}$ for every $y\in \cM$.
\item[(L)] We denote by $\partial_l \bU := \p^{-1} (\de \cM)$ 
the {\em lateral boundary} of $\bU$.
\end{itemize}
\end{ipotesi}

The following is then a corollary of Theorem \ref{t:cm} and the construction algorithm, which summarizes the discussion of the previous paragraph.

\begin{corollary}\label{c:cover}
Under the hypotheses of Theorem \ref{t:cm} and of Assumption \ref{intorno_proiezione}
we have:
\begin{itemize}
\item[(i)] $\supp (\partial (T\res \bU)) \subset \partial_l \bU$, 
$\supp (T\res [-\frac{7}{2}, \frac{7}{2}]^m \times \R^n) \subset \bU$ 
and $\p_\sharp (T\res \bU) = Q \a{\cM}$;
\item[(ii)] $\supp (\langle T, \p, \Phii (q)\rangle) \subset 
\big\{y\, : |\Phii (q)-y|\leq C \bmo^{\sfrac{1}{2m}} 
\ell (L)^{1+\beta_2}\big\}$ for every $q\in L\in \sW$, where\\
$C= C(\beta_2, \delta_2, M_0, N_0,  C_e, C_h)$;
\item[(iii)]  $\langle T, \p, p\rangle = Q \a{p}$ for every $p\in \Phii (\bGam)$.
\end{itemize}
\end{corollary}

We can now introduce multivalued maps on the normal bundle of $\cM$.

\begin{definition}[$\cM$-normal approximation, {cf. \cite[Definition 2.3]{DS4}}]\label{d:app}
An {\em $\cM$-normal approximation} of $T$ is given by a pair $(\cK, F)$ such that
\begin{itemize}
\item[(A1)] $F: \cM\to \Iq (\bU)$ is Lipschitz (with respect to the geodesic distance on $\cM$) and takes the special form 
$F (x) = \sum_i \a{x+N_i (x)}$, with $N_i (x)\perp T_x \cM$ and $x+N_i (x) \in \Sigma$
for every $x$ and $i$.
\item[(A2)] $\cK\subset \cM$ is closed, contains $\Phii \big(\bGam\cap [-\frac{7}{2}, \frac{7}{2}]^m\big)$ and $\bT_F \res \p^{-1} (\cK) = T \res \p^{-1} (\cK)$.
\end{itemize}
The map $N = \sum_i \a{N_i}:\cM \to \Iq (\R^{m+n})$ is {\em the normal part} of $F$.
\end{definition}

\subsection{Main approximation theorem} In the definition above it is not required that the map $F$ approximates efficiently the current
outside the set $\Phii \big(\bGam\cap [-\frac{7}{2}, \frac{7}{2}]^m\big)$. However, we need a map which approximates $T$ with a high degree of accuracy. Moreover, the accuracy must be coupled with the local information coming from the Whitney decomposition. For this region we introduce the ``Whitney regions'' on the center manifold $\cM$.

\begin{definition}[Whitney regions, {cf. \cite[Definition 1.18]{DS4}}]\label{d:cm}
To each $L\in \sW$ we associate a {\em Whitney region} $\cL$ on $\cM$ as follows:
\begin{itemize}
\item[(WR)] $\cL := \Phii (H\cap [-\frac{7}{2},\frac{7}{2}]^m)$, where $H$ is the cube concentric to $L$ with $\ell (H) = \frac{17}{16} \ell (L)$.
\end{itemize}
\end{definition}

\begin{theorem}[Local estimates for the $\cM$-normal approximation, {cf. \cite[Theorem 2.4]{DS4}}]\label{t:approx}
Let $\gamma_2 := \frac{\gamma_1}{4}$, with $\gamma_1$ the constant
of \cite[Theorem 1.4]{DS3}.
Under the hypotheses of Theorem \ref{t:cm} and Assumption~\ref{intorno_proiezione},
if $\eps_2$ is suitably small (depending upon all other parameters), then
there is an $\cM$-normal approximation $(\cK, F)$ such that
the following estimates hold on every Whitney region $\cL$ associated to
a cube $L\in \sW$, with constants $C = C(\beta_2, \delta_2, M_0, N_0, C_e, C_h)$:
\begin{gather}
\Lip (N|
_\cL) \leq C \bmo^{\gamma_2} \ell (L)^{\gamma_2} \quad\mbox{and}\quad  \|N|
_\cL\|_{C^0}\leq C \bmo^{\sfrac{1}{2m}} \ell (L)^{1+\beta_2},\label{e:Lip_regional}\\
|\cL\setminus \cK| + \|\bT_F - T\| (\p^{-1} (\cL)) \leq C \bmo^{1+\gamma_2} \ell (L)^{m+2+\gamma_2},\label{e:err_regional}\\
\int_{\cL} |DN|^2 \leq C \bmo \,\ell (L)^{m+2-2\delta_2}\, .\label{e:Dir_regional}
\end{gather}
Moreover, for any $a>0$ and any Borel $\cV\subset \cL$, we have (for $C=C(\beta_2, \delta_2, M_0, N_0, C_e, C_h)$)
\begin{equation}\label{e:av_region}
\int_\cV |\etaa\circ N| \leq 
C \bmo \left(\ell (L)^{m+3+\sfrac{\beta_2}{3}} + a\,\ell (L)^{2+\sfrac{\gamma_2}{2}}|\cV|\right)  + \frac{C}{a} 
\int_\cV \cG \big(N, Q \a{\etaa\circ N}\big)^{2+\gamma_2}\, .
\end{equation} 
\end{theorem}

Before proceeding, we wish to discuss why the estimate \eqref{e:av_region} is of crucial importance.
Our plan is to show that $N$ is close to a $\D$-minimizing $Q$-valued map and that a large singular set for $T$ induces a large singular set on the latter map. The first point was, in ``straight coordinates'', an effect of the Taylor expansion of the area functional of a graph. Thus, it makes sense to compare the mass of $\bT_F$ with that of $Q\a{\cM}$. 

\begin{theorem}[Expansion of $\mass (\bT_F)$, cf. {\cite[Theorem 3.2]{DS2}}] \label{t:taylor_area}
If $\cM$, $F$ and $N$ are as above and the Lipschitz constant of $N$ is sufficiently small, then
\begin{align}
\mass (\mathbf{T}_F) = {}&  Q \, \cH^m (\cM) - Q\int_\cM \langle H, \etaa\circ N\rangle
+ \frac{1}{2} \int_\cM |D N|^2 + {\rm H. O. T.},  \label{e:taylor_area}
%&\mass (\mathbf{T}_F) =  Q \cH^m (\cM) - Q\int_\cM \langle H, \etaa\circ N\rangle
%+ \frac{1}{2} \int_\cM |D N|^2\nonumber\\
%&\qquad+ \int_\cM \sum_i \Big(P_2 (x,N_i (x)) +  P_3 (x, N_i (x), DN_i (x)) + R_4 (x, DN_i (x))\Big)\, d\cH^m (x)\,,  \label{e:taylor_area}
\end{align}
where
$H$ is the mean curvature vector of $\cM$ and ${\rm H.O.T.}$ contains higher order terms, namely expressions that are at least trilinear in $N$. 
\end{theorem}

Notice in particular that there is a nonvanishing linear term in the expansion. In order to show that $N$ is quasi-harmonic, we therefore need $\langle H, \etaa\circ N\rangle$ to be much smaller than $|DN|^2$. 

\medskip

Before explaining the ideas behind the proof of Theorem \ref{t:approx} we wish to remark that
from \eqref{e:Lip_regional} - \eqref{e:Dir_regional} it is not difficult to infer analogous ``global versions'' of the estimates.

\begin{corollary}[Global estimates, {Cf. \cite[Corollary 2.5]{DS4}}]\label{c:globali} Let $\cM'$ be
the domain\\ $\Phii \big([-\frac{7}{2}, \frac{7}{2}]^m\big)$ and $N$ the map of Theorem \ref{t:approx}. Then,  for $C = C(\beta_2, \delta_2, M_0, N_0, C_e, C_h)$,
\begin{gather}
\Lip (N|_{\cM'}) \leq C \bmo^{\gamma_2} \quad\mbox{and}\quad \|N|_{\cM'}\|_{C^0}
\leq C \bmo^{\sfrac{1}{2m}},\label{e:global_Lip}\\ 
|\cM'\setminus \cK| + \|\bT_F - T\| (\p^{-1} (\cM')) \leq C \bmo^{1+\gamma_2},\label{e:global_masserr}\\
\int_{\cM'} |DN|^2 \leq C \bmo\, .\label{e:global_Dir}
\end{gather}
\end{corollary}

\subsection{A sketch of the proof} Recall first that $\langle T, \p, p\rangle = Q \a{p}$ whenever $p\in \Phii (\bGam)$ and thus we set $N (p) = Q\a{0}$ for any such $p$: it remains therefore to define $N$ on each region of the form $\Phii (L)$ when $L\in \sW$. 

Let us fix $L\in \sW$.
Going back to the construction of the center manifold, we discover that we already have a rather accurate graphical approximation of $T$ in the region of our interest, since we already considered the $\pi_L$ approximation $f_L$. If $f_L$ were a classical single valued function, we could simply parametrize its graph on the normal bundle of $\cM$. Indeed the tangent planes to the graph of $f_L$ are certainly close to $\pi_L$ and on the other hand $\cM$ is constructed patching a suitable smoothing of the average $\etaa \circ f_L$. Thus, in the cylinder $\bC_{6r_L} (p_L, \pi_L)$, the angle between a generic tangent plane to the graph of $f_L$ and a generic tangent plane to $\cM$ is rather small. In the $Q$-valued setting ``reparametrizing'' graphs is a much more subtle issue than in the classical single valued setting . However it is not very hard to prove a suitable theorem (see \cite[Theorem 5.1]{DS2}) which allows to describe the graph of $f_L$ through a Lipschitz map $N_L$ on the normal bundle of $\cM$. Note moreover that $N_L$ can be defined on a much larger domain than $\Phii (L)$. 

\medskip

We next have to face a new difficulty: if $L$ and $H$ are two nearby cubes, the maps $N_L$ and $N_H$ do not necessarily agree on the intersection of their domains.
However, recall that the graphs of $f_L$ and $f_H$ coincide with the current $T$ except for two sets of small measure. Thus the values of $N_L$ and $N_H$ coincide on a very large portion of the set on which both are defined. In turn, the construction algorithm ensures that each $H$ intersects only a finite number of other cubes in $\sW$: such number is bounded a-priori by a geometric constant (for instance, when $m=2$ each square of $\sW$ can intersect at most 12 other squares of $\sW$). So, after removing from each $\Phii (L)$ all those points where $N_L$ does not coincide with all the $N_H$ related to neighboring cubes, we reach a uniquely defined map $N$ on a rather large subset of $\cM\setminus \Phii (\bGam)$. 

\medskip

We now wish to extend this map to a Lipschitz one defined on all $\cM$. It is not difficult to see that $N$ is already {\em globally} Lipschitz on its domain of definition and that it approaches the value $Q\a{0}$ on sequences converging to $\Phii (\bGam)$: this is because the height of $f_L$ (and thus that of $N_L$) is controlled by $C \ell (L)^{1+\beta_2}$  and cubes in $\sW$ close to $\bGam$ necessarily have small sidelength.  However, it does not serve our purposes to give a {\em global} Lipschitz extension of $N$ which does not respect the local properties of the map. In particular we desire an extension that on each Whitney region $\cL$ has small Lipschitz constant: the smaller the scale, the smaller should be the Lipschitz constant. To achieve this property we follow an elementary idea, which we next describe in the special case of dimension $m=2$. As a matter of fact, since $\Phii$ is Lipschitz we can, for the sake of our discussion, assume that $\Phii (L)$ is flat and coincides with the square $L$. We fix the four points $A, B, C$ and $D$ which are the four vertices of the square. We first wish to extend $N$ to four small neighborhoods of these points (if $N$ is not already defined there); we will call the neighborhoods $U (A), U(B), U (C)$ and $U (D)$ and we will fix them to be disks centered at the respective points with radius $\ell (L)/4$. 

\medskip

Take for instance $A$: the latter point is a common vertex for four squares $H_1, H_2, H_3, L$ of $\sW$. We then consider the intersection of the domain of $N$ with $H_1\cup H_2 \cup H_3\cup L$: if we restrict the map $N$ to this set and consider its Lipschitz constant, we can then use the Lipschitz extension theorem for multivalued functions to extend it to a neighborhood $U(A)$ of $A$, without increasing such constant by much. We proceed and extend the map separately to neighborhoods of $A$, $B$, $C$ and $D$. However when we extend the map to the neighborhood $U(B)$ we disregard what we did in the neighborhood of $A$ and we only take into consideration the ``original'' $N$: having chosen such neighborhoods sufficiently small the distance between two points $p$ and $q$ which belong to distinct neighborhoods (say $U(A)$ and $U(B)$) is larger than $\ell (L)/2$ and the height of the extension is no larger than $C \ell (L)^{1+\beta_2}$, thus providing automatically a good Lipschitz bound on $U(A)\cup U(B)$. This procedure can then be repeated for all squares and we have a new map $N'$ which is evidently defined in a neighborhood of any vertex of any $L\in \sW$.

With the same principle we extend $N'$ to neighborhoods $U(AB), U (BC), U (CD), U(DA)$ of the corresponding sides of the square $L$. This seems more problematic because, for instance, $U (AB)$ and $U(BC)$ intersect in a neighborhood of $B$. However $N'$ is already defined on $U(A)\cup U(B)\cup U(C)\cup U(D)$. Thus we need to extend it to the sets $V(AB) = U(AB) \setminus (U(A) \cup U(B))$ and $V(BC) = U(BC)\setminus (U(B)\cup U(C))$. The latter can be now assumed to be disjoint and separated by a distance of the order $c_0 \ell (L)$ for some $c_0>0$: it just suffices to choose the thickness of the neighborhoods $U(AB)$ and $U (BC)$ much smaller than the thickness of the neighborhoods $U(A)$, $U (B)$ and $U(C)$. We can then literally argue as above and gain a second extension of the map, say $N''$, which is defined on the boundary of any square $L$ of $\sW$. At this point the third (and final) extension is achieved by considering each square separately.

\section{Blow-up I: preliminary considerations and frequency function}

We are now ready to discuss the proof of Theorem \ref{t:cod>1}. We assume by contradiction that the theorem is false and, recalling Proposition \ref{p:successione}, we fix an area minimizing current $T$ of dimension $m$, a plane $\pi_0$ (which without loss of generality we assume to be $\R^m\times \{0\}$) an integer $Q>0$ and a sequence of radii $r_k\downarrow 0$ with the following properties:
\begin{itemize}
\item The excess $\bE (T_{0, r_k}, \bC_8 (\pi_0, 0))$ converges to $0$ as $k\uparrow \infty$;
\item The point $0$ is singular, $\Theta (T, 0)=Q$ and $\cH^{m-2 +\alpha}_\infty (\bB_1 \cap \sing_Q (T_{0, r_k}))\geq \eta>0$.
\end{itemize}
We spend the next paragraph to make our discussion above precise. In particular from now on we will assume the following hypotheses:

\begin{ipotesi}\label{ipotesi_principale}
Let $\eps_0\in ]0,1[$, $m,\bar{n}\in \N\setminus \{0\}$ and $l\in \N$. We denote by
\begin{itemize}
\item[(M)] $\Sigma\subset\R^{m+n} = \R^{m+\bar{n}+l}$ an embedded $(m+\bar{n})$-dimensional submanifold of class $C^{3,\eps_0}$;
\item[(C)] $T$ an integral current of dimension $m$ with compact support $\supp (T) \subset \Sigma$,
area minimizing in $\Sigma$.
\end{itemize}
\end{ipotesi}

\begin{ipotesi}[Contradiction]\label{assurdo}
There exist $m \geq 2, \bar{n}, l$, $\Sigma$ and $T$  as in Assumption \ref{ipotesi_principale} such that 
$\cH^{m-2+\alpha}(\sing(T)) >0$ for some $\alpha>0$. 
\end{ipotesi}

Moreover, let us recall the following subsets of $\reg (T)$ and $\sing (T)$, introduced a few sections ago.

\begin{definition}[$Q$-points] \label{d:regular_points}
For $Q\in \N$, we denote by $\rD_Q(T)$ the points of density $Q$ of the current $T$,
and set
\begin{gather*}
\reg_Q (T) := \reg (T)\cap \rD_Q (T) \quad\text{and}\quad
\sing_Q (T) := \sing (T)\cap \rD_Q (T).
\end{gather*}
\end{definition}

The first key player in the proof of Theorem \ref{t:cod>1} is then the following sequence of rescalings of the current $T$ in Assumption \ref{assurdo}.

\begin{proposition}[Contradiction sequence, {cf. \cite[Proposition 1.3]{DS5}}]\label{p:seq}
Under Assumption \ref{assurdo}, there are $m,n, Q \geq 2$, $\Sigma$ and $T$ as in Assumption
\ref{ipotesi_principale},
reals $\alpha,\eta>0$, and a sequence $r_k\downarrow 0$ such that
$0\in \rD_Q (T)$ and the following holds:
\begin{gather}
\lim_{k\to+\infty}\bE(T_{0,r_k}, \B_{6\sqrt{m}}) = 0,\label{e:seq1}\\
\lim_{k\to+\infty} \cH^{m-2+\alpha}_\infty (\rD_Q (T_{0,r_k}) \cap \B_1) > \eta,\label{e:seq2}\\
\cH^m \big((\B_1\cap \supp (T_{0, r_k}))\setminus \rD_Q (T_{0,r_k})\big) > 0\quad \forall\; k\in\N.\label{e:seq3}
\end{gather}
\end{proposition}

In particular, by Proposition~\ref{p:seq} and  simple rescaling arguments,
we assume in the sequel the following.

\begin{ipotesi}\label{i:H'} Let $\eps_3\in ]0, \eps_2[$. Under Assumption~\ref{assurdo}, there exist  $m, n, Q \geq 2$, $\alpha, \eta > 0$,  $T$ and $\Sigma$ for which:
\begin{itemize}
\item[(a)] there is a sequence of radii 
$r_k\downarrow 0$ as in Proposition \ref{p:seq};
\item[(b)] the following holds:
\begin{gather}
T_0\Sigma = \R^{m+\bar{n}} \times \{0\}, \quad 
\supp (\partial T)\cap \B_{6\sqrt{m}} = \emptyset,\quad 0 \in D_Q(T),\label{e:aggiuntive}\\
\|T\| (\B_{6\sqrt{m} r}) \leq
r^m \left( Q\,\omega_m (6\sqrt{m})^m + \eps_3^2\right) \quad \text{for all } r\in (0,1),\label{e:(2.1)-del-cm}\\
 \mathbf{c} (\Sigma\cap \bB_{7\sqrt{m}}) \leq \eps_3\, .\label{e:piattezza_sigma}
\end{gather}
\end{itemize}
\end{ipotesi}

\subsection{A sequence of center manifolds} We wish to approximate the current with an $\cM$-normal approximation $N$ over a center manifold $\cM$. A first attempt could be the following: for some $r$ sufficiently large, the excess $\bE (T_{0,r}, \bC_r (\pi_0, 0))$ will be sufficiently small and $\Sigma_{0,r} := \iota_{0,r} (\Sigma)$ will be sufficiently flat. Assuming without loss of generality that $r=1$, we can then construct a center manifold $\cM$ and an $\cM$-normal approximation. However, we have no guarantee that this approximation is accurate at very small scales around $0$. This would certainly be the case if $0$ belonged to the contact set $\bGam$ of Definition  \ref{d:cm}, but of course it might be the case that $0$ belongs to some cube $L\in \sW$ where  the refining procedure has stopped. If this is the case at a certain small scale around $0$ the graph of $N$ might have a completely different behavior than $T$. 

\medskip

If we set $t_1 := 1$, we can then distinguish two situations. In the first one $0\in \bGam$ and we are thus satisfied with our center manifold and our approximating map, which we denote by $\cM_1$ and $N_1$. In the second one at some distance $s_1$ from the origin we encounter the first cube of size sufficiently large compared to $s_1$, say $\bar{c}_s s_1$ for some specified constant $\bar{c}_s$ (this number turns out to be $1/(16\sqrt{m})$ in our proof, cf. \cite[Section 2]{DS5}: its precise value is however not important, as long as it is a fixed parameter which does not depend upon those which enter the construction of the center manifold). We can assume that $s_1$ is (much) smaller than $t_1$: we just need to set the parameter $N_0$ in the construction of the center manifold accordingly large.

At the scale $s_1$ the two objects $\cM_1$ and $N_1$, i.e. the center manifold and the corresponding approximation, are not anymore serving our purposes. Now consider $T_{0, s_1}$: the latter current might or might not have sufficiently small excess to construct another center manifold. In the first case we set $t_2 = s_1$, otherwise we let $t_2$ be the first radius smaller than $s_1$ at which the excess goes below the desired threshold. We then construct the pair $\cM_2$ and $N_2$ taking $T_{0, t_2}$ as reference area minimizing current.

We formalize the above discussion in the following way. We set
\begin{equation}\label{e:def_R}
\mathcal{R}:=\big\{r\in ]0,1]: \bE (T, \B_{6\sqrt{m} r}) \leq \eps_3^2\big\}\,.
\end{equation}
Observe that, if $\{s_k\}\subset \mathcal{R}$
and $s_k\uparrow s$, then $s\in \mathcal{R}$.
We cover $\cR$ with a collection $\mathcal{F}=\{I_j\}_j$ of intervals
$I_j = ]s_j, t_j]$ defined as follows.
$t_0:= \max \{t: t\in \mathcal{R}\}$. Next assume, by induction, to have defined $t_j$ (and hence also
$t_0 > s_0\geq t_1 > s_1 \geq \ldots > s_{j-1}\geq t_j$) and consider the following objects:
\begin{itemize}
\item[-] $T_j := ((\iota_{0,t_{j}})_\sharp T)\res \bB_{6\sqrt{m}}$, $\Sigma_j := \iota_{0, t_j} (\Sigma) \cap \bB_{7\sqrt{m}}$; moreover, consider
for each $j$ an orthonormal system of coordinates so that, if we denote by $\pi_0$ the $m$-plane $\mathbb R^m\times \{0\}$, then $\bE (T_j, \bB_{6\sqrt{m}}, \pi_0) = \bE(T_j,\B_{6\sqrt{m}})$ (alternatively we can keep the system of coordinates fixed
and rotate the currents $T_j$).
\item[-] Let $\cM_j$ be the corresponding center manifold constructed
in Theorem \ref{t:cm} applied to $T_j$ and $\Sigma_j $ with respect to the $m$-plane $\pi_0$;
the manifold $\cM_j$ is then the graph of a map $\phii_j: \pi_0 \supset [-4,4]^m \to \pi_0^\perp$,
and we set $\Phii_j (x) := (x, \phii_j (x)) \in \pi_0\times \pi_0^\perp$.
\end{itemize}
Then, we consider the Whitney decomposition $\sW^{(j)}$ of $[-4,4]^m \subset \pi_0$ as in  Definition \ref{d:refining_procedure} (applied to $T_j$) and we define
\begin{equation}\label{e:s_j}
s_j := \max\, \left(\{c_s^{-1} \ell (L) : L\in \sW^{(j)} \mbox{ and } c_s^{-1} \ell (L) \geq \dist (0, L)\} \cup \{0\} \right)\, .
\end{equation} 
It can be shown, see below, that $s_j/t_j <2^{-5}$. In particular this ensures that $[s_j, t_j]$ is a (nontrivial) interval. 
Next, if $s_j =0$ we stop the induction. Otherwise we let $t_{j+1}$ be the largest element
in $\mathcal{R}\cap ]0, s_j]$ and proceed as above. Note moreover the following simple consequence of \eqref{e:s_j}:
\begin{itemize}
\item[(Stop)] If $s_j >0$ and $\bar{r} := s_j/t_j$, then there is $L\in \sW^{(j)}$ with 
\begin{equation}\label{e:st}
\ell(L) = c_s\,\bar r\, \qquad \mbox{and} \qquad L\cap \bar B_{\bar r} (0, \pi_0)\neq \emptyset\, 
\end{equation}
(in what follows $B_r (p, \pi)$ and $\bar B_r (p,\pi)$ will denote the open and closed disks 
$\bB_r (p)\cap (p+\pi)$, $\bar{\bB}_r (p) \cap (p+\pi)$);
\item[(Go)] If $\rho > \bar{r} := s_j/t_j$, then
\begin{equation}\label{e:go}
\ell (L) < c_s \rho \qquad \mbox{for all } L\in \sW^{j(k)} \mbox{ with $L\cap B_\rho (0, \pi_0) \neq \emptyset$.}
\end{equation}
In particular the latter inequality is true for every $\rho\in ]0,3]$ if $s_j =0$.
\end{itemize}

The following is a list of easy consequences of the definition. Given two sets $A$ and $B$, we define their
{\em separation} as the number ${\rm sep} (A,B) := \inf \{|x-y|:x\in A, y\in B\}$. Moreover we will use
the notation $\cB_r (p)$ for the geodesic balls in $\cM$.

\begin{proposition}\label{p:flattening}
Assuming $\eps_3$ sufficiently small, then the following holds:
\begin{itemize}
\item[(i)] $s_j < \frac{t_j}{2^5}$ and the family $\mathcal{F}$ is either countable
and $t_j\downarrow 0$, or finite and $I_j = ]0, t_j]$ for the largest $j$;
\item[(ii)] the union of the intervals of $\cF$ cover $\cR$,
and for $k$ large enough the radii $r_k$ in Assumption \ref{i:H'} belong to $\cR$;
\item[(iii)] if $r \in ]\frac{s_j}{t_j},3[$ and $J\in \mathscr{W}^{(j)}_n$ intersects
$B:= \p_{\pi_0} (\cB_r (p_j))$, with $p_j := \Phii_j(0)$,
then $J$ is in the
domain of influence $\mathscr{W}_n^{(j)} (H)$ (see \cite[Definition~3.3]{DS4})
of a cube $H\in \mathscr{W}^{(j)}_e$ with
\[
\ell (H)\leq 3 \, c_s\, r \quad \text{and}\quad 
\max\left\{{\rm sep}\, (H, B),  {\rm sep}\, (H, J)\right\}
\leq 3\sqrt{m}\, \ell (H) \leq \frac{3 r}{16};
\]
\item[(iv)] $\bE (T_j, \B_r )\leq C_0 \eps_3^2 \, r^{2-2\delta_2}$ for
every $r\in]\frac{s_j}{t_j},3[$.
\item[(v)] $\sup \{ \dist (x,\cM_j): x\in \supp(T_j) \cap \p^{-1}_j(\cB_r(p_j))\} \leq C_0\, (\bmo^j)^\frac{1}{2m} r^{1+\beta_2}$ for
every $r\in]\frac{s_j}{t_j},3[$, where
$\bmo^j := \max\{\mathbf{c}(\Sigma_j)^2 , \bE(T_j, \bB_{6\sqrt{m}})\}$. 
\end{itemize}
\end{proposition}

\subsection{Frequency function} For each interval of flattening we have a center manifold $\cM_j$ and an approximating map $N_j$. Since each $r_k$ in Proposition \ref{p:seq} belongs to some interval of flattening $[s_{j(k)}, t_{j(k)}]$, we now hope to find suitable rescalings of the maps $N_{j(k)}$ which converge to a $\D$-minimizing map and which inherits the singularities of the currents $T_{0, r_k}$, thus contradicting the regularity theory for $\D$-minimizing functions. In other words, we hope to succeed in the very same program which failed in Section \ref{s:tentativo}. The reason is that the $N_j$ are now well-centered. In particular we claim that
we have a frequency function estimate for the $N_j$'s, similar to the estimate for $\D$-minimizing maps.

Consider the following Lipschitz (piecewise linear) function $\phi:[0+\infty[ \to [0,1]$ given by
\begin{equation*}%\label{e:def_phi}
\phi (r) :=
\begin{cases}
1 & \text{for }\, r\in [0,\textstyle{\frac{1}{2}}],\\
2-2r & \text{for }\, r\in \,\, ]\textstyle{\frac{1}{2}},1],\\
0 & \text{for }\, r\in \,\, ]1,+\infty[.
\end{cases}
\end{equation*}
For every interval of flattening $I_j = ]s_j, t_j]$,
let $N_j$ be the normal approximation of $T_j$
on $\cM_j$ in \cite[Theorem~2.4]{DS4}.

\begin{definition}[Frequency functions]\label{d:frequency}
For every $r\in ]0,3]$ we define: 
\[
\bD_j (r) := \int_{\cM^j} \phi\left(\frac{d_j(p)}{r}
\right)\,|D N_j|^2(p)\, dp\quad\mbox{and}\quad
\bH_j (r) := - \int_{\cM^j} \phi'\left(\frac{d_j (p)}{r}\right)\,\frac{|N_j|^2(p)}{d_j(p)}\, dp\, ,
\]
where $d_j (p)$ is the geodesic distance on $\cM_j$ between $p$ and $\Phii_j (0)$.
If $\bH_j (r) > 0$, we define the {\em frequency function}
$\bI_j (r) := \frac{r\,\bD_j (r)}{\bH_j (r)}$. 
\end{definition}

The following is the main analytical estimate of the paper, which allows us to
exclude infinite order of contact among the different sheets of a minimizing
current.

\begin{theorem}[Main frequency estimate, {cf. \cite[Theorem 3.2]{DS5}}]\label{t:frequency}
If $\eps_3$ is sufficiently small, then
there exists a geometric constant $C_0$ such that, for every
$[a,b]\subset [\frac{s_j}{t_j}, 3]$ with $\bH_j \vert_{[a,b]} >0$, we have
\begin{equation}\label{e:frequency}
\bI_j (a) \leq C_0 (1 + \bI_j (b)).
\end{equation}
\end{theorem}

Recalling the proofs of Theorem \ref{t:frequency} and Theorem \ref{t:frequency2}, the monotonicity of the corresponding quantity for $\D$-minimizing functions rely on two identities which correspond to ``internal'' and ``external'' variations, cf. Proposition \ref{p:import}. In turn both variations can be interpreted as suitable deformations of the graphs of the functions. We then know that:
\begin{itemize}
\item The variations of the Dirichlet energy for $N_j$ is close to the variation of the mass of its graph;
\item The variation of the mass of its graph is close to the variation of the mass of $T_{0, t_j}$, which in turn is $0$ because it is area minimizing (this must be suitably adjusted if the minimizing property is inside some ambient manifold $\Sigma$).
\end{itemize}
We thus can write down identities which correspond to those of Proposition \ref{p:import} but will include several error terms. In particular, quite careful computations give the following formulae, where, for the moment, we drop the subscripts $j$.

\begin{theorem}[Expansion of outer variations, {cf. \cite[Theorem 4.2]{DS2}}]\label{t:outer}
If $\varphi\in C^1_c (\cM)$ and $X(p) := 
\varphi (\mathbf{p} (p)) (p-\mathbf{p}(p))$, then
\begin{align}
\delta \mathbf{T}_F (X) = \int_\cM \Big(\varphi \, |DN|^2 + 
\sum_i (N_i \otimes D \varphi) : DN_i\Big) - \underbrace{Q \int_\cM \varphi \langle H, \etaa\circ N\rangle}_{{\rm Err}_1} + \sum_{i=2}^3{\rm Err}_i
% + {\rm Err}_2 + {\rm Err}_3\,
\label{e:outer} 
\end{align}
where
\begin{gather}
|{\rm Err}_2| \leq C \int_\cM |\varphi| |A|^2|N|^2\label{e:outer_resto_2}\\
|{\rm Err}_3|\leq C \int_\cM \Big(|\varphi| \big(|DN|^2 |N| |A| + |DN|^4\big) +
|D\varphi| \big(|DN|^3 |N| + |DN| |N|^2 |A|\big)\Big)\label{e:outer_resto_3}\, .
\end{gather}
\end{theorem}

Let $Y$ be a $C^1$ vector field on $T\cM$ with compact
support and define $X$ on $\bU$ setting $X (p) = Y (\p (p))$. Let $\{\Psi_\eps\}_{\eps \in ]-\eta, \eta[}$ be
any isotopy with $\Psi_0 = {\rm id}$ and $\left.\frac{d}{d\eps}\right\vert_{\eps=0}
\Psi_\eps = Y$ and define the following
isotopy of $\mathbf{U}$: $
\Phi_\eps (p) = \Psi_\eps (\mathbf{p} (p)) + (p-\mathbf{p} (p))$.
Clearly $X = \left.\frac{d}{d\eps}\right|_{\eps =0} \Phi_\eps$. 

\begin{theorem}[Expansion of inner variations, {cf. \cite[Theorem 4.3]{DS2}}]\label{t:inner}
If $X$ is as above, then
\begin{align}
\delta \mathbf{T}_F (X) = \int_\cM \Big( \frac{|DN|^2}{2} {\rm div}_{\cM}\, Y -
\sum_i  D N_i : ( DN_i\cdot D_{\cM} Y)\Big)+ \sum_{i=1}^3{\rm Err}_i,\label{e:inner} 
\end{align}
where
\begin{gather}
{\rm Err}_1 = - Q \int_{ \cM}\big( \langle H, \etaa \circ N\rangle\, {\rm div}_{\cM} Y + \langle D_Y H, \etaa\circ N\rangle\big)\, ,\label{e:inner_resto_1}\allowdisplaybreaks\\
|{\rm Err}_2| \leq C \int_\cM |A|^2 \left(|DY| |N|^2  +|Y| |N|\, |DN|\right), \label{e:inner_resto_2}\allowdisplaybreaks\\
|{\rm Err}_3|\leq C \int_\cM \Big( |Y| |A| |DN|^2 \big(|N| + |DN|\big) + |DY| \big(|A|\,|N|^2 |DN| + |DN|^4\big)\Big)\label{e:inner_resto_3}\, .
\end{gather}
\end{theorem}

\section{Blow-up II: Splitting before tilting and bound on the frequency function}

In order to prove Theorem \ref{t:frequency} we hope to achieve two important conclusions from the center manifold
construction:
\begin{itemize}
\item[(A)] Each time we start one of the center manifolds $\cM_j$, the initial value of the frequency function, say $\bI_j (1)$, is bounded
a-priori independently of $j$;
\item[(B)] After applying the first variations, with the help of Theorem \ref{t:outer} and Theorem \ref{t:inner} we hope to derive an inequality of the form
\[
\frac{d}{dr} \bI_j (r) \geq - E(r) \bI_j (r) \, ,
\]
which in turn gives the monotonicity of $\bI_j (r) e^{\int_0^r E(\tau) d\tau}$.
\end{itemize}
Both points are the effect of a phenomenon which we call ``splitting before tilting'' inspired by Rivi\`ere's important paper \cite{Riviere}.

\subsection{The splitting before tilting phenomenon}\label{ss:splitting} Fix a center manifold $\cM$ (we drop the subscript to simplify the notation) and consider again the algorithm defining the corresponding Whitney decomposition, as in Definition \ref{d:refining_procedure}. Ignoring for the moment the special cubes $L$ which stopped because of the condition (NN), we must remember
that all other cubes must have stopped for one of the following two very precise reasons:
\begin{itemize}
\item[(h)] either the height in $\bB_L$ exceeds $C_h \bmo^{\sfrac{1}{2m}} \ell (L)^{1+\beta_2}$;
\item[(e)] or the excess in $\bB_L$ exceeds $C_e \bmo \ell (L)^{2-2\delta_2}$.
\end{itemize}
In the first occasion we use an appropriate ``layered version'' of the classical height bound for currents with small excess.
In order to avoid too many technicalities, we state the bound in a rather informal way.

\begin{theorem}[Layered height bound, {cf. \cite[Theorem A.1]{DS4}}]\label{t:height}
Let $T$ satisfy Assumption \ref{ipotesi_base_app} in the cylinder $\bC_{4r} (0, \pi_0)$ and let $E := \bE (T, \bC_{4r} (0, \pi_0)$. In a slightly smaller cylinder $\bC_\rho (0, \pi_0)$ the current is then supported in $k\leq Q$ disjoint layers of the form 
$B_\rho (0, \pi_0) \times B_{C \rho E^{\sfrac{1}{2m}}} (y_i, \pi_0^\perp)$: in each layer the 
density $\Theta (\cdot, T)$ does not exceed a certain integer $Q_i$ by much and $\sum_i Q_i = Q$. 
\end{theorem}

In particular, if $C_h$ is relatively large compared to $C_e$, at any cube $L$ where the refining procedure has stopped by the (HT) condition the current will form at least {\em two} separate ``layers'' of thickness smaller than $\ell (L)^{1+ (1-\delta_2)/m}$ and parallel to $\pi_L$. The two layers are separated by a distance comparable to the height in $\bB_L$, say  $\frac{1}{2} C_h \bmo^{\sfrac{1}{2m}} \ell (L)^{1+\beta_2}$. Thus, {\em anywhere} in a region of  diameter $C \ell (L)$ which includes $\cL$ (where $C$ will depend on the constant $M_0$) we can expect $|N|$ to be at least as large as $\frac{1}{4} C_h \bmo^{\sfrac{1}{2m}} \ell (L)^{1+\beta_2}$.
Thus, on every region $\Omega$ of measure $c \ell (L)^m$ close to $\cL$ the size of $\int_\Omega |N|^2$ is at least $C \ell (L)^{m+2+2\beta_2}$. We refer to reader to \cite[Proposition 3.1]{DS4} (see also \cite[Section 5.5]{Ema}) for the precise statement.

\medskip

In the second case we would like to say that $\int_\Omega |DN|^2$ is at least $C \ell (L)^{m+2-2\delta_2}$ for any region $\Omega$ which is a ball in $\cM$ of radius $c \ell (L)$, sufficiently close to $\cL$. This is true but much more subtle. 

Recall that, if $H$ is the father of $L$, the excess in $\bB_H$ is smaller than $C_e \bmo \ell (H)^{2-2\delta_2}$. If the parameter $M_0$ is chosen very large, $\bB_L$ and $\bB_H$ are almost concentric and the radius of $\bB_L$ is twice the radius of $\bB_H$. On the other hand we know that, in $\bB_L$, the current $T$ can be approximated by a $\D$-minimizing $Q$-valued map. If the latter were too close to a multiple copy (with multiplicity $Q$) of a classical single valued harmonic function, then the argument illustrated in Section \ref{ss:DG} tells us that the excess in $\bB_L$ should be almost $\frac{1}{4}$ of that in $\bB_H$, cf. Remark \ref{r:alfa=1}. But this is not the case because the ratio is instead at least $2^{-2+2\delta_2}$: although small, the parameter $\delta_2$ makes here a big difference! Thus in $\bB_L$ $T$ is close to a ``non-classical'' $\D$-minimizer $f= \sum_i \a{f_i}$, more precisely we can certainly assume that the Dirichlet energy of $\bar{f} = \sum_i \a{f_i - \etaa\circ f_i}$ is not too small compared to that of $f$. For such maps we have the important property that in no region of their domain of definition can the energy become too small. Passing now to the ``curvilinear coordinates'' we can infer the same conclusion for 
$\int_\Omega |DN|^2$ whenever $\Omega$ is a region of a suitable size sufficiently close to $\cL$. For the precise statement we refer the reader to \cite[Proposition 3.4]{DS4} (see also \cite[Section 5.5]{Ema}).

\medskip

Finally, it is necessary to achieve some contol upon the behavior of $N$ in Whitney regions related to cubes where the refining procedure stopped because of the (NN) conditions. However, it is not difficult to see that these cubes can be partitioned into families, each family being sufficiently close to some cube of type (HH) or (EX). In fact it turns out that only the second condition might really occur, cf. \cite[Corollary 3.2]{DS4}. The idea is simple: if $L$ is an (NN) stopping cube, then it has a neighbor twice as big where the refining procedure stopped as well. We can iterate this consideration until we find a cube of type (HT) or (EX). This cube might be of type (EX): 
the separation condition of type (HT) propagates to nearby cubes of half size by the layered height bound of Theorem \ref{t:height}.
Consider now the union $\cU$ of the Whitney regions relative to all the cubes (NN) which lead to the same cube $L$ of type (EX) via the above procedure. Then the size of $\cU$ is comparable to that of the Whitney region $\cL$ and thus we can bound all the error terms in $\cU$ with the Dirichlet energy of $N$ in $\cL$.

\subsection{Uniform bound on $\bI_j (1)$}

If the intervals of flattening are finite, then there would be nothing to prove. What happens however if the intervals are not finite? Consider some ``starting point'' $t_j$. It essentially can be of two kinds: 
\begin{itemize}
\item $t_j = s_{j-1}$; in this case the center manifold $\cM_{j-1}$ had ceased to be ``good'' at scale $s_{j-1}$ and we encounter a cube $L\in \sW^{j-1}$ which is of size $s_{j-1}/t_{j-1}$ where the refining procedure has stopped. It is the first scale at which we encounter a large cube close to $0$: for this very precise reason it is not difficult to show that it must necessarily be a cube where either the excess or the height condition fails, cf. Proposition \ref{p:flattening}. But then the discussion in Section
\ref{ss:splitting} implies that the approximating map $N_{j-1}$ is well separated at this scale: this means that the denominator of the frequency function $\bI_{j-1}$ is not too small compared to the numerator. It remains to transfer this information to the frequency function $\bI_j$: note however that at the scale which interests us there is a large overlap between the graphs of (suitable rescalings of) $N_j$ and $N_{j-1}$ since they both approximate the same piece of the same current.
\item $t_j < s_{j-1}$; in this case the excess of $T_{0, s_j}$ is not small enough to construct a center manifold. $t_j$ is then the first radius at which the excess goes below the desired threshold. As such we also know that $\bE (T, \bB_{t_j})$ cannot be {\em too} small, i.e. the current cannot be too close to a flat plain at the scale $t_j$. But then we can expect that the denominator $\bH_j (1)$ is not too small compared to the numerator $\bD_j (1)$.
\end{itemize}
This discussion can be made rigorous and we refer the reader to \cite[Theorem 5.1]{DS5} for the details.

\section{Blow-up III: final arguments}

We have now gained a sequence of center manifolds $\cM_{l(k)}$ and of approximations $N_{l(k)}$, so that $r_k$ belongs to $]s_{l(k)}, t_{l(k)}]$, where $\{r_k\}$ is the sequence of Proposition \ref{p:seq}. Let us rescale the center manifolds by a factor $r_k/t_{l (k)}$, so to gain a sequence of manifolds $\bar\cM_k$ and maps $\bar N_k$ at ``scale 1'', which should give rather good approximations of $T_k := T_{0, r_k}$.  We fix for convenience a ``central point'' for each center manifold $\cM_l$: it will be the point $p_l$ lying in the plane $\{0\}\times \R^n$. Correspondingly $\bar{p}_k := p_{l(k)} t_{l(k)}/ r_k$ is the central point of $\bar\cM_k$. The geodesic balls with center $q$ and radius $\rho$ will be denoted by $\cB_\rho (q)$ in any of these manifolds.
We next normalize further the maps $\bar N_{k}$ dividing them by their ``$L^2$ norm'', namely by
\begin{equation}\label{e:aggiunta1}
\bh _k := \left(\int_{\cB_1 (\bar p_{k})} |\bar N_k|^2\right)^{\sfrac{1}{2}}\, ,
\end{equation}
i.e. we set
\begin{equation}\label{e:aggiunta2}
N^b_k := \frac{\bar N_k}{\bh_k}\, .
\end{equation}
Observe that $\bh_k$
 must be positive: by the discussion in Section \ref{ss:splitting}, it can vanish only if $\sW^{l(k)}$
is empty. However in the latter case $T_{0, r_k}$ would coincide with $Q$ copies of the (smooth) manifold $\bar\cM_{k}$, which cannot be the case because $0$ is a singular point.

\medskip

We wish to take a limit for (a suitable subsequence of) $N^b_k$: since $\bar \cM_k$ ``flattens'' (i.e. converges to $\pi_0$) we hope that the limit $N^b_\infty$ is a $Q$-valued map which has $L^2$ norm equal to $1$, because the convergence is strong in $L^2$. Such strong convergence will be achieved if we could prove that the rescaled 
maps have bounded Dirichlet energy, namely if we had an inequality of type
\begin{equation}\label{e:aggiunta3}
\int_{\cB_1 (\bar p_{k})} |DN^b_k|^2 \leq C\, ,
\end{equation}
for some constant $C$ which does not depend on $k$. In turn this bound corresponds to a ``reverse Sobolev'' inequality for the $N_{l(k)}$, i.e.
\[
\int_{\cB_r (p_{l(k)})} |D N_{l(k)}|^2 \leq C r^{-2} \int_{\cB_r (p_{l(k)})} |N_{l(k)}|^2\, .
\]
On the other hand this is precisely the bound that a bound on the frequency function as in Theorem \ref{t:frequency} is expected to yield.
For the details we refer the reader to \cite[Corollary 5.3]{DS5}.

\subsection{The final map $N^b_\infty$} We have now finally gained our ``blown-up'' map $N^b_\infty$, which is the limit (up to subsequences) of the maps $N^b_k$, and we know that it is a $Q$-valued map on $B_1 (0, \pi_0)$. Note that the estimate \eqref{e:av_region} (and the lower bounds discussed in Section \ref{ss:splitting} deliver the extra information that $\etaa\circ N^b_\infty \equiv 0$. This also helps us in the Taylor expansion of the area functional
to conclude that $N^b_\infty$ is $\D$-minimizing, cf. Theorem \ref{t:taylor_area} (we are of course ignoring the complications given by the ambient Riemannian manifold $\Sigma$). We now wish to succeed where the strategy outlined in Section \ref{s:tentativo} failed, i.e. in showing that the blown-up map $N^b_\infty$ must ``remember'' the singularities of the rescaled currents. Note however that we just need to show that the lower bound
\[
\cH^{m-2 +\alpha}_\infty (\bB_1 \cap \sing_Q (T_{0, r_k}))\geq \eta>0\, 
\]
induces a similar lower bound on the Hausdorff measure of the set of points $p$ where $N^b_\infty (p) = Q \a{\etaa \circ {N^b_\infty} (p)}$. Indeed, from such a lower bound and Theorem \ref{t:structure} we would conclude that $N^b_\infty$ is a classical single valued harmonic function counted $Q$ times. On the other hand $\etaa \circ N^b_\infty \equiv 0$ would then imply $N^b_\infty \equiv Q\a{0}$. This would then finally contradict what we concluded from the previous section, namely that
\[
\int |N^b_\infty|^2 =1 \, .
\]
Hence, consider the closed sets ${\rm D}_Q (T_{0, r_k} \cap \bar \bB_1)$ and let $\Gamma$ be their Hausdorff limit (after extraction of a subsequence). We wish to show that most points in $\Gamma$ are points of ``multiplicity $Q$'' for $N^b_\infty$, or briefly $Q$-points for $N^b_\infty$.

\subsection{Persistency of singularities} Recall that $N^b_\infty$ is continuous (because it is $\D$-minimizing). Thus a statement analogous to Theorem \ref{t:persistence} would guarantee that {\em all} points of $\Gamma$ are $Q$-points for $N^b_\infty$. However such a theorem is not available at present. We rather show that the subset of points of $\Gamma$ which fail to be $Q$-points of $N^b_\infty$ must be a set of $\cH^{m-2+\alpha}$-measure, which however is enough for our purposes. We finish therefore this survey by sketching the argument. 

\medskip

Indeed we will follow a slightly different strategy. Summarizing what achieved so far, $N^b_\infty$ is a nontrivial $\D$-minimizing map which has $\etaa\circ N^b_\infty \equiv 0$. Thus the set of points of multiplicity $Q$ for $N^b_\infty$ coincides with the closed set of $p$ where $N^b_\infty (p) = Q\a{0}$. Such set must have $\cH^{m-2+\alpha}$ measure zero. We can then identify a closed set $\Lambda$ for which, at the same time, we have
\begin{itemize}
\item For some positive $\vartheta$
\[
\sum_i |N^b_{\infty,i} - \etaa\circ N^b_{\infty}|^2 \geq 2 \vartheta \qquad \mbox{everywhere on $\Lambda$}\, 
\]
\item $\Lambda$ is the Hausdorff limit of closed subsets $\Lambda_k$ of ${\rm D}_Q (T_{0, r_k})$ with $\cH^{m-2+\alpha}$ measure larger than $\frac{\eta}{2}$.
\end{itemize}
By the H\"older continuity of $N^b_\infty$ there must be a fixed constant $\bar\sigma$ such that
\[
\mint_{B_\sigma (q)} \sum_i |N^b_{\infty, i} - \etaa\circ N^b_{\infty}|^2 \geq \vartheta >0\qquad \forall q\in \Lambda\, ,\forall \sigma < \bar \sigma .
\]
We now fix a positive $\sigma < \bar{\sigma}$, whose choice will be specified only at the very end.

\medskip

By $L^2$ convergence, for $k$ large we inherit the inequality
\begin{equation}\label{e:alto}
\mint_{\cB_\sigma (q)} \sum_i |\bar N_{k,i} - \etaa\circ \bar N_{k}|^2 \geq \frac{\vartheta}{2}\bh_k^2 >0\qquad \forall q\in \p_{\bar{\cM_k}} (\Lambda_k)\, ,
\end{equation}
where $\p_{\bar \cM_k}$ denotes the orthogonal projection onto $\bar\cM_k$. 
 Now observe that, for $k$ large enough, it is also true that 
for any $q\in \p_{\bar \cM_k} (\Lambda_k)$ there is a point $p\in {\rm D}_Q (T_{0, t_{l(k)}})$ which is in the proximity of $q$, at a scale much smaller than $\sigma$. 
A very favorable situation is when $q$ belongs to the contact set where $T_{0, r_k}$ coincides with $Q\a{\bar \cM_k}$ (and thus $p=q$): in this case 
\[
\lim_{\rho\downarrow 0} \mint_{\cB_\rho (q)} \sum_i |\bar N_{k,i} - \etaa\circ \bar N_{k}|^2 = 0
\]
On the other hand, even if this is not the case, it is possible to argue that for some radius $\rho<<\sigma$ the integral above goes below the threshold
$\frac{\vartheta}{4} \bh_k^2$, namely
\begin{equation}\label{e:basso}
\mint_{\cB_\rho (q)} \sum_i |\bar N_{k,i} - \etaa\circ \bar N_{k}|^2\leq  \frac{\vartheta}{4}\bh_k^2 \, .
\end{equation}

Now, the reader familiar with Morrey spaces will realize that \eqref{e:alto} and \eqref{e:basso} force the existence of some intermediate radius $t (q)\in ]\rho, \sigma[$ with
\begin{equation}\label{e:per_coprire}
\frac{c_0\, \vartheta}{\sigma^{\alpha}}  \bh_k^2\leq \frac1{t(q)^{m-2+\alpha}}\, ,
\int_{\cB_{t(q)} (q)} |D \bar N_k|^2,
\end{equation}
where $c_0$ is a universal constant. 

For each $p\in \Lambda_k$ let $t (p)$ be $t (\p_{\bar \cM_k} (q))$.
Now using an elementary covering argument we can cover $\Lambda_k$ with balls $\bB_{10 t (p_i)} (p_i)$ in such a way that the balls $\cB_{2t (p_i)} (p_i)$ are disjoint. Being $\bar \cM_k$ almost flat it is not difficult to see that even the balls $\cB_{t (q_i)} (q_i) := \cB_{t (p_i)} (\p_{\bar \cM_k} (p_i))$ must be disjoint. Since $\Lambda_k$ has  $\cH^{m-2+\alpha}$ measure larger than $\eta/2$, we achieve
\begin{align}
\frac{\eta}{2} \leq &\sum_i \omega_{m-2+\alpha} (10 t(q_i))^{m-2+\alpha} \leq \sum_i \frac{C \sigma^\alpha}{\bh_k^2} \int_{\cB_{t (q_i)} (q_i)} |D\bar N_k|^2\nonumber\\
\leq & \frac{C \sigma^\alpha}{\bh_k^2}\int_{\cB_1 (\bar{p}_k)} |D\bar N_k|^2 \leq C \sigma^\alpha\label{e:argomento_finale}\, .
\end{align}
In the last inequality we have used that $\bh_k^2$, the $L^2$ norm of $\bar N_k$,  controls the Dirichlet energy, recall \eqref{e:aggiunta1}, \eqref{e:aggiunta2} and \eqref{e:aggiunta3}. Although the constant $C$ in \eqref{e:argomento_finale} depends upon the sequence of rescalings, the current $T$ and several other parameters, it does not depend upon $\sigma$ and $\eta$. Thus for a suitable choice of $\sigma$ we reach a contradiction.

\bibliographystyle{plain}
\bibliography{bib-survey-JDG}

\end{document}